\documentclass[11pt]{amsart}
\usepackage[amsthm, noextra]{basemath}
\usepackage{yhmath}
\usepackage{fullpage}
\usepackage{graphicx} 
\usepackage{amsrefs}
\usepackage{amsmath, amssymb, amsthm}
\usepackage{algorithm, algpseudocode}
\usepackage{setspace}
\usepackage{nicematrix}
\usepackage{multicol}

\numberwithin{equation}{theorem}
\RequirePackage[shortlabels]{enumitem}

\DeclareMathOperator{\Pprim}{\mathrm{P.Prim}}

\newcommand{\orbit}[1]{\mathbb{O}({#1})}

\newcommand{\p}{\mathfrak{p}}
\newcommand{\Loc}{\mathrm{Loc}}
\newcommand{\DLoc}{\mathrm{DLoc}}
\newcommand{\DLocl}[1]{\mathrm{DLoc}^{\leq {#1}}}
\newcommand{\n}{\mathbf{n}}
\newcommand{\m}{\mathbf{m}}

\newcommand{\Dx}{\mathrm{Dx}}

\title{The orbit method for the Virasoro algebra}
\author{Tuan Anh Pham}

\date{\today}

\address{School of Mathematics, The University of Edinburgh, United Kingdom}
\email{tuan.pham@ed.ac.uk}

\begin{document}
\begin{abstract}
    Let $W = \CC[t, t^{-1}]\del_t$ be the Witt algebra of algebraic vector fields on $\CC^\times$ and let $\Vir$ be the Virasoro algebra, the unique nontrivial central extension of $W$. In~\cite{petukhov2022poisson}, it was shown that Poisson primitive ideals of $\Sa(W)$ and $\Sa(\Vir)$ can be constructed from elements of $W^*$ and $\Vir^*$ of a particular form, called \emph{local functions}. In this paper, we show how to use a local function on $W$ or $\Vir$ to construct a representation of the Lie algebra. We further show that the annihilators of these representations are new completely prime primitive ideals of $\Ua(W)$ and $\Ua(\Vir)$. We use this to define a Dixmier map from the Poisson primitive spectrum of $\Sa(\Vir)$, respectively $\Sa(W)$, to the primitive spectrum of $\Ua(\Vir)$, respectively $\Ua(W)$, successfully extending the orbit method from finite-dimensional solvable Lie algebras to our countable-dimensional setting. 
    
    Our method involves new ring homomorphisms from $\Ua(W)$ to the tensor product of a localized Weyl algebra and the enveloping algebra of a finite-dimensional solvable subquotient of $W$. We further show that the kernels of these homomorphisms are intersections of the primitive ideals constructed from natural subsets of $W^*$. As a corollary, we disprove the conjecture that any primitive ideal of $\Ua(W)$ is the kernel of some map from $\Ua(W)$ to the first Weyl algebra.
    \end{abstract}
\maketitle
\tableofcontents
\section{Introduction}

{
One of the guiding principles in Lie theory is Kirillov's orbit method~\cite{Kirillov2004LecturesOT}. In the nilpotent case, this can be summarized as a correspondence between the unitary dual of a nilpotent Lie group $G$ (the set of equivalence classes of unitary irreducible representations) and the set of (coadjoint) orbits of $G$ on $\mathrm{lie}(G)^*$. Dixmier~\cite{dixmier1996enveloping} then extended this method to a finite-dimensional solvable complex Lie algebra $\g$ and its adjoint group $G$ and offered a completely algebraic characterization of the orbit method by replacing irreducible unitary representations with their annihilators in the universal enveloping algebra $\Ua(\g)$, which he showed are completely prime primitive ideals. By work of Dixmier, Conze, Duflo, and Rentschler~\cite{borho1973}*{13.4, 15.1}, there is a natural \emph{Dixmier map} $\Dx^\g: \g^* \to \Prim \Ua(\g)$ for such $\g$. In summary, one can construct an irreducible representation $M'_\chi$ from a function $\chi \in \g^*$ such that $\Dx^\g(\chi) = \Ann_{\Ua(\g)} M'_\chi$. They showed that this Dixmier map furthermore is constant on coadjoint orbits to give a surjective, injective and continuous map 
\[ \overline{\Dx}^\g: \g^*/G \to \Prim \Ua(\g). \]
Mathieu ~\cite{mathieu1991bicon} later showed that the map $\overline{\Dx}^\g$ is open, thus a homeomorphism. However, the Dixmier map is not well-defined even for finite-dimensional semisimple complex Lie algebras. 

Recent developments in Poisson algebras have replaced the coadjoint orbit space $\g^*/G$ with the space of Poisson cores of the symmetric algebra $\Sa(\g)$, introduced by Brown and Gordon in ~\cite{brown2002Poisson}. We recall the essential definition that for any Lie algebra $\g$, the symmetric algebra $\Sa(\g)$ is a Poisson algebra with the Kostant-Kirillov Poisson bracket. If $\g$ is finite-dimensional, then $\Sa(\g) \cong \mathcal{O}(\g^*)$, the coordinate ring of $\g^*$. An ideal of a Poisson algebra $A$ is called \emph{Poisson} if it is a Lie ideal under the Poisson bracket. The \emph{Poisson core} $\mathrm{Core}(I)$ of an ideal $I$ of $A$ is the maximal Poisson ideal contained in $I$. A \emph{Poisson primitive ideal} is then the Poisson core of a maximal ideal of $A$. We let $\Pprim A$ denote the set of Poisson primitive ideals of $A$. In~\cite{goodearl2008semiclassical}*{Corollary 8.4}, Goodearl shows that when $\g$ is a finite dimensional complex Lie algebra, there is a homeomorphism 
\begin{align*}
    \g^*/G \xrightarrow{\sim}  \Pprim \Sa(\g),
\end{align*}
establishing a correspondence between $G$-orbits on $\g^*$ and Poisson primitive ideals. Thus, in the case when $\g$ is a finite-dimensional solvable complex Lie algebra, the Dixmier map can be rephrased as in ~\cite{goodearl2008semiclassical}*{Theorem 8.11}
as a homeomorphism  
\begin{align*}
    \widetilde{\Dx}^{\g}: \Pprim \Sa(\g) \xrightarrow{\sim} \Prim \Ua(\g). 
\end{align*}
Phrased this way, the Dixmier map is now completely independent of the Lie group $G$, and thus could be extended to a more general setting: e.g., quantization of the coordinate ring of a Poisson variety ~\cite{goodearl2008semiclassical}. 

In this paper, we construct the Dixmier map for two important infinite-dimensional Lie algebras: the \emph{Witt algebra} or \emph{centerless Virasoro algebra} $W = \CC[t, t^{-1}]\del_t$ of algebraic vector fields on $\CC^\times$, and its famous central extension the \emph{Virasoro algebra} $\Vir = \CC[t, t^{-1}]\del_t \oplus \CC z$, with Lie bracket given by 
\begin{align*}
    [f \del_t, g \del_t] = (fg'-f'g) \del_t + \mathrm{Res}_0 (f'g'' - f''g')z, \quad z \text{ is central.}
\end{align*}
We also consider the subalgebra $\W{-1} = \CC[t]\del_t$ of $W$ of algebraic vector fields on $\CC$.
These countably infinite-dimensional Lie algebras play important roles in mathematics and physics, for example in quantum physics, conformal field theory and vertex operator algebras. In recent decades, there has been significant progress in the ring theory of the enveloping algebras of these Lie algebras. In this introduction, we will state the results for $W$, but in the body, we will provide the details for known and our new results also for $\W{-1}$ and $\Vir$. Notably in ~\cite{sierra2014universal}, $\Ua(W)$ is shown to be not left or right Noetherian. Furthermore, $\Sa(W)$ is proven to satisfy ACC on two-sided radical Poisson ideals in ~\cite{sanchez2023poisson}. ~\cite{petukhov2022poisson} studies the primitive spectrum $\Pprim \Sa(W)$ and, in particular, answers the question of which functions on $W$ have nontrivial Poisson core.

One of the main results of this paper is that there is a Dixmier map from $\Pprim \Sa(W)$ to $\Prim \Ua(W)$, giving a large new class of primitive ideals of $\Ua(W)$. We construct a map $\Dx^{W}: W^* \to \Prim \Ua(W)$ and prove that $\Dx^{W}$ factors through the natural map $W^* \xrightarrow{\sim} \MSpec \Sa(W) \xrightarrow{\mathrm{P.Core}} \Pprim \Sa(W)$.  
\begin{theorem}\label{theo:Dixmier1}[Theorem~\ref{theo: primitive ideals} and Theorem~\ref{theo: Dixmier map}]
    There is a procedure to construct a representation $M'_\chi$ of $W$ from a function $\chi \in W^*$ so that $\Ann_{\Ua(W)} M'_\chi$ is always a primitive ideal, even though $M'_\chi$ is not always simple. This defines a map 
    \begin{align*}
        \Dx^{W}: W^* \to \Prim \Ua(W),  \quad \chi \mapsto \Ann_{\Ua(W)} M'_\chi.
    \end{align*}
    Moreover, $\Dx^{W}$ factors through the natural map $W^* \to \Pprim \Sa(W)$ to induce a map
    \begin{align*}
        \widetilde{\Dx}^{W}: \Pprim \Sa(W) \to \Prim \Ua(W).
    \end{align*}
\end{theorem}
We will call $\Dx^{W}$ the \emph{Dixmier map} and $\widetilde{\Dx}^{W}$ the \emph{strong Dixmier map} of the Witt algebra. 

Let us explain what is proved about $\Pprim \Sa(W)$ in ~\cite{petukhov2022poisson} and then discuss the construction of $\Dx^W$ and $\widetilde{\Dx}^W$. By the extended Nullstellensatz ~\cite{mcconnell2001noncommutative}*{Corollary 9.1.8, Lemma 9.1.2}, maximal ideals of $\Sa(W)$ are in natural bijection with $W^*$. If $\chi \in W^*$, we denote the corresponding maximal ideal of $\Sa(W)$ by $\mf{m}_{\chi}$ and let $P(\chi) = \mathrm{Core}(\mf{m}_\chi)$ be the Poisson core of $\mf{m}_\chi$. One of the main results of ~\cite{petukhov2022poisson} is that $P(\chi) \neq 0$ if and only if $\chi$ is a \emph{local function} on $W$, which is defined as follows:
\begin{definition}[~\cite{petukhov2022poisson}*{Definition 3.0.2.}]
    A \emph{local function on $W$} is a sum of finitely many \emph{one-point local functions} of the form
    \begin{equation}\label{eq:deflocal function1}
        \chi_{x; \alpha_0, \dots, \alpha_n}: W \to \CC; \quad f\del_t \mapsto \alpha_0 f(x) + \alpha_1 f'(x) + \dots + \alpha_n f^{(n)}(x),  
    \end{equation}
    where $x, \alpha_0, \dots, \alpha_n \in \CC$ and $x, \alpha_n \neq 0$. We define the \emph{order} of $\chi_{x; \alpha_0, \dots, \alpha_n}$ to be $n$.  
\end{definition}

When $\chi \in W^*$ is not local, $P(\chi) = 0$. By work of Wallach and Mathieu (see
e.g.~\cite{iyudu2020enveloping}*{Theorem 1.4.}), $0$ is a primitive ideal of $\Ua(W)$, so we can define $\Dx^{W}(\chi) = 0$. Thus, we restrict our attention to local functions on $W$. Our construction of the representation $M'_\chi$ follows the orbit method in the finite-dimensional solvable case: for a local function $\chi$, we first prove the existence of a polarization $\p$, which is a finite codimensional subalgebra $\p$ of $W$ such that $\chi$ induces a one-dimensional representation $\CC_{\chi, \p}$ of $\p$. We define a one-dimensional twisted representation $\CC'_{\chi, \p}$ of $\p$ (see Definition~\ref{def:onedimtwist}) and then induce to a $W$-representation $M'_{\chi, \p} = \Ua(W) \otimes_{\Ua(\p)} \CC'_{\chi, \p}$. We call $M'_{\chi, \p}$ the \emph{local representation} of $W$ induced by $\chi$ and $\p$. Although the polarization $\p$ for $\chi$ is not unique (Example~\ref{ex:polarization}), we prove that the annihilator $\Ann_{\Ua(W)} M'_{\chi, \p}$ is independent of $\p$ (Theorem~\ref{theo:notdeppol}). Thus we can define $Q_\chi = \Ann_{\Ua(W)} M'_{\chi, \p}$. 

Local representations need not be irreducible. However, we show their annihilators are always primitive, hence we can define $\Dx^{W} (\chi) = Q_\chi$. Because $Q_\chi$ is independent of the choice of polarization,we can work with a \emph{canonical polarization} $\p_\chi$ (defined in Proposition~\ref{prop:canpol}) and define $M'_{\chi} := M'_{\chi, \p_\chi}$ to be a \emph{canonical local representation} of $W$. 

\begin{theorem}[Theorem~\ref{theo: primitive ideals} and Corollary~\ref{cor:comprime}]\label{theo: annihilators}
    Let $\chi$ be a local function on $W$ and let $M'_\chi$ be the corresponding canonical local representation of $W$. Then $Q_\chi \coloneqq \Ann_{\Ua(W)} M'_\chi$ is a completely prime primitive ideal of $\Ua(W)$. 
    
    Most of the local representations $M'_\chi$ are irreducible (exact details are provided in Theorem~\ref{theo:simplicitymulti} and  Corollary~\ref{cor:simplicityconverse}).
\end{theorem}
Since the $Q_\chi$ are completely prime ideals, by ~\cite{iyudu2020enveloping}*{Proposition 6.4} $\Ua(W)$ satisifies the ascending chain condition (ACC) for annihilators of local representations. It is an open question whether $\Ua(W)$ satisfies the ACC on arbitrary two-sided ideals. 

The representations $M'_\chi$ have been constructed previously in ~\cite{ondrus2018vir}, where the polarization $\p_\chi$ is called a polynomial subalgebra and the $M'_\chi$ are exactly the modules induced from polynomial subalgebras. Our description via local functions gives a more natural context for these canonical local representations, as they are induced from functions on $W$. The subalgebras of ~\cite{ondrus2018vir} are now realized as canonical polarizations of local functions. Furthermore, our construction follows the guiding principle of Kirillov's orbit method. 

Once the Dixmier map $\Dx^{W}: W^* \to \Prim \Ua(W)$ is constructed, it is natural to try to understand its fibers; recall that in the finite-dimensional solvable case, these are percisely the coadjoint orbits, or, in terms of the Lie algebra, the Dixmier map gives a homeomorphism between $\Pprim \Sa(\g)$ and $\Prim \Ua(\g)$. Thus, one immediately asks if $\Dx^{W}$ factors through $\Pprim \Sa(W)$. We recall notation from ~\cite{petukhov2022poisson} that for $\g$ a Lie algebra, the \emph{pseudo-orbit} of $\chi \in \g^*$ is defined as 
\begin{align*}
    \orbit{\chi} = \{ \mu \in \g^*| P(\mu) = P(\chi)\},
\end{align*}
so pseudo-orbits are precisely fibres of $\g^* \to \Pprim \Sa(\g)$. In these terms, we ask: if $\orbit{\chi} = \orbit{\eta}$, then do we have $Q_\chi = Q_\eta$? To answer this, we must understand more about the annihilators $Q_\chi$.

For simplicity, suppose that $\chi=\chi_{x; \alpha_0, \dots, \alpha_n}$ is a one-point local function of the form \eqref{eq:deflocal function1} of order $n$. Let $\g_n = \kk\{v_0, \dots, v_{n-1} \}$ with the Lie bracket $[v_i, v_j]= (j-i)v_{i+j}$ where $v_{\geq n}= 0$. Note that $\g_n$ is a subquotient of $W$, but not a quotient. Let $\tilde{A}_1 = \CC[t, t^{-1}, \del]$ be the localized first Weyl algebra with $\del t - t \del = 1$. We construct an algebra homomorphism from $\Ua(W)$ to $\tilde{A}_1 \otimes \Ua(\g_n)$, and show that $Q_\chi$ is the preimage of a primitive ideal of $\tilde{A}_1 \otimes \Ua(\g_n)$.
\begin{theorem}[Theorem~\ref{theo:psin}, Corollary~\ref{cor:relatingprimtive} and Corollary~\ref{cor:kereven}]\label{theo:algebrahom}
    Let $n \in \NN$. There are algebra homomorphisms
    \begin{align*}
        \Psi^{\W{-1}}_{\infty} &: \Ua(\W{-1}) \to \CC[t, \partial] \otimes \Ua(\W{0}), &\quad
        f \partial &\mapsto f(t) \partial_t + \sum_{i=0}^{\infty} \frac{f^{(i+1)}(t)}{(i+1)!}v_i, \\
        \Psi^W_{n} &: \Ua(W) \to \tilde{A}_1 \otimes \Ua(\g_n), &\quad
        f \partial &\mapsto f(t) \partial_t + \sum_{i=0}^{n-1} \frac{f^{i+1}(t)}{(i+1)!}v_i.
    \end{align*}
    Furthermore, if $\chi=\chi_{x; \alpha_0, \dots, \alpha_n}$ is a one-point local function on $W$, then $Q_\chi$ is the preimage of a primitive ideal of $\tilde{A}_1 \otimes \Ua(\g_n)$ via $\Psi^W_{n}$. Moreover, if $\chi$ is a one-point local function of order $2k>0$, then $Q_\chi = \ker \Psi^W_{2k}$.
\end{theorem} 
The map $\Psi^W_{n}$ can be seen as the universal map for one-point local functions of order at most $n$  (Corollary~\ref{cor:psiuniversal}) and we refer to $\Psi^{\W{-1}}_{\infty}$ as a ``master homomorphism''. We also define a master homomorphism for $W$ in the body of the paper. These homomorphisms are interesting on their own and to our best knowledge, they have not appeared in the literature previously. 

For an arbitrary local function $\chi$, we tensor together several of the maps $\Psi^W_{n}$ to construct a homomorphism $\Psi^W_{\n} : \Ua(W) \to \tilde{A}_\ell \otimes \Ua(\g_\n)$, where $\tilde{A}_\ell$ is a localization of the $\ell$-th Weyl algebra and $\g_\n$ is a suitable finite-dimensional solvable Lie algebra, such that $Q_\chi \supseteq \ker \Psi^W_\n$. Moreover, $Q_\chi$ is the preimage of a primitive ideal of $\tilde{A}_\ell \otimes \Ua(\g_\n)$ via $\Psi^W_\n$ and further corresponds to a primitive ideal $Q_{\bar{\chi}}$ of $\Ua(\g_\n)$, where $\bar{\chi} \in \g^*_\n$. In fact, we can compute $\bar{\chi}$ from $\chi$ as explained in Notation~\ref{not:barchi}.
In other words, we have
\begin{equation}\label{eq: annihilator intro}
    Q_\chi = (\Psi_\n^W)^{-1}(\tilde{A}_\ell \otimes Q_{\bar{\chi}}).
\end{equation}
Let $G_\n$ be the adjoint alegbraic group of $\g_\n$. By the finite-dimensional orbit method, $Q_{\bar{\chi}}$ only depends on $G_\n \cdot \bar{\chi}$, and we establish a relationship between coadjoint orbits of $G_\n$ and pseudo-orbits in $W^*$.
\begin{theorem}[Corollary~\ref{cor:orbitrelation} and Proposition \ref{prop:anndeppseudo}]
    Let $\chi, \eta$ be local functions on $W$ of order $\neq 0$. Then for suitable $\n$ 
    \begin{align*}
        \orbit{\chi} = \orbit{\eta} \Leftrightarrow  G_\n \cdot \bar{\chi} = G_\n \cdot \bar{\eta} \Leftrightarrow Q_{\bar{\chi}} = Q_{\bar{\eta}}. 
    \end{align*}
    Moreover, if $\orbit{\chi} = \orbit{\eta}$ then $Q_\chi = Q_\eta$. There is thus a well-defined strong Dixmier map $\widetilde{\Dx}^{W}$ given by
    \begin{align*}
        \widetilde{\Dx}^{W}: \Pprim \Sa(W) \to \Prim \Ua(W), \quad  0 \mapsto 0, \quad   
        P(\chi) \neq 0 \mapsto Q_\chi.
    \end{align*}
\end{theorem}

It has been asked in~\cite{conley2024} if every primitive ideal of $\Ua(\W{-1})$ is the kernel of some map to ${A}_1$. We disprove this by considering suitable $Q_\chi$. In particular, we show that if $\chi, \eta$ are one-point local functions such that $P(\chi) \subseteq P(\eta)$, then $Q_\chi \subseteq Q_\eta$. As a consequence, we show that if $\chi$ is a one-point local function on $\W{-1}$ of order $\geq 2$, then $\Ua(\W{-1})/Q_\chi$ has Gelfand-Kirillov dimension at least 3 and does not embed in the first Weyl algebra (Proposition~\ref{prop:GKdim3}). However, a more general version of the conjecture remains open, as every primitive ideal $Q_\chi$ constructed in this paper is the kernel of a map to some Weyl algebra (Corollary~\ref{cor:annkerweyl}). If $\g$ is a finite-dimensional solvable Lie algebra, then $\Dx^\g$ and $\widetilde{\Dx}^\g$ are surjective onto $\Prim \Ua(\g)$. However, their proof in \cite{dixmier1996enveloping}*{Theorem 6.1.7} relies on an induction argument, thus fails to work for $W$. Hence surjectivity of the (strong) Dixmier map is a difficult open question and subject of further research.  

The organisation of the paper is as follows. In Section~\ref{sec:background}, we give background on the Lie algebras of interest, Poisson algebras, 
recall key results from~\cite{petukhov2022poisson} about $\Pprim \Sa(W)$, and recall the Dixmier map for finite-dimensional solvable Lie algebras. In Section~\ref{sec: representation}, we construct local representations of $W$ and prove that the annihilator of a local representation does not depend on the choice of polarization. In Section~\ref{sec:primitive}, we prove part of Theorem~\ref{theo: annihilators}, that annihilators of local representations are indeed primitive. In Section~\ref{sec: psi}, we show that $\Psi^W_n$ and $\Psi^{\W{-1}}_\infty$ in Theorem~\ref{theo:algebrahom} are algebra homomorphisms by showing their associated graded maps 
$\Phi^W_n$ and $\Phi^{\W{-1}}_\infty$ are Poisson homomorphisms. In Section~\ref{sec:annihilators}, we prove~\eqref{eq: annihilator intro} and other results relating primitive ideals $Q_\chi$ of $\Ua(W)$ to primitive ideals $Q_{\bar{\chi}}$ in $\Ua(\g_\n)$. In Section~\ref{sec: Dixmier map}, we finally relate the orbit method for the finite-dimensional Lie algebra $\g_n$ to that of $W$ to prove the remaining parts of Theorem~\ref{theo: annihilators} for $W$ as well as $\W{-1}$ and $\Vir$. We end the paper in Section~\ref{sec:inclusion}, where we study inclusions of annihilators of local representations and motivate some further research questions. 

\textbf{Acknowledgements:} This work was done as part of the author's PhD research at the University of Edinburgh.
We thank Lucas Buzaglo and James Timmins for interesting discussion and questions.
}

\section{Background}\label{sec:background}

In this section, we recall the general notions of Poisson algebra, Poisson ideals and Poisson primitive ideals. We apply these concepts to the symmetric algebras of the infinite-dimensional Lie algebras of interest and collect the central results of~\cite{petukhov2022poisson} on the primitive spectra of these symmetric algebras. Lastly, we restate the classical orbit method for finite-dimensional solvable Lie algebras ~\cite{dixmier1996enveloping}. 

Fix an uncountable algebraically closed base field $\kk$ of characteristic $0$ throughout the paper; all algebras mentioned will be unital $\kk$-algebras and all vector spaces (in particular, Lie algebras) are assumed to be defined over $\kk$. Unless otherwise specified, tensor products are over $\kk$, so we abbreviate $\otimes_\kk$ to $\otimes$.  

The \emph{Witt algebra} $W = \kk[t, t^{-1}]\del$\label{pl:W} is the Lie algebra of algebraic vector fields on $\kk^\times := \kk \backslash \{0\}$. We denote its canonical basis by $e_i = t^{i+1}\del$\label{pl:ei} and put a $\ZZ$-grading on it by setting $\deg e_i = i$. Subalgebras of $W$ of the form $\W{n} = \kk\{e_i | i \geq n\}$ for $n \geq -1$\label{pl:Wn} are of particular interest, especially $\W{-1} = \kk[t]\del$\label{pl:W-1}.

The \emph{Virasoro algebra} $\Vir$ is the unique nontrivial central extension of $W$ and is isomorphic as a vector space to $\kk[t, t^{-1}]\del \oplus \kk z$\label{pl:Vir}. It is endowed with a Lie algebra structure by the formula 
\begin{align*}
    [f \del_t, g \del_t] = (fg'-f'g) \del_t + \mathrm{Res}_0 (f'g'' - f''g')z.
\end{align*}
(Here $\mathrm{Res}_0(f)$ stands for the algebraic residue of $f$ at $0$, i.e., the coefficient of $t^{-1}$ in the Laurent expansion of $f$ at $0$.) We denote the canonical Lie algebra projection $\Vir \to W$\label{pl:pVirW} given by factoring out $z$ by $\Gamma$.

We also consider subalgebras of $W, \W{-1}$ and $\Vir$ of the following form.
\begin{notation}\label{pl:Wf}
    For $f \in \kk[t] \nonzero$ in the case of $\W{-1}$ and $f \in \kk[t, t^{-1}]$ in the case of $W$ and $\Vir$, we define 
    \begin{equation}
        \W{-1}(f) = f\W{-1}, \quad W(f) = fW, \quad  \Vir(f) = W(f) + \kk z.
    \end{equation}
\end{notation}
These subalgebras are also known as polynomial subalgebras in~\cite{ondrus2018vir}.
\subsection{The Poisson spectrum}

Let $A$ be a commutative $\kk$-algebra and $\{\cdot\ , \cdot\}: A \times A \to A $ a skew-symmetric $\kk$-bilinear map. 
We say that $(A, \{\cdot\ , \cdot\})$ is a {\it Poisson algebra}\label{ind:Poissonalg} if $\{\cdot\ , \cdot\}$ satisfies the Leibniz rule on each input and the Jacobi identity. An ideal $I$ of $A$ is called \emph{Poisson} if $\{I, A\}\subseteq I$.\label{ind:Poissonideal} Let $I$ be an ideal of $A$. The sum of all Poisson subideals of $I$ is the maximal Poisson ideal inside $I$; we denote this Poisson ideal by ${\rm Core}(I)$ and refer to it as the {\em Poisson core} 
\label{ind:poissoncore2}
of $I$. A Poisson ideal $I$ is called \emph{Poisson primitive} if $I = \Core(\mathfrak{m})$ for a maximal ideal $\mathfrak{m}$ of $A$; Poisson primitive ideals are prime. Let $\Pprim A$ \label{pl:Pprim} be the set of Poisson primitive ideals of $A$ endowed with the Zariski topology, where the closed subsets are defined by Poisson ideals of $A$.

Consider a Lie algebra $\g$ with $\dim \g < |\kk|$. It is well known that $\Sa(\g)$ possesses a canonical Poisson algebra structure, induced by defining $\{u, v\} = [u, v]$ for any $u, v \in \g$. Moreover, our assumption on the cardinality of the field $\kk$
implies that $\Sa(\g)$ satisfies the extended Nullstellensatz~\cite{mcconnell2001noncommutative}*{Corollary 9.1.8, Lemma 9.1.2}. Thus, maximal ideals of $\Sa(\g)$ can be canonically identified with $\g^*$:
\begin{align*}
    \chi \in \g^* \leftrightarrow \mathfrak{m}_\chi := \ker \ev_\chi,
\end{align*}
where $\ev_\chi$ the induced homomorphism $\Sa(\g) \to \kk$ defined by $\ev_\chi(f) = f(\chi)$. Thus any Poisson primitive ideal of $\Sa(\g)$ is equal to $\Core(\mathfrak{m}_\chi)$ for some $\chi \in \g^*$. Set $P(\chi) := \Core(\mathfrak{m}_\chi)$\label{ind:Pchi}.

\begin{definition}\label{def:psorbit}
    Let $\g$ be any Lie algebra. The \emph{pseudo-orbit} of $\chi \in \g^*$ is 
    \begin{align*}
        \orbit{\chi} = \{\nu \in \g^* \text{ such that } P(\chi) = P(\nu)\}. 
    \end{align*}
    The \emph{dimension} of $\orbit{\chi}$ is defined to be $\GK \Sa(\g)/P(\chi)$. (Here if $R$ is a $\kk$-algebra then $\GK R$ denotes the \emph{Gelfand-Kirillov dimension} of $R$; see ~\cite{krause2000growth}.)

    If $\g = \mathrm{Lie}(G)$ is the Lie algebra of a connected algebraic group and $\chi \in \g^*$, then $\orbit{\chi}$ equals the coadjoint orbit $G \cdot \chi$ and $\dim \orbit{\chi}$ equals the dimension of $G \cdot \chi$ as a variety. 
\end{definition}

The Poisson primitive spectrum $\Pprim \Sa(\g)$ is studied in~\cite{petukhov2022poisson} for $\g = \W{-1}$, $W$ or $\Vir$. 

\begin{theorem}[\cite{petukhov2022poisson}*{Theorem Theorem 3.1.1., Theorem 3.2.1., Theorem 3.3.1.}]
    Let $\g$ be $\W{-1}, W$ or $\Vir$. Then $P(\chi)$ is non-trivial, i.e. $P(\chi) \neq 0$ for $\g = \W{-1}$ or $W$ and $P(\chi) \neq (z-\chi(z))$ for $\g = \Vir$, if and only if $\chi$ is a \emph{local function} on $\g$, which is defined below. 
\end{theorem}

\begin{definition}\label{def: local function}[cf. ~\cite{petukhov2022poisson}*{Definition 3.0.2}]
    Let $x, \alpha_0, \alpha_1, ...,\alpha_n \in \kk$ with $\alpha_0, \dots, \alpha_n$ not simultanously zero, we define $\chi_{x; \alpha_0, ...,\alpha_n} \in \W{-1}^\ast$ by  
\begin{equation}\label{eq:one-pointlf}
    \chi_{x; \alpha_0, ...,\alpha_n}: f \partial \mapsto \alpha_0 f(x) + \alpha_1 f'(x) + ... +\alpha_n f^{(n)}(x),
\end{equation}
    where $f^{(i)}$ denotes taking the $i$th-derivative of $f$. The same formula defines elements of $W^*$, although we need to require that $x \neq 0$.
\begin{enumerate}
    \item A \emph{local function on $\W{-1}$} is a sum of finitely many functions of the form
   ~\eqref{eq:one-pointlf} with (possibly) distinct  $x$;
    \item A \emph{local function on $W$} is a sum of finitely many functions of the form~\eqref{eq:one-pointlf} with (possibly) distinct $x \neq 0$.
    \item A \emph{local function on $Vir$} is the pullback of a local function on $W$ via the canonical map $\Gamma: Vir \to W$. 
\end{enumerate}
A local function of the form~\eqref{eq:one-pointlf} is called a \emph{one-point local function}. Let $\chi = \chi_{x; \alpha_0, ...,\alpha_n}$ be a one-point local function with $\alpha_n \neq 0$. We call $x$ the \emph{base point} of $\chi$ and $n$ the \emph{order} of $\chi$. 
\end{definition}

\begin{notation}\label{not:localfunc}
    Let $\g$ be either $\W{-1}, W$ or $\Vir$ and let $\chi$ be a local function on $\g$. We will write 
    \begin{align*}
        \chi= \chi_1 + \dots  + \chi_{\ell},
    \end{align*}
    where $\chi_i$ is a one-point local function of order $n_i$ based at $x_i \in \kk$ ($x_i \neq 0$ for $W$ and $\Vir$) and the $x_i$ are pairwise distinct. We then define the \emph{order} of $\chi$ to be $\n = (n_1, \dots , n_\ell)$, the \emph{support} of $\chi$ to be $\supp(\chi) = \{x_1, \dots , x_\ell \}$ and denote $m_i = \left\lfloor \frac{n_i}{2}\right\rfloor$. By~\cite{petukhov2022poisson}*{Proposition 4.3.6. and Remark 4.2.12.}, $\dim \orbit{\chi} = \sum_{i=1}^{\ell} (2m_i+2)$.

    Let $\Loc^{\leq \n}$ be the space of all $\ell$-point local functions on $\g$ of order $\widehat{\n} \leq \n$, i.e., $\widehat{n}_i \leq n_i$ for all $i$.
\end{notation}

In~\cite{petukhov2022poisson}, the authors also determined when $P(\chi) = P(\eta)$ for $\chi, \eta \in \g^*$, i.e., the pseudo-orbit of $\orbit{\chi}$ is fully characterized in the following two results.

\begin{proposition}[\cite{petukhov2022poisson}*{Theorem 4.3.1}]
    Let $\g$ be either $\W{-1}, W$ or $\Vir$. Let $\chi, \eta$ be local functions on $\g$ of order $\n, \widehat{\n}$ following Notation \ref{not:localfunc}. Then 
    \begin{equation}
        \orbit{\chi} = \orbit{\eta} \Leftrightarrow \text{ up to reordering } \n = \widehat{\n} \text{ and } \orbit{\chi_i} = \orbit{\eta_i} \text{ for all } i. 
    \end{equation}
\end{proposition}

\begin{proposition}[\cite{petukhov2022poisson}*{Theorem 4.2.1.}]
    Let $\g$ be as before. Let $\chi = \chi_{x; \alpha_0, \dots, \alpha_n}$ be a one local function on $\g$ of order $n$, and let $m = \left\lfloor \frac{n}{2}\right\rfloor$.
    \begin{enumerate}
        \item[(1)] If $n=2m$ then $\orbit{\chi} = \orbit{e^*_{n-1}}$.
        \item[(2)] If $n=1$, then $\orbit{\chi} = \orbit{\alpha_1 e^*_{n-1}}$. For $\alpha, \beta \in \kk$, $\orbit{\alpha e^*_{0}} = \orbit{\beta e^*_{0}}$ if and only if $\alpha = \beta$. 
        \item[(3)] If $n = 2m+1>1$, then there is $\alpha \in \kk$ such that $\orbit{\chi} = \orbit{e^*_{n-1} + \alpha e^*_{m}}$. For $\alpha, \beta \in \kk$, then $\orbit{e^*_{2m} + \alpha e^*_{m}} = \orbit{e^*_{2m} + \beta e^*_{m}}$ if and only if $\alpha = \pm \beta$.
    \end{enumerate} 
\end{proposition}

Lastly, they also established the inclusions of orbit closures of one-point local functions, which can almost be computed just from the dimension of the pseudo-orbits.

\begin{proposition}[\cite{petukhov2022poisson}*{Corollary 4.2.9.}]
    Let $\g$ be either $\W{-1}, W$ or $\Vir$. Let $\chi, \eta$ be one-point local functions on $\g$. Then 
    \begin{equation}
        P(\chi) \subsetneqq P(\eta) \Leftrightarrow \dim \orbit{\eta} < \dim \orbit{\chi} \text{ and } \chi \text{ is not of order } 1.
    \end{equation}
\end{proposition}

\subsection{The orbit method for finite-dimensional solvable Lie algebras}
For a unital ring $R$, let $\Prim R$~\label{pl:Prim} denote the set of (left) primitive ideals of $R$ endowed with the Jacobson topology. 

We will now recall the classical orbit method, and in particular the definition of the (strong) Dixmier map, for a finite-dimensional solvable Lie algebra over $\kk$. Throughout this subsection, unless otherwise specified, let $\g$ be any Lie algebra. 
As compared to the nilpotent cases, we will need the following one-dimensional twist associated to a subalgebra $\mathfrak{h}$ of $\g$.

\begin{definition}\label{def:onedimtwist}
    Let $\mf{h}$ be a subalgebra of finite codimension of a Lie algebra $\g$. For all $x \in \mf{h}$, $\ad_\g (x)$ is a linear map $\g \to \g$ such that ${\ad_\g(x)}({\mf{h}}) \subseteq \mf{h}$. Thus, there is a quotient linear map $\ad_{\g/\mf{h}} (x): \g/\mf{h} \to \g/\mf{h}$. Since $\mf{h}$ has  finite codimension in $\g$, $\tr (\ad_{\g/\mf{h}}) (x)$ is well-defined. We denote 
    \begin{equation}\label{eqdef:onedimtwist}
        \rho_{\mf{h},\g}: \mf{h} \to \kk; \quad x \mapsto \frac{1}{2}\tr (\ad_{\g/\mf{h}}) (x),
    \end{equation}
    and call it a \emph{one-dimensional twist} of $\mathfrak{h}$. By ~\cite{dixmier1996enveloping}*{5.5.1}, $\rho_{\mf{h},\g}$ is a character of $\mathfrak{h}$ and thus defines a representation of $\mathfrak{h}$, which we also refer to as $\rho_{\mf{h},\g}$. Moreover, if $\mf{t} \subseteq \mf{h} \subseteq \g$ are Lie algebras where $\g/\mf{t}$ is finite-dimensional, then 
\begin{equation}\label{eq:twistmiddle}
    \rho_{\mf{t},\g}(x) = \rho_{\mf{t},\mf{h}}(x) + \rho_{\mf{h},\g}(x) \quad \text{ for all } x \in \mf{t}.
\end{equation}
\end{definition}

Let us recall the definition of a polarization, which is crucial in constructing representations for the orbit method.
\begin{definition}
    For $\chi \in \g^*$, we define an induced bilinear form $B_\chi: \g \times \g \to \kk$ by $(u, v) \mapsto \chi ([u, v])$.\label{pl:Bchi} A \emph{polarization} $\p$\label{pl:polarization} of $\chi$ is a finite codimensional subspace of $\g$ that is maximally totally isotropic with respect to $B_\chi$ and is also a subalgebra of $\g$. 
\end{definition}

With those key ingredients, we are now prepared to define the representations from the orbit method and consider their annihilators, which will be the main objects of interest in the paper. 
\begin{definition}\label{def:orbitrep}
    Let $\chi \in \g^*$ and let $\p$ be a polarization of $\chi$. Then $\chi$ is a character of $\p$, thus induces a representation $\kk_\chi = \kk \cdot 1_\chi$ of $\p$ where $x \cdot 1_\chi = \chi(x) 1_\chi$ for $x \in \p$. We can form a \emph{twisted one-dimensional representation} $\kk'_{\chi, \p, \g}$ of $\p$:
    \begin{equation}\label{eq:twistedone}
        x \cdot 1_\chi = \bigl(\chi(x) + \rho_{\p,\g}(x)\bigr)1_\chi \text{ for } x \in \p.
    \end{equation} 
    Note that $\kk'_{\chi, \p, \g} \cong \kk_{\chi} \otimes \rho_{\p, \g}$ as a representation of $\p$. We form an induced representation of $\g$
    \begin{equation}
        M'_{\chi, \p} = \Ua(\g) \otimes_{\Ua(\p)} \kk'_{\chi, \p, \g}, 
    \end{equation}
    and call it the \emph{representation of $\g$ induced by $\chi$ and $\p$}. Let $Q_{\chi, \p}$ be the annihilator
    \begin{align*}
        Q_{\chi, \p} = \Ann_{\Ua(\g)} M'_{\chi, \p}.
    \end{align*}
    It is also useful to define the following induced representation of $\g$, which we will refer to as \emph{untwisted induced representation} of $\g$
    \begin{equation}\label{eq:untwistrep}
        M_{\chi, \p} = \Ua(\g) \otimes_{\Ua(\p)} \kk_{\chi}. 
    \end{equation}
\end{definition}

For an arbitrary Lie algebra $\g$, a polarization might not exist for all $\chi \in \g^*$. Even when polarizations exist, they may not be unique. However, if $\g$ is a finite-dimensional solvable Lie algebra, then not only do polarizations always exist, but $Q_{\chi, \p}$ also does not depend on the choice of polarization. Further, in this case $Q_{\chi, \p}$ is primitive. This gives a natural map $\g^* \to \Prim \Ua(\g)$, called the \emph{Dixmier map}. We summarize the situation in the following theorem. 

\begin{theorem}[~\cite{dixmier1996enveloping}*{Theorem 6.5.12} and ~\cite{mathieu1991bicon}]\label{theorem: Kirillov method}
    Let $\g$ be a finite-dimensional solvable Lie algebra. Then for every $\chi \in \g^*$, there exists a polarization $\p$ for $\chi$. Moreover, $Q_{\chi, \p}$ defined in~\ref{def:orbitrep} is independent of the choice of $\p$, and thus, we set $Q_{\chi} = Q_{\chi, \p}$. Furthermore, $Q_\chi$ is primitive and we have a map $\Dx^\g: \g^* \to \Prim \Ua(\g)$ induced by $\chi \mapsto Q_\chi$. 
    
    Let $G$ be the adjoint algebraic group of $\g$ and let $\g^*/G$ denote the space of coadjoint orbits equipped with the quotient Zariski topology. If $\eta \in G \cdot \chi$, then $Q_\eta = Q_\chi$,
    i.e., the Dixmier map $\Dx^\g$ factors through the the quotient $\g^* \to \g^*/G$ to give a well-defined map   
    \[
\begin{tikzcd}
    \g^* \arrow[r, "\Dx^\g"] \arrow[d] & \Prim \Ua(\g) \\
    \g^*/G \arrow[ur, "\overline{\Dx}^{\g}", swap] &  
  \end{tikzcd} \text{ given by }
  \begin{tikzcd}
    \chi \arrow[r, mapsto] \arrow[d, mapsto] & Q_\chi \\
    G \cdot \chi \arrow[ur, mapsto] &  
  \end{tikzcd}
\]
    Moreover, $\overline{\Dx}^{\g}$ is a homeomorphism between $\g^*/G$ and $\Prim \Ua(\g)$.
\end{theorem}

For a finite-dimensional Lie algebra $\g$, Kirillov, Kostant, and Souriau independently proved that the map $\g^*/G \to \Pprim \Sa(\g)$ induced by $G\cdot \chi \mapsto P(\chi)$ is a homeomorphism (\cite{Kirillov2004LecturesOT}*{\S.2.2, Theorem 2}), which yields the following \emph{strong Dixmier map} $\widetilde{\Dx}^{\g}: \Pprim \Sa(\g) \rightarrow \Prim \Ua(\g)$ if $\g$ is solvable. 
\begin{theorem}[\cite{goodearl2008semiclassical}*{Theorem 8.5}]\label{theo:strongDx} 
    Let $\g$ be a finite-dimensional solvable Lie algebra. Then there is a well-defined homeomorphism
    \begin{align}
    \widetilde{\Dx}^{\g}: \Pprim \Sa(\g) \rightarrow \Prim \Ua(\g), \quad P(\chi) \mapsto Q_{{\chi}}.
    \end{align}
\end{theorem}
\section{Local representations from the orbit method}\label{sec: representation}
Let $\g$ be either $\W{-1}, W$ or $\Vir$. If $\g = \W{-1}$ or $\g = W$, let $z = 0$. For $\chi \in \g^*$, by ~\cite{petukhov2022poisson}*{Theorem 3.3.1} $P(\chi) \neq (z-\chi(z))$ if and only if $\chi$ is a local function on $\g$. By~\cite{iyudu2020enveloping}*{Theorem 1.4}, $(z - \chi(z))$ is a primitive ideal in $\Ua(\g)$, thus, our strong Dixmier map $\widetilde{\Dx}^\g: \Pprim \Sa(\g) \to \Prim \Ua(\g)$, if it exists, should map $(z -\chi(z)) \mapsto (z-\chi(z))$. Thus for the rest of the paper, we restrict our attention to local function on $\g$ unless otherwise specified. 

In this section, we first construct a canonical class of polarization for every local function $\chi$ on $\g$ (Proposition \ref{prop:canpol}) and construct a twisted representation $M'_{\chi, \p}$ of $\g$ from a polarization $\p$ of $\chi$. We then establish the important result that although there are various polarizations $\p$ of a local function $\chi$, the annihilator $Q_{\chi, \p}$ of $M'_{\chi, \p}$ is independent of the choice of polarization (Theorem~\ref{theo:notdeppol}). We can thus restrict our attention to a special class of polarization, which is moreover $\kk[t, t^{-1}]$-submodule and allows us to work with a nice basis. We further show that it suffices to consider local functions on $W$ (Propositions~\ref{prop:repWVir} and~\ref{prop:repWW-1}).

One of the main difficulties in generalizing the orbit method to an arbitrary Lie algebra is that a polarization of a function on the Lie algebra may not exist. However, we will construct a special \emph{canonical polarization} for any local function $\chi$ on $\g$, which is moreover a $\kk[t]$ or $\kk[t, t^{-1}]$-submodule of $\g$. Before that, we show an auxillary lemma.
\begin{lemma}\label{lem:poldim}
    Let $\g$ be either $\W{-1}, W$ or $\Vir$ and let $\chi$ be a local function on $\g$ following Notation~\ref{not:localfunc}. Let $\p$ be any polarization of $\chi$. Then  $\codim \p = \sum_{i}^{\ell} (m_i +1)$.
\end{lemma}
\begin{proof}
    By ~\cite{petukhov2022poisson}*{Lemma 4.3.4.}, we have $\g \cdot \chi = \bigoplus_i \g \cdot \chi_i$. Thus
    \begin{align*}
        \codim \p = \frac{1}{2} \codim \g^\chi = \frac{1}{2} \sum_{i}^{\ell} \dim \g \cdot \chi_i = \frac{1}{2}\sum_i^{\ell} \rank B_{\chi_i} = \sum_{i=1}^{\ell} (m_i+1),
    \end{align*}
    where the last equality follows from ~\cite{petukhov2022poisson}*{Theorem 4.2.1}.
\end{proof}

\begin{proposition}\label{prop:canpol}
    Let $\g$ be either $\W{-1}, W$ or $\Vir$ and let $\chi$ be a local function on $\g$ following Notation~\ref{not:localfunc}. Then $ \p_\chi = \g(\prod_{i=1}^{\ell} (t-x_i)^ {m_i +1})$ is a polarization of $\chi$. 
\end{proposition}

\begin{proof}
    We first consider the case where $\g = W$ or $\W{-1}$. For any $f \in \kk[t]$ if $\g = \W{-1}$ and $f \in \kk[t, t^{-1}]$ if $\g = W$,
    \begin{equation}\label{eq:canpolhelp}
        [fg\partial, fh\partial] = f^2[g\partial, h\partial],
    \end{equation}
    so $\p_\chi$ is a subalgebra of $\g$. Applying~\eqref{eq:canpolhelp} with $f = \prod_{i=1}^{\ell} (t-x_i)^{m_i+1}$ we have $B_\chi(fg\partial, fh\partial) = 0$, thus $\p_\chi$ is a totally isotropic subspace of $(W, B_\chi)$. By dimension count, $\p_\chi$ is a maximal totally isotropic subspace, and thus a polarization.

    Finally, if $\g = \Vir$, then $\chi(z) = 0$. Let $\Gamma(\chi)$ denote the corresponding local function on $W$. Thus $\p_\chi = \p_{\Gamma(\chi)} + \kk z = \Vir(\prod_{i=1}^{\ell} (t-x_i)^{m_i+1})$ is a polarization of $\chi$ as it is a subalgebra, a totally isotropic subspace, and maximal by dimension count. 
\end{proof}

We note that in general, a polarization for a local function $\chi$ on $\g$ is not unique as explained in the following example. For simplicity, we work with local functions on $W$; the examples can be adapted for $\W{-1}$ and $\Vir$.

\begin{example}\label{ex:polarization}
    Let $\chi = \chi_{x; \alpha_0, \alpha_1, \alpha_2, \alpha_3}$ be a local function on $W$, where $x \neq 0$ and $(\alpha_2, \alpha_3) \neq (0, 0)$. Let $\p$ be any polarization of $\chi$. From Lemma~\ref{lem:poldim}, $\codim \p = 2$. From ~\cite{petukhov2022poisson}*{Proposition 6.2.2.}, we have the classification of codimension-2 subalgebras of $W$: $\p$ is either $W((t-x)(t-y)), W^{2; 1}_{x; \alpha}$ or $W^{2; 2}_{x; \alpha}$. It is easy to check that $W((t-x)^2)$ is a polarization for $\chi$ and $W((t-x)(t-y))$ is not a polarization for $\chi$ if $y \neq x$. 
    
    If $\p = W^{2; 1}_{x; \alpha}$, i.e., $\p$ contains $W((t-x)^3)$ with additional generator $(t-x)\del + \alpha(t-x)^2 \del$, then note that $[\p, \p] = W((t-x)^3)$. Thus, if $\chi = \chi_{x; \alpha_0, \alpha_1, \alpha_2}$ ($\alpha_2 \neq 0$), then $\chi([\p, \p]) = 0$ and $W^{2; 1}_{x; \alpha}$ is another polarization of $\chi$ for any $\alpha$. 

    On the other hand, if $\p = W^{2; 2}_{x; \alpha}$, i.e., $\p$ contains $W((t-x)^4)$ and with additional generators $(t-x)\del+ \alpha (t-x)^3\del$ and $(t-x)^2 \del$, then $[\p, \p] = (t-x)^2 \del + W((t-x)^4)$. Thus, if $\chi = \chi_{x; \alpha_0, \alpha_1, 0, \alpha_3}$ ($\alpha_3 \neq 0$), then $\chi([\p, \p]) = 0$ and $W^{2; 2}_{x; \alpha}$ is another polarization of $\chi$ for any $\alpha$.
\end{example}

Let $\p$ be a polarization of a local function $\chi$ on $\g$. As $\p$ has finite codimension in $\g$, following Definition~\ref{def:orbitrep}, we can define a (twisted) one-dimensional representation $\kk'_{\chi, \p, \g}$ of $\p$ and induce it to a representation $M'_{\chi, \p}$ of $\g$. Its annihilator $Q_{\chi, \p}$ in $\Ua(\g)$ is the ideal of main interest in the paper. We now show that the annihilator $Q_{\chi, \p}$ of $M'_{\chi, \p}$ does not depend on the choice of polarization $\p$ (Theorem~\ref{theo:notdeppol}) and start with two auxiliary lemmas.

For a polynomial $f \in \kk[t]$, set
\begin{equation}\label{pl:radf}
    \supp(f) = \{x \in \kk| f(x) = 0 \}, \text{ and } \rad(f) := \prod_{x \in \supp(f)}(t - x).
\end{equation}

\begin{lemma}\label{lem:polsandwich}
    Let $\g$ be $\W{-1}, W$ or $\Vir$ and let $\chi$ be a local function on $\g$. Let $\p$ be any polarization of $\chi$. Then there exists $f \in \kk[t]$ such that $\chi(f\del) = 0$, $\supp(f) \supseteq \supp(\chi)$ and 
    \begin{align*}
    \g(f) \subseteq \p \subseteq \g(\rad(f)).
    \end{align*}
\end{lemma}
\begin{proof}
    If $\g = \Vir$, then by ~\cite{petukhov2022poisson}*{Corollary 3.3.3.}, $z \in \p$, so we may reduce to considering the corresponding subalgebra of $W$. Moreover, by ~\cite{petukhov2022poisson}*{Corollary 6.4.3.}, since the lattices of finite codimension subalgebras of $\W{-1}, W$ and $\Vir$ are isomorphic, we only need to prove the statement for $W$. Since $\p$ has finite codimension, by ~\cite{petukhov2022poisson}*{Proposition 3.2.7} there exists $h \in \kk[t]$ such that 
    \begin{align*}
        W(h) \subseteq \p \subseteq W(\rad(h)).
    \end{align*}
    Thus, $W(h^2) = [W(h), W(h)] \subseteq [\p, \p]$ and since $\chi([\p, \p]) = 0$, we have $\chi(W(h^2)) = 0$. Hence, $\supp(h) = \supp(h^2) \supseteq \supp(\chi)$ and if we let $f = h^2$, noting that $\rad(f) = \rad(h)$ we obtain the lemma. 
\end{proof}

\begin{corollary}\label{cor:polsandwich}
    Let $\g$ be $\W{-1}, W$ or $\Vir$ and let $\chi$ be a local function $\g$. Let $\p_\chi$ be the polarization of $\chi$ defined in Proposition~\ref{prop:canpol} and $\p$ be any other polarization of $\chi$. Then there exists $f \in \kk[t]$ such that $\chi(f\del) = 0$, $\supp(f) \supseteq \supp(\chi)$ and 
    \begin{align*}
    \g(f) \subseteq \p_\chi, \p \subseteq \g(\rad(f)).
    \end{align*}
\end{corollary}
\begin{proof}
    By the same reasoning as in Lemma~\ref{lem:polsandwich}, we only need to prove the statement for $W$. By Lemma~\ref{lem:polsandwich}, there exists $h \in \kk[t]$ such that $\supp(h) \supseteq \supp(\chi)$ and
    \begin{align*}
        W(h) \subseteq \p \subseteq W(\rad(h)).
    \end{align*}
    Let $f = \mathrm{lcm} (h,(t-x_1)^{m_1+1}\dots (t-x_\ell)^{m_\ell+1})$. Note that $\supp(h) \supseteq \supp(\chi) =  (t-x_1)\dots (t-x_\ell)$, thus $(t-x_1)\dots (t-x_\ell) | \rad(h)$ and hence $\rad(f) = \rad(h)$. Thus 
    \begin{align*}
        W(f) \subseteq \p_\chi,\p \subseteq W(\rad(f)).
    \end{align*}
\end{proof}

We now show the theorem that the annihilator $Q_{\chi, \p}$ induced from $\chi$ and a polarization $\p$ does not depend on the choice of polarization. 
\begin{theorem}\label{theo:notdeppol}
    Let $\g$ be $\W{-1}, W$ or $\Vir$ and let $\chi$ be a local function on $\g$. Let $\p_\chi$ be the polarization of $\chi$ defined in Proposition~\ref{prop:canpol} and $\p$ be any other polarization of $\chi$. Define $ Q_{\chi, \p_\chi}$ and $Q_{\chi, \p}$ as in Definition~\ref{def:orbitrep}. Then
    \begin{equation}
        Q_{\chi, \p_\chi} = Q_{\chi, \p}.
    \end{equation}
\end{theorem}
\begin{proof}
    By Corollary~\ref{cor:polsandwich}, there exists $f$ such that $\g(f) \subseteq \p_\chi, \p \subseteq \g(\rad(f))$ and $\chi(\g(f)) = 0$. Let $\mf{q} = \g(\rad(f))$. We define the intermediate twisted representations of $\mf{q}$ to be
    \begin{align*}
        K'_{\chi, \p_\chi} = \Ua(\mf{q}) \otimes_{\Ua({\p_\chi})} \kk'_{\chi, \p_\chi, \mf{q}}\text{ and } K'_{\chi, \p} = \Ua(\mf{q}) \otimes_{\Ua(\p)} \kk'_{\chi, \p, \mf{q}}.
    \end{align*}
    Following Definition~\ref{def:onedimtwist},  we define the one-dimensional twist $\rho_{\g, \mf{q}}$ with  $\rho_{\g, \mf{q}}(f\del) = \frac{1}{2} \tr (\ad(\g/\mf{q})) (f\del)$ for $f\del \in \mf{q}$. We obtain the twisted representations of $\mf{q}$: 
    \begin{align*}
        K''_{\chi, \p_\chi} = K'_{\chi, \p_\chi} \otimes \rho_{\g, \mf{q}} \text{ and }
        K''_{\chi, \p} = K'_{\chi, \p} \otimes \rho_{\g, \mf{q}}.
    \end{align*}
    \color{black}
    From ~\cite{dixmier1996enveloping}*{Proposition 5.2.3.}, we have isomorphisms of representations of $\g$
    \begin{align}
        M'_{\chi, \p_\chi} \cong \Ua(\g) \otimes_{\Ua(\mf{q})} K''_{\chi, \p_\chi} \text{ and } M'_{\chi, \p} \cong \Ua(\g) \otimes_{\Ua(\mf{q})} K''_{\chi, \p}.
    \end{align}
    Thus, by ~\cite{dixmier1996enveloping}*{Proposition 5.2.6.} it suffices to show that $\Ann_{\Ua(\mf{q})} K'_{\chi, \p_\chi} = \Ann_{\Ua(\mf{q})} K'_{\chi, \p}$.
    
    Since $\chi(\g(f)) = 0$ and $\g(f)$ is a Lie ideal of $\mf{q}$, $\Ua(\g(f))_+ \cdot K'_{\chi, \p_\chi} = 0 = \Ua(\g(f))_+ \cdot K'_{\chi, \p}$. Hence,
    \begin{align*}
        \Ann_{\Ua(\mf{q})} K'_{\chi, \p_\chi} &= \ker \Ua(\mf{q}) \to \End(K'_{\chi, \p_\chi}) = \ker \left(\Ua(\mf{q}) \to \Ua(\mf{q}/ \g(f))  \to \End(K'_{\chi, \p_\chi}) \right); \\
        \Ann_{\Ua(\mf{q})} K'_{\chi, \p} &= \ker  \Ua(\mf{q}) \to \End(K'_{\chi, \p}) = \ker \left(\Ua(\mf{q}) \to \Ua(\mf{q}/ \g(f))  \to \End(K'_{\chi, \p}) \right).
    \end{align*}
    Note that $\mf{q}/\g(f)$ is a finite-dimensional completely solvable Lie algebra (since $\kk$ is algebraically closed), so by ~\cite{dixmier1996enveloping}*{Theorem 6.1.4}
    \begin{align*}
        \Ann_{\Ua(\mf{q}/\g(f))} K'_{\chi, \p_\chi} = \Ann_{\Ua(\mf{q}/\g(f))} K'_{\chi, \p}.
    \end{align*}
    Thus, $\Ann_{\Ua(\mf{q})} K'_{\chi, \p_\chi} = \Ann_{\Ua(\mf{q})} K'_{\chi, \p}$ and the theorem is proven. 
\end{proof}

As we are only interested in the annihilator $Q_{\chi, \p}$, which by Theorem~\ref{theo:notdeppol} does not depend on $\p$, we can choose to work with the local representations that corresponds to the polarization defined in Proposition~\ref{prop:canpol}. We make this choice because the polarization in Proposition~\ref{prop:canpol} has a nice form with concrete generators, and is moreover a $\kk[t]$ or $\kk[t, t^{-1}]$-submodule of $\g$. 
\begin{definition}\label{def:canrep}
    Let $\g$ be either $\W{-1}, W$ or $\Vir$ and let $\chi$ be a local function on $\g$ following Notation~\ref{not:localfunc}. Then we call the polarization $\p_\chi$ in Proposition~\ref{prop:canpol} \emph{the canonical polarization} of $\chi$. The local representation $M'_{\chi, \p_\chi}$ corresponding to the canonical polarization of $\chi$ is called the \emph{canonical local representation} of $\g$ corresponding to $\chi$. From now on, we abbreviate the notation of canonical local representation to $M'_{\chi}$ and its annihilator to $Q_\chi$.
\end{definition}

The following notation will be useful later in the computation of canonical local representations.
\begin{notation}\label{not:leadterm}
    Let $V$ be a vector space, $S$ be a basis of $V$ and $\mc{C}$ be a total order on $S$. Then for every $v \in V$, we write $v= \sum_{i = 0}^{n} \lambda_i v_i$, where $v_i \in S$, $\lambda_i \in \kk^*$ and 
    $v_0 > \dots > v_n$ with respect to the order $\mc{C}$. We define the \emph{leading term} of $v$ with respect to order $\mc{C}$ to be $v_0$, and write 
    \begin{align*}
        \LT_{\mc{C}}(v) = v_0.
    \end{align*}

    This notation will be used repeatedly for canonical local representations. Let $\g$ be $\W{-1}, W$ or $\Vir$ and let $\chi$ be a local function on $\g$ and let $\p_\chi$ be the canonical polarization of $\chi$. Let $(f_1, \dots, f_N)$ be a basis for $\g/\p_\chi$ and $\mc{C}$ be the order $f_1 > \dots > f_N$ , then by the PBW theorem,
    \begin{align*}
        S = \{f_1^{k_1} \dots  f_N^{k_N} 1_\chi | k_i \in \NN\}
    \end{align*}
    is a basis for $M'_{\chi}$. Moreover we can order $S$ lexicographically, thus, for $z \in M'_{\chi}$, we can write $\LT_{\mc{C}}(z)$ for the leading term of $z$ with respect to $\mc{C}$. 
\end{notation}

We now prove two propositions relating canonical local representations of $\W{-1}$ and $\Vir$ to the corresponding canonical local representations of $W$. 
\begin{proposition}\label{prop:repWW-1}
    Let $\chi$ be a local function on $W$ and $\chi|_{\W{-1}}$ be the corresponding function on $\W{-1}$. Let $M'_\chi$ and $M'_{\chi|_{\W{-1}}}$ be the corresponding canonical local representations of $W$ and $\W{-1}$ respectively. Let $\iota$ denote the canonical embedding $\Ua(\W{-1}) \hookrightarrow \Ua(W)$. Then  
    \begin{equation}
        M'_{\chi|_{\W{-1}}} \cong \res{\iota} M'_\chi  \text{ as representations of }\W{-1}.
    \end{equation}
\end{proposition}
\begin{proof}
    Let the canonical polarization for $\chi$ be $W(f)$, where $f = (t-x_1)^{m_1+1} \dots  (t-x_\ell)^{m_\ell+1}$. Then $\W{-1}(f)$ is the canonical polarization for $\chi|_{\W{-1}}$. Let $M = m_1 + \dots + m_\ell + \ell$. We can choose $w_i \in \W{-1}$ such that 
    \begin{align*}
        W = W(f) \oplus \bigoplus_{i=1}^{M-1} \kk w_i \text{ and } 
        \W{-1} = \W{-1}(f) \oplus \bigoplus_{i=1}^{M-1} \kk w_i,  
    \end{align*}    
    Thus, by the PBW theorem, bases for $M'_\chi$ and $M'_{\chi|_{\W{-1}}}$ are
    \begin{align*}
        \mathcal{S}^W = \{ w_{-1}^{k_{-1}}\dots w_{M-1}^{k_{M-1}} 1_\chi | k_i \in \NN \} \text{ and }  \mathcal{S}^{\W{-1}} = \{ w_{-1}^{k_{-1}}\dots w_{M-1}^{k_{M-1}} 1_{\chi|_{\W{-1}}} | k_i \in \NN \} \text{ respectively.}
    \end{align*}
    Under the canonical embedding $\iota: \Ua(\W{-1}) \hookrightarrow \Ua(W)$, we can view $\Ua(\W{-1})$ as a subring of $\Ua(W)$, thus, we have a $\Ua(\W{-1})$-module homomorphism 
    \begin{equation}
        \begin{aligned}
            \theta: M'_{\chi|_{\W{-1}}} \to \res{\iota} M'_\chi, \text{ induced by mapping } 1_{\chi|_{\W{-1}}} \mapsto 1_\chi. 
        \end{aligned}
    \end{equation}
    The homomorphism $\theta$ by construction maps $\mathcal{S}^W$ to $\mathcal{S}^{\W{-1}}$ bijectively, and so $\theta$ is an isomorphism. 
\end{proof}

\begin{proposition}\label{prop:repWVir}
    Let $\chi$ be a local function on $\Vir$ and recall the canonical projection $\Gamma: \Vir \twoheadrightarrow W$ sending $z \to 0$. Abusing notation, let $\Gamma(\chi)$ be the corresponding local function on $W$. Let $M'_\chi$ and $M'_{\Gamma(\chi)}$ be the corresponding canonical local representations of $W$ and $\Vir$ respectively. Then 
    \begin{align*}
        M'_\chi \cong \res{\Gamma} M'_{\Gamma(\chi)} \text{ as representations of } \Vir.
    \end{align*}
\end{proposition}
\begin{proof}
    We have $z \in \p_\chi$, $\chi(z) = 0$ and $\p_{\Gamma(\chi)} = \Gamma(\p_\chi)$ is the canonical polarization of ${\Gamma(\chi)} \in W^*$. Thus $\Gamma$ induces a $\Ua(\Vir)$-module homomorphism $\theta: M'_\chi \to \res{\Gamma} M'_{\Gamma(\chi)}$. By definition, $\theta$ sends a basis of $M'_\chi$ to one of $\res{\Gamma} M'_{\Gamma(\chi)}$, thus it is an isomorphism. 
\end{proof}

Lastly, in this section, we want to prove the important result that for $\chi$ a local function on $\g$ following Notation~\ref{not:localfunc}, the canonical local representation $M'_\chi$ decomposes into the tensor product $M'_{\chi_1} \otimes \dots \otimes M'_{\chi_\ell}$. This follows from~\cite{ondrus2018vir}*{Proposition 6.1}.

\begin{proposition}\label{prop:tensor1point}
    Let $\g$ be $\W{-1}$, $W$ or $\Vir$ and $\chi = \chi_1 + \dots  + \chi_\ell $ be a local function on $\g$ following Notation~\ref{not:localfunc}. Then
\begin{equation}
    M'_\chi \cong M'_{\chi_1} \otimes \dots  \otimes M'_{\chi_\ell}.
\end{equation}
\end{proposition}
\begin{proof}
    Let $\tilde{\chi} = \chi_1 + \dots + \chi_{\ell-1}$. It suffices to show that $M'_\chi \cong M'_{\tilde{\chi}} \otimes M'_{\chi_\ell}$. The canonical polarization $\p_\chi = \p_{\tilde{\chi}} \cap \p_{\chi_\ell}$ has infinite dimension, thus by ~\cite{ondrus2018vir}*{Proposition 6.1} $M'_\chi \cong M'_{\tilde{\chi}} \otimes M'_{\chi_\ell}$. 
\end{proof}

\section{Annihilators of local representations are primitive}\label{sec:primitive}
Let $\g$ be either $\W{-1}, W$ or $\Vir$.  . In this section, we show that the Dixmier map $\Dx^\g: \g^* \to \Prim \Ua(\g)$ is well-defined by proving that $Q_\chi$ is primitive when $\chi$ is a local function on $\g$. If $\g = W$ or $\g = \W{-1}$, let $z = 0$. Then as a result, we show that the following map is well-defined. 
\begin{equation}\label{eq:Dixmier}
    \Dx^\g: \g^* \to \Prim \Ua(\g); \quad 
    \chi \text{ local } \mapsto Q_\chi,  \quad 
    \chi \text{ not local } \mapsto (z-\chi(z)).
\end{equation}

\subsection{Simplicity of canonical local representations}
Let $\g$ be either $\W{-1}, W$ or $\Vir$. In this subsection, we will show that in almost all cases, the canonical local representations of $\g$ are irreducible. (The remaining cases will be remedied by Proposition~\ref{prop:primreplace}.) More precisely, we aim to show that for $\chi = \chi_1+ \dots + \chi_\ell$ a local function on $\g$ following Notation~\ref{not:localfunc}, then $M'_\chi$ is an irreducible representation of $\g$ if $\chi_i \neq \chi_{x_i; \alpha, \frac{1}{2}}$ for all $i$. Theorem~\ref{theo:simplicityone} and~\ref{theo:simplicitymulti} are also proven in \cite{ondrus2018vir}*{Theorem 5.6.} using different computations and notation. Our proof was obtained independently and makes it clear how simiplicity of the modules depends on the form of local functions $\chi$. We also establish an explicit element $Y$ that brings an arbitrary element $z \in M'_\chi$ to the generator $1_\chi$. 

From Propositions~\ref{prop:repWW-1} and~\ref{prop:repWVir}, we note that it is enough to consider the case $\g = \W{-1}$. Thus from now on in this section, we are only concerned with $\Ua(\W{-1})$-modules. 

We will first prove the statement for one-point local functions. Let $\chi = \chi_{x; \alpha_0, \dots, \alpha_n}$ be a nonzero one-point local function on $\W{-1}$ and $\chi \neq \chi_{x; \alpha, \frac{1}{2}}$. Let $m = \left\lfloor \frac{n}{2}\right\rfloor$. We begin with setting up some notations and define some special elements, which will be crucial in proving simplicity of $M'_\chi$. We then explain the proof strategy in Remark~\ref{rem:proofstrategy}.

\begin{definition}\label{def:u,a,xi}
    Let $x \in \kk$ and $q \in \kk[t]\backslash \{0\}$ such that $q(x) \neq 0$. For $j \geq -1$, we define 
    \begin{equation}\label{eq:tildeudef}
        u_{j, x} = (t-x)^{j+1}\del, \quad \tilde{u}_{q, j, x} = q u_{j, x} \in \W{-1}.
    \end{equation}
    For $0 \leq i \leq m -1$ since $\tilde{u}_{q, n-1-i, x}\in \p_\chi = W((t-x)^{m+1})$, we can define $a_{q, i, x} \in \kk$ to be given by 
    \begin{equation}\label{eq: a k d definition}
        \tilde{u}_{q, n-1-i, x} \cdot 1_\chi = a_{q, i, x} 1_\chi.     
    \end{equation}
    We define special elements, 
    \begin{equation}\label{eq:xidef}
        \xi_{q, i, x} = \frac{\tilde{u}_{q, n-1-i, x}-a_{q, i, x}}{(2i-n+1) a_{q, 0, x}} \in \Ua(\W{-1}).
    \end{equation}
    Note that when the base point is clear (e.g. when we consider a one-point local function), we often drop the dependence on $x$ and abbreviate $u_{j, x}, \tilde{u}_{q, j, x}, a_{q, i, x}, \xi_{q, i, x}$ to $u_{j}, \tilde{u}_{q, j}, a_{q, i}, \xi_{q, i}$ respectively.
\end{definition}

We now discuss the proof strategy of the one-point case in Theorem~\ref{theo:simplicityone}. 
\begin{remark}[Proof strategy]\label{rem:proofstrategy}
    Let $z = u_{-1}^{k_{-1}} \dots  u_{m-1}^{k_{m-1}} 1_\chi \in M'_\chi$. We wish to show that the generator $1_\chi$ of $M'_\chi$ lies in $\Ua(\W{-1})\cdot z$. To be able to adapt the proof for multi-point local representations, we will prove the stronger statement that for all $q \in \kk[t] \backslash \{0\}$ with $q(x) \neq 0$, if $Y_q =  \xi_{q, m-1}^{k_{m-1}}\dots \xi_{q, 0}^{k_{0}} \tilde{u}_{q, n-1+k_{-1}}$ then $Y_qz = c1_\chi$ for some $c \in \kk^\times$. Firstly, we consider the case where $u_{-1}$ and $u_0$ do not occur in $z$. This is proven in Lemma~\ref{lem:xiui0} and Lemma~\ref{lem:xiui}. Then, we add $u_0$ into the picture in Lemma~\ref{lem:xi0u0} and finally $u_{-1}$ in Lemma~\ref{lem:cancelu-1}. Lemma~\ref{lem:xi0general} will explain why we only need to look at ``monomial'' $z \in M'_\chi$. 
\end{remark}

\begin{lemma}\label{lem:rho0}
    Let $\chi$ be a nonzero one-point local function on $\W{-1}$ of order $n > 1$ based at point $x \neq 0$, and let $\p_\chi$ be its canonical polarization. Then for all $f\del \in \p_\chi$
    \begin{align*}
        \rho_{\p_\chi, \W{-1}} (f\del) = 0.
    \end{align*}
\end{lemma}
\begin{proof}
    Recall that $\p = \W{-1}((t-x)^{m+1})$ where $m = \left\lfloor \frac{n}{2}\right\rfloor$, thus $\W{-1} = \p_\chi \oplus \kk \{u_{-1}, \dots, u_{m-1} \}$. For all $k \geq m+1$ and $h \in \kk[t, t^{-1}]$, we have $[(t-x)^k h\del, \W{-1}(t-x)] \subseteq \p_\chi$. Furthermore, since $k \geq m+1 \geq 2$ we have 
    \begin{align*}
        [(t-x)^{k}h\del, \del] = -(t-x)^{k}h'\del - k(t-x)^{k-1}h\del \in \W{-1}(t-x).
    \end{align*}
    Consequently $\rho_{\p_\chi, \W{-1}} ((t-x)^k h\del) = \frac{1}{2}\tr(\ad_{\W{-1}/\p_\chi}((t-x)^k h\del)) = 0$.
\end{proof}

The following lemma proves an important result that $a_{q, 0} \neq 0$ for all $q \in \kk[t]$ and most local function $\chi$. 
\begin{lemma}\label{lem:anonzero}
    Let $\chi \neq \chi_{x; \alpha_0, \frac{1}{2}}$ be a one-point local function on $\W{-1}$. Then for all $q \in \kk[t] \backslash \{0\}$ with $q(x) \neq 0$, we have $a_{q, 0, x} \neq 0$.
\end{lemma}
\begin{proof}
    If $n > 1$, then by Lemma~\ref{lem:rho0}, we have $\rho_{\p_\chi, \W{-1}}(q\del) = 0$. Thus $\tilde{u}_{q, n-1, x} \cdot 1_\chi = \chi(\tilde{u}_{q, n-1, x}) = n! \alpha_n q(x) \neq 0$. On the other hand, if $n$ is $0$ or $1$, the canonical polarization is $\p_\chi = \W{-1}(t-x)$ and $[(t-x)q\del, \del] = -(t-x)q'\del - q\del$, thus $\rho_{\p_\chi, \g}((t-x)q\del) = -\frac{1}{2}q(x)$. Hence 
    \begin{equation}\label{eq:espu}
        \epsilon(\tilde{u}_{q, 0, x}) 1_\chi = (\alpha_1-\frac{1}{2})q(x)1_\chi,
    \end{equation}
    which is nonzero if $\alpha_1 \neq \frac{1}{2}$.
\end{proof}

We will set up the main assumption in this subsection until the end of Theorem~\ref{theo:simplicityone}.
\begin{assumption}
    From now on to the end of Theorem~\ref{theo:simplicityone}, we will only work with a one-point local function $\chi$, and will assume $\chi \neq \chi_{x; \alpha_0, \frac{1}{2}}$. Motivated by~\eqref{eq:espu}, for a one-point local function of order 0, we will set $\chi_{x; \alpha_0} = \chi_{x; \alpha_0, 0}$ and let its order be $n = 1$ (so that $u_{n-1} \in \p_\chi$) and let $m = \left\lfloor \frac{n}{2}\right\rfloor$. As explained in Definition \ref{def:u,a,xi}, we will drop the dependence on $x$ in $\xi_{q, i, x}, \tilde{u}_{q, i, x}, a_{q, i, x}$ and $u_{i, x}$. 
\end{assumption}

\begin{notation}\label{not:(var)sigma}
    Let $\sigma_i$ denote an arbitrary word in $\Ua(\W{-1})$ of the form $u_{i}^{k_i} \dots  u_{m-1}^{k_{m-1}}$. Let $\varsigma_i$ denote the image in $\Ua(\W{-1})$ of any word (not necessarily in PBW order) in the $u_j$'s for $j \geq i$. Thus, $\sigma_i$ is a special case of $\varsigma_i$. 
\end{notation}

We have the following important computation for the Lie bracket of $\tilde{u}_j$ and $u_i$, which we will use repeatedly. Write $q' = (t-x)^o p$, where $p(x) \neq 0$ and $\deg_t p < \deg_t q$. Then 
\begin{align}
    [\tilde{u}_{q, j} ,u_i] &= (i-j) \tilde{u}_{q, i+j} - q' u_{i+j+1} = (i-j) \tilde{u}_{q, i+j} - \tilde{u}_{p, i+j+1+o},\label{eq:brackettilde}\\ 
    \text{ and so } [\xi_{q, j},u_i] &= \frac{1}{(2j- n +1)a_{q, 0}} \bigl( (i+j+1-n )\tilde{u}_{q, i+n-1-j} - \tilde{u}_{p, i+n-j+o} \bigr).\label{eq:bracketxi}
\end{align}

We now show two useful lemmas.

\begin{lemma}\label{lem:tildeua}
    We have
    \begin{equation*}
        \tilde{u}_{q, n-1-i} 1_\chi =
        \begin{cases}
            a_{q, i}1_\chi & \text{ if } i = 0, \dots ,m-1 \\
            0                     & \text{ if } i <0.
        \end{cases}
    \end{equation*}
\end{lemma}
\begin{proof}
    For $i = 0,\dots ,m-1$, this is by definition of $a_{q, i}$. If $n = 1$ and $i < 0$, since $\tilde{u}_{q, -i} \in [ \p_\chi, \p_\chi ]$, then $\tilde{u}_{q, -i} \cdot 1_\chi = 0$. Otherwise, if $n> 1$ and $i<0$, it follows from $\chi(\tilde{u}_{q, n-1-i}) = 0$ and Lemma~\ref{lem:rho0}.
\end{proof}

\begin{lemma}\label{lem:tildeuaword}
    Let $q \in \kk[t] \backslash \{0\}$ with $q(x) \neq 0$. For $\varsigma_0, \varsigma_1 \in \Ua(\W{-1})$ following Notation \ref{not:(var)sigma}, we have
    \begin{equation}\label{eq:tildeuaword1}
        \tilde{u}_{q, n-1+j} \varsigma_1  1_\chi =
        \begin{cases}
            a_{q, 0} \varsigma_1  1_\chi & \text{ if } j=0       \\
            0               & \text{ if } j \geq 1,
        \end{cases}
    \end{equation}
    and 
    \begin{equation}\label{eq:tildeuaword2}
        \tilde{u}_{q, n-1+j} \varsigma_0  1_\chi = 0 \quad \text{ if } j \geq 1.
    \end{equation}
\end{lemma}
\begin{proof}
    In the proof, let $\varsigma$ denote either $\varsigma_0$ or $\varsigma_1$. We prove the lemma by strong induction on $\deg_t q$. For base case $\deg_t q = 0$, by scaling we can assume $q = 1$, i.e., $\tilde{u}_{q, i} = u_{i}$ for all $i$. We then induct on the length $\ell$ of $\varsigma$. The base case $\ell = 0$ is Lemma~\ref{lem:tildeua}. Now assume the result for $q=1$ and $|\varsigma| = \ell$, let  consider the word $u_i \varsigma$ ($i \geq 1$ in~\eqref{eq:tildeuaword1} and $i \geq 0$ in~\eqref{eq:tildeuaword2}). In both cases note that $i+j \geq 1$, so by the induction hypothesis, 
    \begin{equation}\label{eq:bracket0}
        [{u}_{n - 1 + j} , u_{i}] \varsigma 1_\chi = (i+1 -n - j ){u}_{n+i+j-1} \varsigma  1_\chi = 0. 
    \end{equation}
    By~\eqref{eq:bracket0},  ${u}_{n-1+j} u_i \varsigma 1_\chi = u_i {u}_{n-1+j} \varsigma 1_\chi = c u_i \varsigma 1_\chi$ by induction,  where $c$ is 0 or $a_{q, 0}$. So~\eqref{eq:tildeuaword1} and~\eqref{eq:tildeuaword2} hold for words of length $\ell+1$ and we have shown the base case $\deg_t q = 0$.

    Induction step on $\deg_t q$: we again proceed by induction on the length $\ell$ of $\varsigma$, and the base case $\ell = 0$ is again Lemma~\ref{lem:tildeua}. Then assuming the induction hypothesis if $|\varsigma| = \ell$, consider the word $u_i \varsigma$ ($i \geq 1$ in~\eqref{eq:tildeuaword1} and $i \geq 0$ in~\eqref{eq:tildeuaword2}). We have
    \begin{equation}\label{eq:bracketui}
        [\tilde{u}_{q, n - 1 + j} , u_{i}] \varsigma 1_\chi = (i+1 -n - j )\tilde{u}_{q, n+i+j-1} \varsigma  1_\chi - \tilde{u}_{p, n+i+j+o} \varsigma  1_\chi. 
    \end{equation}
    We note that the second term is 0 by the induction hypothesis on $\deg_t q$ and since $i+j \geq 1$, the first term is 0 by the induction hypothesis on $\ell$. Thus $\tilde{u}_{q, n-1+j} u_{i} \varsigma  1_\chi = u_{i} \tilde{u}_{q, n-1+j} \varsigma  1_\chi = c u_{i}\varsigma 1_\chi$ by induction on $\ell$, where $c$ is 0 or $a_{q, 0}$. Thus we have shown the induction step on $\deg_t q$. 
\end{proof}
\begin{notation}\label{not:phi}
    For $k \in \NN$, and for any ring $R$ and $r, s \in R$, define 
    \begin{align*}
        \phi_k(r, s) = r^{k-1} s + r^{k-2} s r \dots  + s r^{k-1} \in R.
    \end{align*}
\end{notation}

We can rephrase Lemma \ref{lem:tildeuaword} as the following corollary.
\begin{corollary}
    For all $r \in \Ua(\W{1})$ and $j \geq 1$, we have
    \begin{align}
        \phi_k (r, \tilde{u}_{q, n-1}) \varsigma_1 1_\chi &= k a_{q, 0} r^{k-1}\varsigma_1 1_\chi\label{eq:phisame}, \text{ and } \\ 
        \phi_k (r, \tilde{u}_{q, n-1 +j}) \varsigma_0 1_\chi &= 0\label{eq:phi0}.
    \end{align}
\end{corollary}
\begin{proof}
    Again, let $\varsigma$ be $\varsigma_0$ or $\varsigma_1$. For all $i \in \NN$, $r^i \tilde{u}_{q, n-1} r^j \varsigma 1_\chi = c r^{i+j} \varsigma 1_\chi$ by equations~\eqref{eq:tildeuaword1} and~\eqref{eq:tildeuaword2} where $c$ is 0 or $a_{q, 0}$. 
\end{proof}

\begin{lemma}\label{lem:bracketxiui}
    For $1 \leq i \leq m-1$, $k \geq 1$ and $q \in \kk[t] \backslash \{0\}$ with $q(x) \neq 0$, then 
    \begin{equation*}
        [\xi_{q, i}^{k}, u_i]\varsigma_1 1_\chi = k \xi^{k-1}_{q, i} \varsigma_1 1_\chi.
    \end{equation*}
\end{lemma}

\begin{proof}
    In this proof, to simplify notation, let $\xi = \xi_{q, i}$ and $X = \frac{1}{2i-n+1} \tilde{u}_{p, n+o}$. We first show
    \begin{equation}\label{eq:bracketxiuileibniz}
        a_{q, 0}[\xi^k, u_i] = \phi_{k}(\xi, \tilde{u}_{q, n-1}) - \phi_{k}(\xi, X)
    \end{equation}
    by induction on $k$. For $k = 1$, by~\eqref{eq:bracketxi}, we have $a_{q, 0} [\xi, u_i] = \tilde{u}_{q, n-1} - \frac{1}{2i-n+1} \tilde{u}_{p, n+o}$. Then
    \begin{align*}
        a_{q, 0}[\xi^{k+1}, u_i] = a_{q, 0}([\xi^k, u_i]\xi + \xi^k[\xi, u_i]) = \bigl(\phi_{k}(\xi, \tilde{u}_{q, n-1}) - \phi_{k}(\xi, X)\bigr) \xi + \xi^k(\tilde{u}_{q, n-1} - X)
    \end{align*}
    by the induction hypothesis. Note that $\phi_{k+1}(\xi, x) = \phi_{k}(\xi, x) \xi + \xi^k x$, thus $a_{q, 0}[\xi^{k+1}, u_i] = \phi_{k+1}(\xi, \tilde{u}_{q, n-1}) - \phi_{k+1}(\xi, X)$ and we obtain~\eqref{eq:bracketxiuileibniz}.

    We have $\phi_{k} (\xi, X) \varsigma_1 1_\chi = \frac{1}{2j-n+1}\phi_{k} (\xi, \tilde{u}_{p, n+o}) \varsigma_1 1_\chi$ which is 0 by~\eqref{eq:phi0}. So by~\eqref{eq:bracketxiuileibniz}, $[\xi^{k}, u_i]\varsigma_1 1_\chi = \frac{1}{a_{q, 0}} \phi_{k}(\xi, \tilde{u}_{q, n-1}) \varsigma_1 1_\chi = k \xi^{k-1} \varsigma_1 1_\chi$, where the last equation follows from~\eqref{eq:phisame}. 
\end{proof}

We are now prepared to consider the case $u_{-1}$ and $u_{0}$ are not in $z$. This is Lemmas~\ref{lem:xiui0} and~\ref{lem:xiui}.

\begin{lemma}\label{lem:xiui0}
    For $i \geq 1$ and $0 \leq j<i$, we have
   \begin{equation*}
       \xi_{q, j} \varsigma_i 1_\chi = 0.
   \end{equation*}
\end{lemma}
\begin{proof}
   We will induct on the length $\ell$ of $\varsigma_i$. The base case $\ell = 0$ is Lemma~\ref{lem:tildeua}. Assuming the induction hypothesis, we need show that for $i' \geq i > j \geq 0$, then $\xi_{q, j}u_{i'} \varsigma_i 1_\chi = 0$. Note that by~\eqref{eq:bracketxi}
   \begin{align*}
    (2j- n +1)[\xi_{q, j}, u_{i'}] \varsigma_i 1_\chi  = (i'+j+1-n) \tilde{u}_{q, i'+n-1-j} \varsigma_i 1_\chi  - \tilde{u}_{p, i'+n-j+o} \varsigma_i 1_\chi.
   \end{align*}
   Since $i'-j \geq 1$ and $i \geq 1$, then both $\tilde{u}_{q, i'+n-1-j} \varsigma_i 1_\chi$ and $\tilde{u}_{p, i'+n-j+o} \varsigma_i 1_\chi$ are zero by~\eqref{eq:tildeuaword2}. Thus, we have $[\xi_{q, j}, u_{i'}] \varsigma_i 1_\chi= 0$ and so
   \begin{align*}
       \xi_{q, j} u_{i'} \varsigma_i 1_\chi =  u_{i'}\xi_{q, j}\varsigma_i 1_\chi = 0,
   \end{align*}
   where the second equation follows from the induction hypothesis. 
\end{proof}

\begin{lemma}\label{lem:xiui}
    Write $\sigma_{i} = u_i^{k} \sigma_{i+1}$. For $1 \leq i \leq m-1$, and  $q \in \kk[t] \backslash \{0\}$ with $q(x) \neq 0$, then  
    \begin{equation}\label{eq:xiui}
        \xi_{q, i}^{k}\sigma_i 1_\chi =  \xi_{q, i}^{k}u_i^{k} \sigma_{i+1}1_\chi = k! \sigma_{i+1} 1_\chi.
    \end{equation}
\end{lemma}
\begin{proof}
    We will show the lemma by induction on the smallest variable appearing in the word $\sigma_i$, i.e., downward induction on $i$. The base case is when $\sigma$ is the empty word (i.e., the smallest variable is $\infty$), it is tautologically true. Assuming the induction hypothesis, we need to show~\eqref{eq:xiui} for $i \geq 1$ and for all $k \in \NN$.

    We proceed by induction on $k$. The base case $k = 0$ is obvious and for $k = 1$, we have 
    \begin{align*}
        \xi_{q, i}u_{i}\sigma_{i+1} 1_\chi = u_{i}\xi_{q, i}\sigma_{i+1} 1_\chi + [\xi_{q, i}, u_{i}]\sigma_{i+1} 1_\chi.
    \end{align*}
    By Lemma~\ref{lem:xiui0}, $\xi_{q, i}\sigma_{i+1} 1_\chi = 0$. Thus 
    \begin{equation}\label{eq:xiuibase}
        \xi_{q, i}u_{i}\sigma_{i+1} 1_\chi = [\xi_{q, i}, u_{i}]\sigma_{i+1} 1_\chi = \sigma_{i+1} 1_\chi,
    \end{equation}
    where the last equation follows from Lemma~\ref{lem:bracketxiui}. Assume~\eqref{eq:xiui} is true for $k$, then
    \[ \begin{array}{lll}
            \xi_{q, i}^{k+1}u_{i}^{k+1}\sigma_{i+1} 1_\chi
             & = \xi_{q, i}u_{i} \xi_{q, i}^{k}u_{i}^{k}\sigma_{i+1} 1_\chi + \xi_{q, i}[\xi_{q, i}^{k}, u_{i}]u_{i}^{k}\sigma_{i+1} 1_\chi                                                                                      \\
             & = \xi_{q, i}u_{i} \xi_{q, i}^{k}u_{i}^{k}\sigma_{i+1} 1_\chi + k\xi_{q, i}^{k} u_{i}^{k} \sigma_{i+1} 1_\chi          & (\text{by Lemma~\ref{lem:bracketxiui}}) \\
             & = k! \xi_{q, i}u_{i} \sigma_{i+1} 1_\chi  + k k! \sigma_{i+1} 1_\chi                                                           & (\text{by induction})                                                              \\
             & = (k+1)! \sigma_{i+1}1_\chi.                                                                                               & (\text{by~\eqref{eq:xiuibase}})
        \end{array} \]
    Thus, ~\eqref{eq:xiui} is true for all $k \in \NN$ and the lemma is proven.
\end{proof}
Lemmas~\ref{lem:xiui0} and~\ref{lem:xiui} mean that if $j >0$ and $i \geq 1$
\begin{equation*}
    \xi_{q, i}^{k_i+j} u_i^{k_i} \sigma_{i+1} 1_\chi = k_i! \xi_{q, i}^{j} \sigma_{i+1} 1_\chi = 0.
\end{equation*}
We now add $u_0$ and $u_{-1}$ into the picture. We first show the following useful lemma, which gives special elements that annihilate the monomial $u_{-1}^{k_{-1}} \varsigma_0 1_\chi \in M'_\chi$. 

\begin{lemma}\label{lem:u-10}
    For $j >0$, $k_{-1} \in \NN$ and $q \in \kk[t] \backslash \{0 \}$ such that $q(x) \neq 0$, then
    \begin{equation}\label{eq: tensor multi u_{-1} 0}
        \tilde{u}_{q, n-1+j+k_{-1}} u_{-1}^{k_{-1}} \varsigma_0 1_\chi = 0.
    \end{equation}
\end{lemma}
\begin{proof}
    We will prove the lemma by strong induction on $\deg_t q$. For $q$ a constant, without loss of generality we can assume $q = 1$. Thus, we need to show that ${u}_{n-1+j+k_{-1}} u_{-1}^{k_{-1}} \varsigma_0 1_\chi = 0$. We will prove it by induction on $k_{-1}$. The base case ${u}_{n-1+j} \varsigma_0 1_\chi = 0$ is a direct application of~\eqref{eq:tildeuaword2}. Assuming~\eqref{eq: tensor multi u_{-1} 0} is true for $q = 1$ and $k_{-1}$, we then have 
    \begin{align*}
        [{u}_{n+j+k_{-1}+1}, u_{-1}]u_{-1}^{k_{-1}} \varsigma_0 1_\chi = -(n+j+k_{-1}+2) u_{n+j+k_{-1}} u_{-1}^{k_{-1}} \varsigma_0 1_\chi = 0
    \end{align*}
    by the induction hypothesis. Thus ${u}_{n+j+k_{-1}+1}u_{-1}^{k_{-1}+1} \varsigma_0 1_\chi = u_{-1}{u}_{n+j+k_{-1}+1}u_{-1}^{k_{-1}} \varsigma_0 1_\chi$, which is again 0 by the induction hypothesis. We have thus established~\eqref{eq: tensor multi u_{-1} 0} in the base case $\deg_t q = 0$.
    
    For the induction step on $\deg_t q$, we proceed by induction on $k_{-1}$ again with the base case $\tilde{u}_{q, n-1+j} \varsigma_0 1_\chi = 0$ a direct application of~\eqref{eq:tildeuaword2}. Assume~\eqref{eq: tensor multi u_{-1} 0} is true for $k_{-1}$. Then by~\eqref{eq:bracketxi},
    \begin{equation*}
    [\tilde{u}_{q, n+j+k_{-1}}, u_{-1}]u_{-1}^{k_{-1}} \varsigma_0 1_\chi = -(n+j+k_{-1}+1) \tilde{u}_{q, n-1+j+k_{-1}} u_{-1}^{k_{-1}} \varsigma_0 1_\chi   - \tilde{u}_{p, n+j+k_{-1}+o} u_{-1}^{k_{-1}} \varsigma_0 1_\chi,
    \end{equation*}
    where the first term of the RHS is 0 by the induction hypothesis on $k_{-1}$ and the second term of the RHS is 0 by the induction hypothesis on $\deg_t p < \deg_t q$. Thus $ [\tilde{u}_{q, n+j+k_{-1}}, u_{-1}]u_{-1}^{k_{-1}} \varsigma_0 1_\chi =0$, and 
    \begin{align*}
        \tilde{u}_{q, n+j+k_{-1}} u_{-1}^{k_{-1}+1} \varsigma_0 1_\chi = u_{-1} \tilde{u}_{q, n+j+k_{-1}} u_{-1}^{k_{-1}}  \varsigma_0 1_\chi = 0,
    \end{align*}
    where the last equation follows from induction on $k_{-1}$. We have obtained the induction step for $\deg_t q$. 
\end{proof}

The following lemma is useful in our computation that involve the commutator of elements of special form.
\begin{lemma}\label{lem:simleibniz}
    Let $\g$ be a Lie algebra and let $a, b \in \g$ such that $[b, a] = x + y$, and $[b, x] = 0$, then in $\Ua(\g)$ we have
    \begin{equation}\label{eq:simleibniz}
        [b^k, a] = k b^{k-1} x + \phi_k(b, y).
    \end{equation}
\end{lemma}
\begin{proof}
    This is a straightforward induction argument. 
\end{proof}
We are now ready to prove the following two lemmas, which give the main computation when we add $u_{-1}$ and $u_{0}$ into the picture.
\begin{lemma}\label{lem:xi0u0}
    Let $k_0 \in \ZZ_{\geq 0}$. Then 
    \begin{equation*}
        \xi_{q, 0}^{k_0}u_0^{k_0}\sigma_1 1_\chi = k_0!\sigma_1 1_\chi.
    \end{equation*}
\end{lemma}
\begin{proof}
    By~\eqref{eq:bracketxi}, $[\xi_{q, 0}, u_0] = \frac{1}{(1-n)a_{q, 0}} ((1-n)\tilde{u}_{q, n-1} - \tilde{u}_{p, n+o})$. As $[\xi_{q, 0}, \tilde{u}_{q, n - 1}] = 0$, by Lemma~\ref{lem:simleibniz}
    \begin{equation}\label{eq:bracketxi0u0}
        [\xi^{k_0}_{q, 0}, u_0] = \frac{k_0}{a_{q, 0}} \xi^{k_0}_{q, 0}\tilde{u}_{q, n-1} - \phi_{k_0} (\xi_{q, 0}, \frac{1}{1-n}\tilde{u}_{p, n+o}) 
    \end{equation}
    We will now prove the lemma by induction on $k_0$. The base case $k_0 =0$ is tautologically true. For $k_0 = 1$, by~\eqref{eq:bracketxi0u0}
    \begin{align*}
     [\xi_{q, 0}, u_0]\sigma_1 1_\chi = \frac{1}{(1-n)a_{q, 0}} \bigl((1-n)\tilde{u}_{q, n-1} - \tilde{u}_{p, n+o}\bigr)\sigma_1 1_\chi.
    \end{align*}
    By~\eqref{eq:tildeuaword1}, $\tilde{u}_{q, n-1}\sigma_1 1_\chi = a_{q, 0} \sigma_1 1_\chi$, and by~\eqref{eq:tildeuaword2}, $\tilde{u}_{p, n+o}\sigma_1 1_\chi = 0$, thus $[\xi_{q, 0}, u_0]\sigma_1 1_\chi = \sigma_1 1_\chi$. By Lemma~\ref{lem:xiui0}, $u_0\xi_{q, 0}\sigma_1 1_\chi = 0$, we thus have 
    \begin{align*}
        \xi_{q, 0}u_0\sigma_1 1_\chi =\xi_{q, 0}u_0\sigma_1 1_\chi - u_0 \xi_{q, 0}\sigma_1 1_\chi= [\xi_{q, 0}, u_0]\sigma_1 1_\chi = \sigma_1 1_\chi.
    \end{align*}
    Assume the lemma is true for $k_0$. By~\eqref{eq:phi0} we have $\phi_{k_0} (\xi_{q, 0}, \frac{1}{1-n}\tilde{u}_{p, n+o}) u_0^{k_0} \sigma_1 1_\chi = 0$. By~\eqref{eq:bracketxi0u0}
    \begin{equation}\label{eq:bracketxi0u0leibniz}
        [\xi^{k_0}_{q, 0}, u_0] u_0^{k_0} \sigma_1 1_\chi = \frac{k_0}{a_{q, 0}} \xi^{k_0-1}_{q, 0}\tilde{u}_{q, n-1} u_0^{k_0} \sigma_1 1_\chi.
    \end{equation}
    We then have
    \[
        \begin{array}{lll}
            \xi_{q, 0}^{k_0+1}u_0^{k_0+1}\sigma_1 1_\chi & = \xi_{q, 0} u_0 \xi_{q, 0}^{k_0} u_0^{k_0}\sigma_1 1_\chi + \xi_{q, 0} [\xi_{q, 0}^{k_0}, u_0]u_0^{k_0}\sigma_1 1_\chi                                    &                                              \\
        & = \xi_{q, 0} u_0 \xi_{q, 0}^{k_0} u_0^{k_0}\sigma_1 1_\chi  + \frac{k_0}{a_{q, 0}}\xi_{q, 0}\tilde{u}_{q, n-1}\xi_{q, 0}^{k_0-1}u_0^{k_0}\sigma_1 1_\chi & (\text{by~\eqref{eq:bracketxi0u0leibniz}}) \\
        & = \xi_{q, 0} u_0 \xi_{q, 0}^{k_0} u_0^{k_0}\sigma_1 1_\chi  + \frac{k_0}{a_{q, 0}}\tilde{u}_{q, n-1}\xi_{q, 0}^{k_0}u_0^{k_0}\sigma_1 1_\chi & (\text{as } [\xi_{q, 0}, \tilde{u}_{q, n-1}] = 0) \\
        & = k_0!\sigma_1 1_\chi + \frac{k_0}{a_{q, 0}} k_0! \tilde{u}_{q, n - 1} \sigma_1 1_\chi                                                                            & (\text{by induction})   \\ 
        & = k_0!\sigma_1 1_\chi + k_0 k_0! \sigma_1 1_\chi                                                                            & (\text{by~\eqref{eq:tildeuaword1}})                     \\
        & = (k_0+1)!\sigma_1 1_\chi.                                                                                    &
        \end{array}
    \]
\end{proof}
\begin{lemma}\label{lem:cancelu-1}
    Let $k_0, k_{-1} \in \mathbb{Z}_{\geq 0}$. We have
    \begin{equation*}
        \xi_{q, 0}^{k_0} \tilde{u}_{q, n-1+k_{-1}}u^{k_{-1}}_{-1}  u^{k_0}_{0} \sigma_1 1_\chi = C \sigma_1 1_\chi, \quad \text{ for some } C \in \kk^\times.
    \end{equation*}
\end{lemma}

\begin{proof}
    We will prove this by induction on $k_{-1}$. For the base case if $k_{-1} = 0$, then
    \[
        \begin{array}{lll}
            \xi_{q, 0}^{k_0}\tilde{u}_{q, n-1}u_{0}^{k_{0}} \sigma_1 1_\chi & = \tilde{u}_{q, n-1} \xi_{q, 0}^{k_0} u_0^{k_0} \sigma_1 1_\chi & (\text{as } [\xi_{q, 0}, \tilde{u}_{q, n-1}] = 0) \\
            & = k_0! \tilde{u}_{q, n-1} \sigma_1 1_\chi            & (\text{by Lemma~\ref{lem:xi0u0}})                  \\
            & = k_0! a_{q, 0} \sigma_1 1_\chi  & (\text{by~\eqref{eq:tildeuaword1}})
        \end{array}
    \]
    By~\eqref{eq:brackettilde}, $[\tilde{u}_{q, n+k_{-1}}, u_{-1}] u_{-1}^{k_{-1}}\sigma_0 1_\chi = - (n + 1 + k_{-1})  \tilde{u}_{q, n-1+k_{-1}} u_{-1}^{k_{-1}}\sigma_0 1_\chi - \tilde{u}_{p, n+k_{-1}+o} u_{-1}^{k_{-1}} \sigma_0 1_\chi$ and note that by Lemma~\ref{lem:u-10} $\tilde{u}_{p, n+k_{-1}+o} u_{-1}^{k_{-1}} \sigma_0 1_\chi = 0$. Thus
    \begin{align*}
        \xi_{q, 0}^{k_0} \tilde{u}_{q, n+k_{-1}} u_{-1}^{k_{-1}+1} \sigma_0 1_\chi & =  \xi_{q, 0}^{k_0} u_{-1} \tilde{u}_{q, n+k_{-1}} u_{-1}^{k_{-1}} \sigma_0 1_\chi - (n+k_{-1}+1) \xi_{q, 0}^{k_0} \tilde{u}_{q, n+k_{-1}-1} u_{-1}^{k_{-1}} \sigma_0 1_\chi \\
        & = 0 - (n+1+k_{-1}) C \sigma_1 1_\chi,
    \end{align*}
    for some $C \in \kk^\times$, where the first term is 0 by Lemma~\ref{lem:u-10} again, and we use the induction hypothesis for the second term. We obtain the induction step and the lemma follows. 
\end{proof}

The following lemma explains why we only need to look at the leading term of an arbitrary element in $M'_\chi$. 

\begin{lemma}\label{lem:xi0general}
    Let $k_i, k'_i \in \NN$ for $i = -1, \dots , m-1$. Let $q \in \kk[t] \backslash \{ 0\}$ with $q(x) \neq 0$. If there exists $i$ such that $k_i > k'_i$, then 
    \begin{equation}\label{eq:xi0general}
        \xi^{k_{m-1}}_{q, m-1}\dots  \xi^{k_1}_{q, 1} \xi^{k_0}_{q, 0} \tilde{u}_{q, k_{-1}+n-1} u_{-1}^{k'_{-1}}u_{0}^{k'_{0}}\dots u_{m-1}^{k'_{m-1}}1_\chi = 0.
    \end{equation}
\end{lemma}
\begin{proof}
    Let $Y_q = \xi^{k_{m-1}}_{q, m-1}\dots  \xi^{k_1}_{q, 1} \xi^{k_0}_{q, 0} \tilde{u}_{q, k_{-1}+n-1}$ and $z' = u_{-1}^{k'_{-1}}u_{0}^{k'_{0}}\dots u_{m-1}^{k'_{m-1}}1_\chi$ so we seek to prove that $Y_q z' = 0$. Let $i$ be the minimal integer such that $k'_i < k_i$. We now have three cases. If $i=-1$, then Lemma~\ref{lem:u-10} gives us $Y_q z' = 0$. If $i = 0$, then by Lemma~\ref{lem:cancelu-1}, we have 
    \begin{align*}
        Y_qz' = \xi^{k_{m-1}}_{q, m-1} \dots  \xi^{k_0 - k'_0}_{q, 0} \xi^{k'_0}_{q, 0}  u_{n-1+k_{-1}} u_{-1}^{k_{-1}} u_{0}^{k'_{0}}\dots u_{m-1}^{k'_{m-1}} 1_\chi = c_1 \xi^{k_{m-1}}_{q, m-1} \dots  \xi^{k_0 - k'_0}_{q, 0} u_{1}^{k'_{1}}\dots u_{m-1}^{k'_{m-1}} 1_\chi,
    \end{align*}    
    which is then $0$ by Lemma~\ref{lem:xiui0}. Otherwise, if $i \geq 1$, then 
        \[ 
            \begin{array}{lll}
                Y_qz' &= \xi^{k_{m-1}}_{q, m-1}\dots \xi^{k_i-k_i'}_{q,i} \xi^{k_i'}_{q, i} \dots  \xi^{k'_0}_{q, 0}  u_{n-1+k_{-1}} u_{-1}^{k_{-1}} u_{0}^{k'_{0}}\dots u_{m-1}^{k'_{m-1}} 1_\chi &\\
                &= c_1 \xi^{k_{m-1}}_{q, m-1}\dots  \xi^{k_i-k_i'}_{q, i} \xi^{k_i'}_{q, i} \dots  \xi^{k'_1}_{q, 1} u_{1}^{k'_{1}}\dots u_{m-1}^{k'_{m-1}}1_\chi & (\text{for } c_1 \neq 0 \text{ by Lemma~\ref{lem:cancelu-1}}) \\
                &= c_2 \xi^{k_{m-1}}_{q, m-1}\dots  \xi^{k_i-k_i'}_{q, i} u^{k'_{i+1}}_{i+1} \dots u^{k'_{m-1}}_{m-1} 1_\chi & (\text{for } c_2 \neq 0 \text{ by Lemma~\ref{lem:xiui}}) \\
                &= 0 &( \text{by Lemma~\ref{lem:xiui0}}).
            \end{array}
        \]
\end{proof}

We now show the main theorems of this subsection. We first show that for a one-point local function $\chi \neq \chi_{x; \alpha_0, \frac{1}{2}}$ on $\W{-1}$, the canonical local representations $M'_\chi$ is irreducible.

\begin{theorem}\label{theo:simplicityone}
    Let $\chi = \chi_{x; \alpha_0, \dots , \alpha_n}$ be a one-point local function on $\W{-1}$, where $\alpha_n \neq 0$. Then $M'_\chi$ is a simple $\Ua(\W{-1})$-module if $\chi \neq \chi_{x; \alpha_0, \frac{1}{2}}$.
\end{theorem}
The inverse direction of Theorem~\ref{theo:simplicityone} is correct, and is a direct result of Proposition~\ref{prop:untwistrephelp} and \cite{petukhov2022poisson}*{Proposition 7.1.5}, which we will discuss in Subsection~\ref{subsec:primitive}. 
\begin{proof}
   The canonical polarization of $\chi$ in $\W{-1}$ is $\p_\chi = \W{-1}((t-x)^{m+1})$, thus $\W{-1} = \p \oplus \kk \{u_{-1}, \dots, u_{m-1}\}$. Applying the PBW theorem with the order $u_{-1} > \dots > u_{m-1}$, we have a total lexicographic order $\mc{C}$ on the basis $S = \{u_{-1}^{k_{-1}} \dots u_{m-1}^{k_{m-1}} 1_\chi\}$ of $M'_\chi$. For any element in $z \in M'_\chi$, we denote 
   \begin{align*}
    z_\text{max}\coloneqq \LT_{\mc{C}} (z) = c_0 u^{k_{-1}}_{-1}  u^{k_0}_{0} \dots  u^{k_{m-1}}_{m-1} 1_\chi,
   \end{align*}
   for some $c_0 \in \kk^\times$. We set $q = 1$ (and thus omit it) and define $Y\coloneqq \xi^{k_{m-1}}_{m-1}\dots  \xi^{k_1}_1 \xi^{k_0}_0 u_{k_{-1}+n-1}$. We will show that $Yz = Y z_\text{max} = c 1_\chi$ for $c \in \kk^\times$. 

   From Lemma~\ref{lem:xi0general}, for every monomial $z' = u_{-1}^{k'_{-1}}u_{0}^{k'_{0}}\dots u_{m-1}^{k'_{m-1}} 1_\chi$ occuring in $z - z_\text{max}$, then $Yz' = 0$. Thus $Yz = Yz_\text{max} = c_0 Y u^{k_{-1}}_{-1}  u^{k_0}_{0} \dots  u^{k_{m-1}}_{m-1} 1_\chi$. By Lemma~\ref{lem:cancelu-1} and~\ref{lem:xiui}
    \begin{align*}
        Yz = Yz_\text{max} =  c_0 C 1_\chi,
    \end{align*}
    for some $C \in \kk^\times$. Thus $M'_\chi$ is a simple $\Ua(\W{-1})$-module.
\end{proof}
We then extend the result to the general case of multi-point local functions on $\W{-1}$. 
\begin{theorem}\label{theo:simplicitymulti}
    Let $\chi$ be a local function on $\W{-1}$ following Notation~\ref{not:localfunc}. Then $M'_\chi$ is a simple $\Ua(\W{-1})$-module if its component $\chi_d \neq \chi_{x_d; \alpha_0, \frac{1}{2}}$ for all $d$. 
\end{theorem}
The inverse direction of Theorem~\ref{theo:simplicitymulti} is proven in Corollary~\ref{cor:simplicityconverse}. 
\begin{proof}
    For $1 \leq d \leq \ell$, the canonical polarization of $\chi_d$ is $\mathfrak{p}_{\chi_d} = \W{-1}((t-x_d)^{m_d+1})$, and $u_{x_d, i} = (t-x_d)^{i+1}\partial$ for $i \in \ZZ$. The PBW basis for $M'_\chi$ is
    \begin{align*}
        S  = \{ u_{-1, x_1}^{k_{-1, 1}} \dots u_{ m_1-1, x_1}^{k_{m_1-1, 1}} 1_{\chi_1} \otimes \dots \otimes u_{-1, x_\ell}^{k_{-1, \ell}} \dots u_{m_\ell-1, x_\ell}^{k_{m_\ell-1, \ell}}1_{\chi_\ell} | k_{i, j} \in \NN \}.
    \end{align*}
    We define a total order 
    \begin{equation}\label{eq: monoidal order simple}
        \mc{C}: u_{-1, x_1} > \dots > u_{m_1-1, x_1} > u_{2, -1} > \dots > u_{2, m_2-1} > \dots > u_{-1, x_\ell} > \dots > u_{\ell, m_\ell}
    \end{equation}
    and order $S$ lexicographically. Thus, for an arbitrary element $z \in M'_\chi$, we denote 
    \begin{align*}
        z_{\text{max}}\coloneqq \LT_\mc{C}(z) = c_0 u_{-1, x_1}^{k_{-1, 1}}u_{0, x_1}^{k_{0, 1}}\dots u_{m_1-1, x_1}^{k_{m_1-1, 1}}  1_{\chi_1}
        \otimes \dots  \otimes 
        u_{-1, x_\ell}^{k_{-1, \ell}}u_{0, x_\ell}^{k_{0, \ell}}\dots u_{m_\ell-1, x_\ell}^{k_{m_\ell-1, \ell}} 1_{\chi_\ell}, 
    \end{align*} 
    for some $c_0 \in \kk^\times$. For $d = 2, \dots, \ell$, let $K_{d}$ be the maximal power of $u_{-1, d}$ in $z$ and let
    \begin{equation}
        q_1 = (t-x_2)^{n_2+1+K_{2}} \dots  (t-x_\ell)^{n_\ell+1+K_{\ell}}.
    \end{equation}
    Note that $q_1(x_1) \neq 0$ and we define $\xi_{q_1, j, x_1}$ and $\tilde{u}_{q_1, j, x_1}$ as in Definition~\ref{def:u,a,xi}. For $d > 1$, let $\tilde{q}_{1, d} = q_1/(t-x_d)^{n_d+1+K_{d}}$, note that  $\tilde{q}_{1,d}  \in \kk[t] \backslash \{0 \}$ and $\tilde{q}_{1, d} (x_d) \neq 0$. We furthermore have 
    \begin{align}
        \tilde{u}_{q_1, j, x_1} = \tilde{u}_{(t-x_1)^j \tilde{q}_{1, d}, n_d+K_{d}, x_d} \quad \text{for all } j \in \NN.
    \end{align}
    For any monomial $w = u_{-1, x_1}^{k'_{-1, 1}}\dots u_{m_1-1, x_1}^{k'_{m_1-1, 1}} 1_{\chi_1} \otimes \dots  \otimes u_{-1, x_\ell}^{k'_{-1, \ell}}\dots u_{m_\ell-1, x_\ell}^{k'_{m_\ell-1, \ell}} 1_{\chi_\ell}$ occuring in $z$, by definition $k'_{-1, d} \leq K_d$ for all $d$, thus by Lemma~\ref{lem:u-10}, for all $j \geq 0$ and $d >1$
    \begin{equation*}
        \tilde{u}_{q_1, j, x_1} \cdot u_{-1, x_d}^{k'_{-1, d}}u_{0, x_d}^{k'_{0, d}}\dots u_{m_d-1,  x_d}^{k'_{m_d-1, d}} 1_{\chi_d} = \tilde{u}_{(t-x_1)^j \tilde{q}_{1, d}, n_d+K_{d}, x_d}\cdot u_{-1, x_d}^{k'_{-1, d}}u_{0, x_d}^{k'_{0, d}}\dots u_{m_d-1, x_d}^{k'_{m_d-1, d}} 1_{\chi_d}=  0.
    \end{equation*}
    Hence, 
    \begin{equation}\label{eq: theo simple}
        \xi_{q_1, j, x_1} \cdot w = \Delta^\ell(\xi_{q_1, j, x_1}) (w) = (\xi_{q_1, j, x_1} u_{-1, x_1}^{k'_{-1, 1}} \dots u_{m_1-1, x_1}^{k'_{m_1-1, 1}} 1_{\chi_1})
        \otimes \dots \otimes u_{-1, x_\ell}^{k'_{-1, \ell}}\dots u_{m_\ell-1, x_\ell}^{k'_{m_\ell-1, \ell}} 1_{\chi_\ell}.
    \end{equation}
    Let 
    \begin{equation*}
        Y_1 = \xi_{q_1, m_1-1, x_1}^{k_{m_1-1, 1}} \dots  \xi_{q_1, 0, x_1}^{k_{1, 0}} \tilde{u}_{q_1, n_1 -1 + k_{-1, 1}, x_1}.
    \end{equation*}
    By~\eqref{eq: theo simple}, we have
    \begin{equation}\label{eq:theosimple2}
        Y_1 \cdot w = \Delta^\ell(Y_1) (w) =(Y_1 u_{-1, x_1}^{k'_{-1, 1}}\dots u_{m_1-1, x_1}^{k'_{m_1-1, 1}} 1_{\chi_1}) 
        \otimes \dots  \otimes u_{-1, x_\ell}^{k'_{-1, \ell}}\dots u_{m_\ell-1, x_\ell}^{k'_{m_\ell-1, \ell}} 1_{\chi_\ell}.
    \end{equation}
    By Lemma~\ref{lem:xi0general} and Lemma~\ref{lem:cancelu-1} applied to $M_{\chi_1}$, $Y_1 w \neq 0$ if and only if $k'_{1, i} = k_{1, i}$ for $i=-1, \dots , m-1$. By Lemma~\ref{lem:cancelu-1} and~\ref{lem:xiui}, $Y_1 z \neq 0$ and only contains monomial of the form $1_{\chi_1} \otimes u_{-1, x_2}^{k'_{-1, 2}}\dots u_{m_2-1, x_2}^{k'_{m_2-1, 2}} 1_{\chi_2} \otimes \dots  \otimes u_{-1, x_\ell}^{k'_{-1, \ell}}\dots u_{m_\ell-1, x_\ell}^{k'_{m_\ell-1, \ell}} 1_{\chi_\ell}$. We then continue the process and for $d \neq 2$, let $K'_{d}$ be the maximal power of $u_{-1, d}$ in $Y_1 z$ (note that $K'_1 = 0$) and let
    \begin{equation}
        q_2 = (t-x_1)^{n_1+ 1+ K'_1} (t-x_3)^{n_3+1+K'_{3}} \dots  (t-x_\ell)^{n_\ell+1+K'_{\ell}}.
    \end{equation}
    Noting that $q_2(x_2) \neq 0$, we then define $\xi_{q_2, j, x_2}$ and $\tilde{u}_{q_2, j, x_2}$ as in Definition~\ref{def:u,a,xi}. For $d \neq 2$, let $\tilde{q}_{2, d} = q_2/(t-x_d)^{n_d+1+K'_{d}}$, note that $\tilde{q}_{2, d}  \in \kk[t] \backslash \{0 \}$ and $\tilde{q}_{2, d} (x_d) \neq 0$. We furthermore have 
    \begin{align}
        \tilde{u}_{q_2, j, x_2} = \tilde{u}_{(t-x_2)^j \tilde{q}_{2, d}, n_d+K'_{d}, x_d}.
    \end{align}
    Similar to~\eqref{eq: theo simple} for any monomial $w = 1_{\chi_1} \otimes u_{-1, x_2}^{k'_{-1, 2}}\dots u_{m_2-1, x_2}^{k'_{m_2-1, 2}} 1_{\chi_2} \otimes \dots  \otimes u_{-1, x_\ell}^{k'_{-1, \ell}}\dots u_{m_\ell-1, x_\ell}^{k'_{m_\ell-1, \ell}} 1_{\chi_\ell}$ in $Y_1 z$, by Lemma~\ref{lem:u-10}, we have for all $j \geq 0$
    \begin{align*}
        \tilde{u}_{q_2, j, x_2} \cdot 1_{\chi_1} &= \tilde{u}_{(t-x_2)^j \tilde{q}_{2, 1}, n_1, x_1}\cdot 1_{\chi_1}=  0 \quad \text{ and }\\ 
        \tilde{u}_{q_2, j, x_2} \cdot u_{-1, x_d}^{k'_{-1, d}}\dots u_{m_d-1,  x_d}^{k'_{m_d-1, d}} 1_{\chi_d} &= \tilde{u}_{(t-x_2)^j \tilde{q}_{2, d}, n_d+K'_{d, -1}, x_d}\cdot u_{-1, x_d}^{k'_{-1, d}}\dots u_{m_d-1,  x_d}^{k'_{m_d-1, d}} 1_{\chi_d}=  0 \text{  for } d > 2.
    \end{align*}
    Hence
    \begin{align*}
        \xi_{q_2, j, x_2} \cdot w = \Delta^\ell(\xi_{q_2, j, x_2}) (w) = 1_{\chi_1} \otimes (\xi_{q_2, j, x_2} u_{-1, x_2}^{k'_{-1, 2}}\dots u_{m_2-1, x_2}^{k'_{m_2-1, 2}} 1_{\chi_2})
        \otimes \dots \otimes u_{-1, x_\ell}^{k'_{-1, \ell}}u_{0, x_\ell}^{k'_{0, \ell}}\dots u_{m_\ell-1, x_\ell}^{k'_{m_\ell-1, \ell}} 1_{\chi_\ell}.
    \end{align*}
    Let
    \begin{equation*}
        Y_2 = \xi_{q_2, m_2-1, x_2}^{k_{2, m_2-1}} \dots  \xi_{q_2, 0, x_2}^{k_{2, 0}} \tilde{u}_{q_2, n_2 -1 + k_{2, -1}, x_2}.
    \end{equation*}
    As in~\eqref{eq:theosimple2}, by Lemma~\ref{lem:xi0general} applied to $M_{\chi_2}$, $Y_2 w \neq 0$ if and only if $k'_{2, i} = k_{2, i}$ for $i=-1, \dots , m-1$. By Lemma~\ref{lem:cancelu-1} and~\ref{lem:xiui}, $Y_2 \cdot (Y_1 z) \neq 0$ and only contains monomials of the form $1_{\chi_1} \otimes 1_{\chi_2} \otimes u_{3, -1}^{k'_{3, -1}}\dots u_{3, m_3-1}^{k'_{3, m_3-1}} 1_{\chi_3} \otimes \dots  \otimes u_{-1, x_\ell}^{k'_{-1, \ell}}\dots u_{m_\ell-1, x_\ell}^{k'_{m_\ell-1, \ell}} 1_{\chi_\ell}$. We then repeat the process and define analogously $q_3, \dots, q_\ell$ and $Y_3, \dots, Y_\ell$ so that
    \begin{align*}
        Y_\ell \dots  Y_2 Y_1 \cdot z = C 1_{\chi_1} \otimes 1_{\chi_2} \otimes \dots  \otimes 1_{\chi_\ell},
    \end{align*}
    for $C \in \kk^\times$. Thus $M'_{\chi_1} \otimes \dots  \otimes M'_{\chi_\ell}$ is a simple $\Ua(\W{-1})$-module.
\end{proof}

\subsection{Primitive ideals from the orbit method}\label{subsec:primitive}
Let $\g$ be either $\W{-1}, W$ or $\Vir$. Let $\chi$ be a local function on $\g$ and adopt Notation~\ref{not:localfunc}. In this subsection, we give another characterization of the local representation $M'_\chi$, prove the converse of Theorem~\ref{theo:simplicitymulti}, and show that all $Q_\chi$ are indeed primitive.

\begin{lemma}\label{lem:maxideals}
    Let $\mathfrak{g}$ be any Lie algebra and $\chi \in \mathfrak{g}^\ast$. Let $J \coloneqq J_\chi$ be the left ideal of $\Ua(\g)$ generated by $g - \chi(g)$ for $g \in \g$, i.e., $J = \Ua(\mathfrak{g})\cdot(\{g - \chi(g)|g \in \mathfrak{g}\})$. Then
    \begin{align*}
        \Ua(\mathfrak{g})/J = 
        \begin{cases}
            \kk  & \text{ if } \chi|_{[\mathfrak{g}, \mathfrak{g}]}= 0,\\
            0 & \text{ otherwise}.
           \end{cases}
    \end{align*}
\end{lemma}
\begin{proof}
    We first show that in both case, we have $J + \kk = \Ua(\mathfrak{g})$. Let us consider any arbitrary monomial $g_1 \dots g_k +J$, where $g_i \in \mathfrak{g}$. We have 
    \begin{align*}
        g_1 \dots g_k + J &= \chi(g_k) g_1 \dots g_{k-1} + J \quad \quad(\text{as } g_k - \chi(g_k) \in J)\\ 
        &= \dots = \prod_{i=1}^{k}\chi(g_i) + J. 
    \end{align*}
    As the monomials $g_1 \dots g_k$ for $g_i \in \mathfrak{g}$ span $\Ua(\mathfrak{g})$, we have $\Ua(\g) \subseteq J + \kk$, thus $J + \kk = \Ua(\mathfrak{g})$.

    Suppose there exist any $g, h \in \mathfrak{g}$ such that $\chi([g, h]) \neq 0$. Since $g - \chi(g)$ and $h - \chi(h) \in J$, we have $[g, h] = [g - \chi(g), h - \chi(h)] \in J$. As $[g, h] - \chi([g, h]) \in J$, thus $\chi([g, h]) \in J$, and $J = \Ua(\mathfrak{g})$. 

    Otherwise, if $\chi|_{[\mathfrak{g}, \mathfrak{g}]} = 0$, then $\chi$ is a character of $\mathfrak{g}$ and there exists a one-dimensional left $\Ua(\mathfrak{g})$-module $\kk_\chi$, where $g \cdot 1_\chi = \chi(g) 1_\chi$ for $g \in \mathfrak{g}$. Note that $J$ is exactly defined as the kernel of the left $\Ua(\mathfrak{g})$-action on $\kk_\chi$, so $\Ua(\mathfrak{g})/J \simeq \kk$.
\end{proof}

\begin{lemma}\label{lem:rewriterep}
    Let $\g$ be a Lie algebra and $\p \subseteq \g$ a subalgebra. Let $\epsilon$ be any one-dimensional representation $\kk_\epsilon$ of $\p$, $\{f_i\}_{i \in I}$ be a basis for $\p/[\p, \p]$, and $M_\epsilon$ be the induced module over $\Ua(\g)$, i.e., $M_\epsilon = \Ua(\g) \otimes_{\Ua(\p)} \kk_\epsilon$. Then we can write $M_\epsilon = \Ua(\g)/J$, where $J$ is a left-ideal of $\Ua(\g)$ generated by   
    \begin{align}\label{eq: generators}
        \{[\p, \p], f_i - \epsilon(f_i) | i \in I\}.
    \end{align}
\end{lemma}
\begin{proof}
    Write the one-dimensional representation  $\kk_\epsilon$ of $\Ua(\p)$ as $\Ua(\p)/K$, where $K$ is a maximal left ideal of $\Ua(\p)$. Then the induced module $M_\epsilon$ can be written as
    \begin{align*}
        M_\epsilon = \Ua(\g) \otimes_{\Ua(\p)} (\Ua(\p)/K) = \Ua(\g)/\Ua(\g)K.
    \end{align*}
    Thus, we can write $J$ as $\Ua(W)K$. Let $K'$ be the left ideal of $\Ua(\p)$ generated by~\eqref{eq: generators}. Let $1_\epsilon$ be a generator of $\kk_\epsilon$. As $[\p, \p]\cdot 1_\epsilon = 0$, hence $[\p, \p] \subset K$. Moreover, $f_i \cdot 1_\epsilon = \epsilon(f_i) 1_\chi$ for $i \in I$, thus $K \supseteq K'$. Applying Lemma~\ref{lem:maxideals} for $\epsilon \in \p^*$, $\epsilon|_{[\p, \p]} = 0$ and note that $\p = [\p, \p] \oplus \kk \{f_i | i \in I \}$, we see that $K'$ is maximal, hence $K = K'$. 
\end{proof}

The following lemma allows us to rewrite the canonical local representation $M'_\chi$ as an untwisted local representation $M_\eta$ (defined in \eqref{eq:untwistrep}) of a possible different local function $\eta$.  
\begin{proposition}\label{prop:untwistrephelp}
    Let $\g$ be $\W{-1}, W$ or $\Vir$ and let $\chi = \chi_{x; \alpha_0, \dots, \alpha_n}$ be a nonzero one-point local function on $\g$ with $\alpha_n \neq 0$. Let $m = \left\lfloor \frac{n}{2}\right\rfloor$, then 
    \begin{equation}
        M'_\chi \cong 
        \begin{cases}
            M_\chi & \text{ if } m > 0 \\ 
            M_{\chi_{x; \alpha_0, \alpha_1-\frac{1}{2}}} &\text{ if } m = 0.
        \end{cases}
    \end{equation}
\end{proposition}
\begin{proof}
    If $m >0$, then $n>1$ and by Lemma~\ref{lem:rho0}, $\rho_{\p_\chi, \W{-1}}((t-x)^k h \del) = 0$ for $k \geq m+1$. Thus for the twisted one-dimensional representation $\kk'_{\chi, \p, \g}$ defined in \eqref{eq:twistedone}, we have $f\del \cdot 1_\chi = \chi(f\del)$ for $f\del \in \p_\chi$. Then by Lemma~\ref{lem:rewriterep}, we have 
    \begin{equation*}
        M'_\chi \cong \Ua(\g)/\Ua(\g) \bigl([\p_\chi, \p_\chi], u_i - \chi(u_i) | i =m, \dots, 2m-1 \bigr) = M_\chi.
    \end{equation*}

    On the other hand, if $n = 0$ or $1$, then by~\eqref{eq:espu}, we have $\rho_{\p, \g}((t-x)q\del) = -\frac{1}{2}q(x)$. Thus, by Lemma~\ref{lem:rewriterep}, we have for $(\alpha_0, \alpha_1) \neq (0, 0)$
    \begin{align*}
        M'_{\chi_{x; \alpha_0, \alpha_1}} \cong \Ua(\g)/\Ua(\g)([\p_\chi, \p_\chi], u_0 - \alpha_1+\frac{1}{2}) \cong M_{\chi_{x; \alpha_0, \alpha_1-\frac{1}{2}}}.
    \end{align*}
\end{proof}

\begin{notation}\label{not:chi'}
    Let $\g$ be $\W{-1}, W$ or $\Vir$ and let $\chi = \chi_1 + \dots + \chi_\ell$ be a local function on $\g$ following Notation~\ref{not:localfunc}. Suppose $\chi_d = \chi_{x_d, \alpha_{0, d}, ..., \alpha_{d, n_d}}$, we define 
    \begin{align*}
        \chi'_d = \begin{cases}
            \chi_d & \text{ if } m_d > 0 \\
            \chi_{x_d, \alpha_{0, d}, \alpha_{1, d}-\frac{1}{2}} & \text{ if } m_d = 0.
        \end{cases}
    \end{align*}
    We then define $\chi' = \chi'_1 + \dots + \chi'_d$, and called it the \emph{twisted local function of} $\chi$. This notation is chosen so that $M'_\chi \cong M_{\chi'}$ (cf. Proposition~{prop:untwistrep}).  
\end{notation}

The following lemma shows that the twisted local function $\chi'$ defined above it lies in the same pseudo-orbit of $\chi$ (Definition~\ref{def:psorbit}), and thus changing these local functions does not affect our strong Dixmier map $\widetilde{\Dx}^\g$. 

\begin{lemma}\label{lem:obritchi'}
    Let $\g$ be $W, \W{-1}$ or $\Vir$ and let $\chi, \eta$ be local functions on $\g$. Then
    \begin{equation}
        \orbit{\chi} = \orbit{\eta} \Leftrightarrow \orbit{\chi'} = \orbit{\eta'}
    \end{equation} 
\end{lemma}
\begin{proof}
    By ~\cite{petukhov2022poisson}*{Theorem 4.3.1}, we reduce to the case of one-point local functions, i.e., it is equivalent to show that $\orbit{\chi_i} = \orbit{\eta_i} \Leftrightarrow \orbit{\chi'_i} = \orbit{\eta'_i}$. Then the statement follows immediately from the definition of $\chi'_i$ and $\eta'_i$, and ~\cite{petukhov2022poisson}*{Theorem 4.2.1 c'}. 
\end{proof}

The next proposition can be rephrased as the ``preservation of prime/twist'', i.e., we can rewrite the canonical local representation $M'_\chi$ (the induced module from a one-dimesional twist $\kk'_{\chi}$) as the untwisted canonical local representation $M_{\chi'}$ (the induced module from the twisted local function $\chi'$). This also explains our choice of Notation~\ref{not:chi'}.
\begin{proposition}\label{prop:untwistrep}
    Let $\g$ be $\W{-1}, W$ or $\Vir$ and let $\chi = \chi_1 + \dots + \chi_\ell$ be a local function on $\g$ following Notation~\ref{not:localfunc}. Let $\chi' = \chi'_1 + \dots + \chi'_\ell$ be defined as in Notation~\ref{not:chi'}. Then the canonical local representation $M'_\chi$ is isomorphic to $M_{\chi'}$.
\end{proposition}
\begin{proof}
    By Proposition~\ref{prop:tensor1point}, we have $M'_\chi \cong M'_{\chi_1} \otimes \dots \otimes M'_{\chi_\ell}$.
    A similar proof to Proposition~\ref{prop:tensor1point} for $M_{\chi'}$ yields $M_{\chi'} \cong M_{\chi'_1} \otimes \dots \otimes M_{\chi'_\ell}$. By Proposition~\ref{prop:untwistrephelp} $M'_{\chi_d} \cong M_{\chi'_d}$, we obtain the desired isomorphism. 
\end{proof}
With that, \cite{petukhov2022poisson}*{Proposition 7.1.5} can now be rephrased as the inverse of Theorem \ref{theo:simplicitymulti}.
\begin{corollary}\label{cor:simplicityconverse}
    Let $\g$ be $\W{-1}, W$ or $\Vir$ and let $\chi = \chi_1 + \dots + \chi_\ell$ be a local function on $\g$ following Notation~\ref{not:localfunc}. Then $M'_\chi$
    is an irreducible $\g$-module if and only if its component$\chi_d \neq \chi_{x_d; \alpha_0, \frac{1}{2}}$ for all $d$. 
\end{corollary}
\begin{proof}
We only need to show the converse of Theorem~\ref{theo:simplicitymulti}. If $\chi_d = \chi_{x_d; \alpha_0, \frac{1}{2}}$, then $M'_{\chi_d} \cong M_{\chi_{\alpha_0, 0}}$, and from ~\cite{petukhov2022poisson}*{Proposition 7.1.5}, this is not a simple $\Ua(\g)$-module. So $M_\chi$ is not an irreducible representation of $\g$. 
\end{proof}

We will also show that $\Ann_{\Ua(\g)} M'_\chi$ is primitive even if Theorem~\ref{theo:simplicitymulti} does not apply. This generalizes ~\cite{conley2007}*{Theorem 1.2.}, where it is shown (in our notation) that 
\begin{equation}\label{eq:primreplace}
    \Ann_{\Ua(\Vir)} M'_{\chi_{x; \alpha, \frac{1}{2}}} = \Ann_{\Ua(\Vir)} M'_{\chi_{x; \beta, -\frac{1}{2}}} \text{ for all } \alpha, \beta \in \kk, 
\end{equation}
so $\Ann_{\Ua(\Vir)} M'_{\chi_{x; \alpha,\frac{1}{2}}}$ is still a primitive ideal of $\Ua(\Vir)$. In particular, we show that all annihilators of local representations are primitive ideals.

\begin{proposition}\label{prop:primreplace}
    Let $\g = \W{-1}, W$ or $\Vir$ and let $\hat{\chi}$ be a local function on $\g$. Let $x \in \kk \backslash \supp(\hat{\chi})$,  $\alpha, \beta \in \kk$, define $\chi = \chi_{x; \alpha, \frac{1}{2}} + \hat{\chi}$ and $\eta = \chi_{x; \beta, -\frac{1}{2}} + \hat{\chi}$. Then $Q_\chi = Q_\eta$.
\end{proposition}

\begin{proof}
    Note that $Q_\chi = \ker \varrho_{\chi}$ where $\varrho_{\chi}: \Ua(\g) \to \End(M'_\chi)$ is given by the action of $\Ua(\g)$ on $M'_\chi$. By Proposition~\ref{prop:tensor1point}, $M'_\chi = M'_{\chi_{x; \alpha, \frac{1}{2}}} \otimes M'_{\hat{\chi}}$, we can write $\varrho_{\chi}$ as follows:
    \begin{align*}
        \varrho_{\chi}: \Ua(\g) \xrightarrow{\Delta} \Ua(\g) \otimes \Ua(\g) \xrightarrow{\varrho_{{\chi_{x; \alpha, \frac{1}{2}}}} \otimes \varrho_{{\hat{\chi}}}} \End(M'_{\chi_{x; \alpha, \frac{1}{2}}} \otimes M'_{\hat{\chi}}).
    \end{align*}
    Thus by ~\cite{bergman2016tensor}*{Theorem 1.}
    \begin{align*}
        Q_\chi &= \ker \varrho_{\chi} = \Delta^{-1}(\ker (\varrho_{{\chi_{x; \alpha, \frac{1}{2}}}} \otimes \varrho_{{\hat{\chi}}})) = \Delta^{-1}(\Ua(\g) \otimes \ker \varrho_{{\hat{\chi}}} + \ker \varrho_{{\chi_{x; \alpha, \frac{1}{2}}}} \otimes \Ua(\g)) \\ 
        &= \Delta^{-1}(\Ua(\g) \otimes Q_{\hat{\chi}}  +  Q_{\chi_{x; \alpha, \frac{1}{2}}} \otimes \Ua(\g)).
    \end{align*}
    and likewise for $M'_\eta$, we have:
    \begin{align*}
        Q_\eta = \ker \varrho_{\eta} = \Delta^{-1}(\Ua(\g) \otimes Q_{\hat{\chi}}  +  Q_{\chi_{x; \beta, -\frac{1}{2}}} \otimes \Ua(\g)).
    \end{align*}
    Since $Q_{\chi_{x; \alpha, \frac{1}{2}}}= Q_{\chi_{x; \beta, -\frac{1}{2}}}$ by \eqref{eq:primreplace}, thus $Q_\chi = Q_\eta$.
\end{proof}

Putting together Corollary \ref{cor:simplicityconverse} and Proposition \ref{prop:primreplace}, we obtain the following main theorem and corollary of the section. 
\begin{theorem}\label{theo: primitive ideals}
    Let $\g = \W{-1}, W$ or $\Vir$ and let $\chi$ be a local function on $\g$. Then $Q_\chi$ is a primitive ideal of $\Ua(\g)$.
    \qed
\end{theorem}

\begin{corollary}
    Let $\g$ be as above. If $\g = W$ or $\g = \W{-1}$, let $z = 0$. Then the following map is well-defined. 
    \begin{equation}\label{eq:Dixmier2}
        \Dx^\g: \g^* \to \Prim \Ua(\g); \quad 
        \chi \text{ local } \mapsto Q_\chi,  \quad 
        \chi \text{ not local } \mapsto (z-\chi(z)).
    \end{equation}
    \qed
\end{corollary}

Let $\g = \W{-1}, W$ or $\Vir$. So far, we have constructed a well-defined nontrivial map $\Dx^\g: \g^* \to \Prim \Ua(\g)$\label{pl:Dx}. Recall that functions on $\g$ also yield Poisson primitive ideals of $\Sa(\g)$ by taking the Poisson core of the corresponding maximal ideal of $\Sa(\g)$. In other words, we have the following diagram 
\[
\begin{tikzcd}
    \g^* \arrow[r] \arrow[dr] & \Pprim \Sa(W) \\
    & \Prim \Ua(W)
  \end{tikzcd} \text{ induced by }
  \begin{tikzcd}
    \chi \text{ local} \arrow[r, mapsto] \arrow[dr, mapsto] & P(\chi) \\
    & Q_\chi.
  \end{tikzcd}
\]
So if there is a strong Dixmier map, it must be of the following form for $\g = \W{-1}$ or $\g = W$.
    \begin{equation}\label{eq:dxWW-1}
        \widetilde{\Dx}^{\g} : \Pprim \Sa(\g) \to \Prim \Ua(\g): \quad  0\mapsto 0 , P(\chi) \neq 0 \mapsto Q_\chi.
    \end{equation}
Similarly, the candidate for the strong Dixmier map for $\Vir$ is as follows
    \begin{align}\label{eq:dxVir}
        \widetilde{\Dx}^{\Vir}: \Pprim \Sa(\Vir) &\longrightarrow \Prim \Ua(\Vir) \nonumber \\
        P(\chi) = (z-\chi(z)) &\longmapsto (z-\chi(z)), \nonumber \\
        P(\chi) \neq (z-\chi(z)) &\longmapsto Q_\chi = \Gamma^{-1}(Q_{\Gamma(\chi)}),
    \end{align}
where $\Gamma: \Ua(\Vir) \to \Ua(W)$ is the projection given by $z \mapsto 0$ and $\Gamma(\chi)$ is the corresponding element of $\chi$ in $W^*$ by abusing notation.

The remaining sections will prove that the maps $ \widetilde{\Dx}^{\W{-1}}, \widetilde{\Dx}^{W}$ and $\widetilde{\Dx}^{\Vir}$ are well-defined. 

\section{A master homomorphism for local functions}\label{sec: psi}
To prove Theorem~\ref{theo:Dixmier1}, we will relate the primitive ideal $Q_\chi = \Ann_{\Ua(W)} M'_\chi$ to a primitive ideal of the universal enveloping algebra of a finite-dimensional solvable Lie algebra. A key ingredient is proving the existence of the algebra homomorphisms $\Psi^W_n$ and $\Psi^{\W{-1}}_n$ defined in Theorem \ref{theo:algebrahom}. To that end, we first define some notation that will be used throughout the rest of the paper. We will introduce the rings of interest and then give the overview of the section.

\begin{notation}\label{def:Weylalg}
    For $m \geq 1$, let $A_m$ be the $m$-th Weyl algebra generated by $s_i, \del_i$ for $i = 1, \dots, n$ subject to the relation $[s_i, s_j] = [\del_i, \del_j] = 0$ and $[\del_i, s_j] = \delta_{ij}$. We will abuse notation and write $A_m = \kk [s_1, \dots, s_m, \del_1, \dots, \del_m]$.\label{pl:Weyl}
    Let $\widetilde{A}_\ell$ denote the localization $A_\ell[s_1^{-1}, \dots, s_{\ell}^{-1}]$ of $A_\ell$.  
    
    For $n = 1$, we sometimes write $A_1 = \kk[t, \del]$ and $\widetilde{A}_1 = \kk[t, t^{-1}, \del]$. 
\end{notation}

\begin{definition}\label{def:ringt}
    For all $n \in \NN$, we consider the finite-dimensional solvable Lie algebra $\g_n = \W{0}/\W{n}$\label{pl:gn} with basis $\{v_{0}, \dots, v_{n-1} \}$\label{pl:vi}, where $v_i = e_i + \W{n}$. For $n' > n$, we define the projection $\tau_{n}: \g_{n'} \to \g_{n}$ given by $v_{\geq n} \mapsto 0$. Note that $\W{-1}$ and $W$ are simple Lie algebras, however, their subquotient $\g_n$ is solvable for every $n \in \NN$. 
    
    Define the rings $T_n = \kk[t, \partial] \otimes_{\kk} \Ua(\g_n)$\label{pl:Tn} and $\widetilde{T}_n = \kk[t, t^{-1}, \partial] \otimes_{\kk} \Ua(\g_n)$\label{pl:Tn-1}. Let $T_\infty =  \kk[t, \partial] \otimes_{\kk} \Ua(\W{0})$ be the ``infinite version'' of $T_n$.
\end{definition}

In this section, we will show that the linear maps 
\begin{align}
    \Psi^{\W{-1}}_{n}: \W{-1} &\rightarrow T_n, \quad 
    f \partial \mapsto f(t) \partial_t + \sum_{i=0}^{n-1} \frac{f^{(i+1)}(t)}{(i+1)!}v_i, \quad \text{and }\label{eq:psiW-1n}\\
    \Psi^{W}_{n}: W &\rightarrow \widetilde{T}_n, \quad 
    f \partial \mapsto f(t) \partial_t + \sum_{i=0}^{n-1} \frac{f^{(i+1)}(t)}{(i+1)!}v_i\label{eq:psiWn},
\end{align}
are Lie algebra homomorphisms and thus induce algebra homomorphisms $\Psi^{\W{-1}}_{n}: \Ua(\W{-1}) \to T_n$ and $\Psi^{W}_{n}: \Ua(W) \to \widetilde{T}_n$. 

We recall the family of algebra homomorphisms $\pi_\gamma: \Ua(W) \to \widetilde{A}_1$ in ~\cite{conley2007} defined by  $\pi_\gamma(f\del) = f\del + \gamma f'$ for $\lambda \in \kk$. 
Since $\widetilde{T}_1 = \widetilde{A}_1 \otimes \kk[v_0]$, then as $v_0 \in Z(\widetilde{T}_1)$, the map 
\begin{equation}\label{eq: map psi 1}
    \Psi^W_{1}: \Ua(W) \to \widetilde{T}_1, \quad f\del \mapsto f\del + f'v_0
\end{equation}
is a ring homomorphism. Note that $\Psi^W_1$ appeared before in ~\cite{conley2007},~\cite{sierra2014universal}, and~\cite{sierra2017}. The homomorphism $\Psi^W_n$ generalizes $\Psi^W_1$ to $n > 1$. To our knowledge, the homomorphisms $\Psi^{\W{-1}}_n$ and $\Psi^{W}_n$ have not appeared before in the literature. 

We further show the striking result that there exists a ``master algebra homomorphism'' $\Psi^{\W{-1}}_{\infty}$ through which all $\Psi^{\W{-1}}_n$ factor. The linear map  
    \begin{align}\label{eq: psi infty}
        \Psi^{\W{-1}}_{\infty} : \W{-1} \rightarrow T_\infty = A_1 \otimes \Ua(\W{0}), \quad f \partial \mapsto f\partial + \sum_{i = 0}^\infty \frac{f^{(i+1)}}{(i+1)!}e_i, 
    \end{align}
induces a graded algebra embedding $\Psi^{\W{-1}}_{\infty}: \Ua(\W{-1}) \rightarrow T_\infty$. There is a similar version for $W$, however, for negative powers of $t$, the sum on the right-hand side of~\eqref{eq: psi infty} may be infinite, thus, one must work in the appropriate completion of the ring $T_\infty[t^{-1}]$.

In Subsection~\ref{subsec:phin}, we construct Poisson homomorphisms $\Phi^{\W{-1}}_\infty$, $\Phi^{\W{-1}}_n$, and $\Phi^W_n$, which we will see are the associated graded maps of $\Psi^{\W{-1}}_\infty$, $\Psi^{\W{-1}}_n$, and $\Psi^W_n$. As a result, we prove that $\Psi^{\W{-1}}_\infty$, $\Psi^{\W{-1}}_n$, and $\Psi^W_n$ are ring homomorphisms in Subsection~\ref{subsec:psin}. Lastly, in Subsection~\ref{subsec:psiinfty}, we complete appropriately to show the existence of the master homomorphism $\Psi^{W}_\infty$, and prove that $\Psi^{\W{-1}}_\infty$ and $\Psi^{W}_\infty$ are furthermore injective.

\subsection{The Poisson versions of our maps}\label{subsec:phin}

In this subsection, we construct the maps $\Phi^{\W{-1}}_\infty$, $\Phi^{\W{-1}}_n$, and $\Phi^W_n$, and prove that they are Poisson homomorphisms. To motivate our construction, let $\kk[t, t^{-1}, y_0]$ be a Poisson algebra via $\{ y_0, t \} = 1$. Recall from~\cite{petukhov2022poisson}*{Equation 2.2.2} that for $\gamma \in \kk$ the ring homomorphism 
\begin{align*}
    p_\gamma: \Sa(W) \to \kk[t, t^{-1}, y_0], \quad f \del \mapsto f y_0 + \gamma f',
\end{align*}
is a Poisson map. Note that $\ker p_\gamma = \cap_{x, \alpha_0 \in \kk} \mf{m}_{x; \alpha_0, \gamma}$ is the Poisson core of $\chi_{x; \alpha_0, \gamma}$ for any $x, \alpha_0, \gamma$ with $(\alpha_0, \gamma) \neq (0, 0)$. Furthermore, by ~\cite{petukhov2022poisson}*{Equation 2.2.2}, if we define a Poisson structure on $\kk[t, t^{-1}, y_0, y_1]$ where additionally, $y_1$ is Poisson central, then  
\begin{align*}
    \Phi_1^W : \Sa(W) \to \kk[t, t^{-1}, y_0, y_1], \quad  f\del \mapsto f y_0 + f' y_1,
\end{align*}
is also a Poisson homomorphism. The map $\Phi_1^W$ is clearly the universal map for one-point local functions $\chi = \chi_{x; \alpha_0, \alpha_1}$ of order at most 1, where $x \neq 0$, in the sense that 
\begin{equation}\label{eq:Poissonuniversal}
    \ker \Phi_1^W = \bigcap_{\gamma \in \kk} p_\gamma = \bigcap_{\substack{x \in \kk^\times \\ \alpha_0, \gamma \in \kk \\ (\alpha_0, \gamma) \neq (0, 0)}} P(\chi_{x; \alpha_0, \gamma}) = \bigcap_{\substack{x \in \kk^\times \\ \alpha_0, \gamma \in \kk \\ (\alpha_0, \gamma) \neq (0, 0)}} \mf{m}_{\chi_{x; \alpha_0, \gamma}}.
\end{equation}

Let us enlarge the codomain ring from $\kk[t, t^{-1}, y_0, y_1]$ to infinitely many variables.
\begin{definition}\label{def:poissonS}
    Let $\mathbb{S}_\infty = \kk[t, y_0, y_1, \dots, y_n, \dots]$ be a polynomial ring in countably many variables\label{pl:S}. We define a Poisson structure on $\mathbb{S}_\infty$ by 
    \begin{align}
        &\{y_0, t\} = 1, \quad \{y_i, y_0\} = \{y_i, t\} = 0 \text{ for all } i \neq 0 \label{eq:Poissont1}\\ 
        &\{y_i , y_j\} = \frac{1}{i!} (j-i) (j+1) (j+2) \dots (j+i-1) y_{i+j-1} \label{eq:Poissont2}.
    \end{align} 
    Note that for all $n$, $(y_{\geq n+1})$ is a Poisson ideal of $\mathbb{S}_\infty$. Thus, the quotient $\mathbb{S}_n = \mathbb{S}_\infty/(y_{\geq n+1})$ is a Poisson algebra as well.
\end{definition}

With the above motivation in mind, we consider the ring homomorphisms
\begin{align}
    \Phi_\infty^{\W{-1}} : \Sa(\W{-1}) \to \mathbb{S}_\infty, \quad  f\del &\mapsto \sum_{i=0}^{\infty} f^{(i)} y_i, \label{eq:phi-1inf}\\ 
    \Phi_n^{\W{-1}} : \Sa(\W{-1}) \to \mathbb{S}_n, \quad f\del &\mapsto \sum_{i=0}^{n} f^{(i)} y_i, \text{ and }\label{eq:phi-1n}\\
    \Phi_n^{W} : \Sa(W) \to \mathbb{S}_n[t^{-1}], \quad f\del &\mapsto \sum_{i=0}^{n} f^{(i)} y_i.
    \label{eq:phin}
\end{align}
Note that since $f \in \kk[t]$, $\Phi_\infty^{\W{-1}} (f\del)$ is a finite sum. The maps $\Phi_n^{\W{-1}}$ and $\Phi_n^{W}$ are the universal homomorphisms for one-point local functions of order $n$, whereas $\Phi_\infty^{\W{-1}}$ is the universal map for all one-point local functions in a similar sense to \eqref{eq:Poissonuniversal}. 
In the following theorem, we show that these maps are Poisson homomorphisms. These maps are the key ingredients to construct the homomorphisms $\Psi_\infty^{\W{-1}}$, $\Psi_n^{\W{-1}}$ and $\Psi_n^{W}$ later in Theorem~\ref{theo:algebrahom}.

\begin{theorem}\label{theo:Poissonfinite}
    If $\mathbb{S}_\infty$ and $\mathbb{S}_n$ are equipped with the Poisson structures from~Definition~\ref{def:poissonS}, then the maps $\Phi^{\W{-1}}_\infty$, $\Phi^{\W{-1}}_n$ and $\Phi^{W}_n$ defined in~\eqref{eq:phi-1inf}, \eqref{eq:phi-1n} and~\eqref{eq:phin} are Poisson homomorphisms. Furthermore, if we set $\deg e_n = n$, $\deg t = 1, \deg y_i = i-1$, then they are graded. 
\end{theorem}

\begin{proof}
    To treat the finite case and the infinite case the same way, we set $y_i = 0$ for $i \geq n+1$ in the finite case. The proof for $W$ is the same as $\W{-1}$, so we provide the proof for $\W{-1}$. By the product rule
    \begin{align*}
        (fg'-f'g)^{(n-1)} = \sum_{i=0}^n d_{i, n-i} f^{(i)} g^{(n-i)} \text{ for some } d_{i, n-i} \in \ZZ.
    \end{align*}
    Thus 
    \begin{align}
        &d_{i,j} = d_{i-1, j} + d_{i, j-1}, \text{  for } i, j \geq 1 \label{eq:sysofeq} \\ 
        &d_{0, j} = -d_{j, 0} = 1, \text{and } d_{0,0} = 0 \nonumber.
    \end{align}
    By induction, we can prove that $d_{i, j}=-d_{j, i}$, thus we restrict our attention to the case $j > i \geq 1$. We will show by induction on $i$ that the unique solution to the system of equations~\eqref{eq:sysofeq} is 
    \begin{equation}\label{eq: formula for d}
        d_{i, j} = \frac{1}{i!} (j-i) \prod_{k=1}^{i-1} (j+k). 
    \end{equation}
     
    For $i=1$, we can recur on $j$ to find that $d_{1, j} = j-1$. Suppose~\eqref{eq: formula for d} is true for $i = m-1$, we need to show it for $m$. We will now induct on $j \geq i$. The base case $j = i$ follows from $d_{i, i}=-d_{i, i}=0$. Assuming the induction hypothesis for $j= n-1$, we have:
    \begin{align*}
        d_{m, n} &= \frac{1}{(m-1)!} (n-m+1) (n+1) (n+2) \dots (n+m-2) + \frac{1}{m!} (n-m-1) n (n+1) \dots (n+m-2) \\ 
        &= \frac{1}{m!} (n+1) (n+2) \dots (n+m-2) \Bigl ( m(n-m+1) + (n-m-1) n \Bigr ) \\
        &= \frac{1}{m!} (n+1) (n+2) \dots (n+m-2) (n+m-1) (n-m).
    \end{align*}
    Thus,~\eqref{eq: formula for d} is true for all $i, j \in \ZZ$.

    Now, we show $\Phi^{\W{-1}}_{\infty}$ and $\Phi^{\W{-1}}_{n}$ are Poisson homomorphisms. It is sufficient to show that 
    \begin{equation}\label{eq: Poisson equal}
        \sum_{k \geq 0} (fg' - f'g)^{(k)} y_k = \big \{ \sum_{i \geq 0} f^{(i)} y_i,   \sum_{j \geq 0} g^{(j)} y_j \big \} \text{ for all } f, g \in \kk[t].
    \end{equation}
    The LHS of~\eqref{eq: Poisson equal} is
    \begin{align*}
        \sum_{k \geq0 } y_{k} (fg' - f'g)^{(k)} &= \sum_{k \geq 0} y_{k} \sum_{i=0}^{k+1} d_{i, k+1 - i}f^{(i)}g^{(k+1-i)} = \sum_{j\geq0} \sum_{i\geq0} y_{i+j-1} d_{i, j}f^{(i)}g^{(j)} . 
    \end{align*}
    From~\eqref{eq:Poissont1},~\eqref{eq:Poissont2} and~\eqref{eq: formula for d}
    \begin{align*}
        \{y_0 f, y_i g^{(i)}\} &= y_0 fg^{(i+1)}, \text{ for } i \geq 1 \text{ and }\\ 
        \{y_i f^{(i)}, y_j g^{(j)}\} &= f^{(i)}g^{(j)}\{y_i , y_j \} = d_{i, j} f^{(i)}g^{(j)} y_{i+j-1} \text{ for } i, j \geq 1.
    \end{align*}
    Thus, the RHS of~\eqref{eq: Poisson equal} is
    \begin{align*}
        &\{y_0 f, y_0g\} + \sum_{i \geq 1} (\{y_0 f, y_i g^{(i)}\} + \{y_i f^{(i)}, y_0 g \}) + \sum_{i \geq 1}\sum_{j \geq 1} \{y_i f^{(i)}, y_j g^{(j)}\}\\ 
        &= y_0 (fg'-f'g) + \sum_{i \geq 1} (fg^{(i+1)} - f^{(i+1)}g)y_i  + \sum_{i \geq 1}\sum_{j \geq 1} d_{i, j} f^{(i)}g^{(j)} y_{i+j-1} \\ 
        &= \sum_{i \geq 0}\sum_{j \geq 0} y_{i+j-1} d_{i, j} f^{(i)}g^{(j)}.
    \end{align*}
    Hence,~\eqref{eq: Poisson equal} follows and $\Phi^{\W{-1}}_\infty$ and $\Phi^{\W{-1}}_n$ are Poisson maps. Lastly, it is obvious that $\Phi^{\W{-1}}_\infty$ and $\Phi^{\W{-1}}_n$ respects the gradings defined on $\Sa(W)$, $\mathbb{S}_\infty$, and $\mathbb{S}_n$.
\end{proof}

\begin{lemma}\label{lem:Siso}
    The rings $\mathbb{S}_\infty$ and $\mathbb{S}_n$ are isomorphic to $\kk[t, y_0] \otimes \Sa(\W{0})$ and $\kk[t, y_0] \otimes \Sa(\g_n)$ respectively as graded Poisson algebras.
\end{lemma}
\begin{proof}
    Let $\varphi: \kk[t, y_0] \otimes \Sa(\W{0}) \to \mathbb{S}_\infty$ be a ring isomorphism induced by $\varphi(t) = t, \varphi(y_0) = y_0, \varphi(e_i) = (i+1)!y_{i+1}$ for $i \geq 0$. For $i, j \geq 0$, we have
    \begin{align*}
        \{\varphi(e_i), \varphi(e_j)\} &=  \{(i+1)!y_{i+1}, (j+1)!y_{j+1}\} = (i+1)!(j+1)! \{y_{i+1}, y_{j+1}\} \\
        &= (i+1)!(j+1)! \frac{1}{(i+1)!} (j-i)(j+2)\dots(j+i+1)y_{i+j+1} \\ 
        &= (j-i) (j+i+1)!y_{i+j+1} = (j-i) \varphi(e_{i+j}).
    \end{align*}
    Thus, $\varphi$ is a Poisson algebra isomorphism and it respects the grading, so we can identify $\mathbb{S}_\infty$ with $\kk[t, y_0] \otimes \Sa(\W{0})$. Furthermore, $\varphi((e_{\geq n +1})) = (y_{\geq n})$, thus $\mathbb{S}_n$ is isomorphic to $\kk[t, y_0] \otimes \Sa(\g_n)$ as a Poisson algebra.
\end{proof}

Identifying the codomain rings using the isomorphisms in Lemma \ref{lem:Siso}, the maps $\Phi^{\W{-1}}_{\infty}$, $\Phi^{\W{-1}}_{n}$ and $\Phi^{W}_{n}$ are defined by
\begin{align}
    \Phi^{\W{-1}}_{\infty} : \Sa(\W{-1}) \rightarrow  \kk[t, y_0] \otimes_{\kk} \Sa(\W{0}), \quad
    f\partial & \mapsto fy_0 + \sum_{i= 0}^{n-1} \frac{f^{(i+1)}}{(i+1)!} e_i, \label{eq:phi-1inf2}\\
    \Phi^{\W{-1}}_{n} : \Sa(\W{-1}) \rightarrow  \kk[t, y_0] \otimes_{\kk} \Sa(\g_n), \quad
    f\partial & \mapsto fy_0 + \sum_{i= 0}^{n-1} \frac{f^{(i+1)}}{(i+1)!} v_i, \text{ and } \label{eq:phi-1n2}\\
    \Phi^{W}_{n} : \Sa(W) \rightarrow  \kk[t, t^{-1}, y_0] \otimes_{\kk} \Sa(\g_n), \quad f\partial  & \mapsto fy_0 + \sum_{i= 0}^{n-1} \frac{f^{(i+1)}}{(i+1)!} v_i. \label{eq:phin2}
\end{align}
\begin{corollary}\label{cor:phiWn}
    The maps $\Phi^{\W{-1}}_{\infty}$, $\Phi^{\W{-1}}_{n}$ and $\Phi^{W}_{n}$ defined in \eqref{eq:phi-1inf2}, \eqref{eq:phi-1n2} and \eqref{eq:phin2} are graded Poisson homomorphisms. 
    \qed
\end{corollary}

\subsection{The ring homomorphisms \texorpdfstring{$\Psi^{\W{-1}}_{\infty}$, $\Psi^{\W{-1}}_{n}$ and $\Psi^{W}_{n}$}{Psi n}}\label{subsec:psin}
In this subsection, we quantize the Poisson homomorphisms $\Phi^{\W{-1}}_{\infty}$, $\Phi^{\W{-1}}_{n}$ and $\Phi^{W}_{n}$ to obtain the algebra homomorphisms $\Psi^{\W{-1}}_{\infty}$, $\Psi^{\W{-1}}_{n}$ and $\Psi^{W}_{n}$. Furthermore, we show that by taking the appropriate tensor product of these maps, we can construct more homomorphisms that are the multi-point versions of $\Psi^{\W{-1}}_{n}$ and $\Psi^{W}_{n}$.

\begin{theorem}\label{theo:psin}
    There are Lie algebra homomorphisms 
    \begin{align}
        \Psi^{\W{-1}}_{\infty} : \W{-1} &\rightarrow {T}_\infty = A_1\otimes \Ua(\W{0}), &\quad 
        f \partial &\mapsto f(t)\partial + \sum_{i = 0}^{\infty} \frac{f^{(i+1)}(t)}{(i+1)!}e_i, \label{eq:psiW-1infty}\\
        \Psi^{\W{-1}}_{n} : \W{-1} &\rightarrow {T}_n = A_1\otimes \Ua(\g_n), &\quad 
        f \partial &\mapsto f(t)\partial + \sum_{i = 0}^{n-1} \frac{f^{(i+1)}(t)}{(i+1)!}v_i, \text{ and } \label{eq:psinW-1}\\
        \Psi^{W}_{n} : W &\rightarrow \widetilde{T}_n = \tilde{A}_1 \otimes \Ua(\g_n), &\quad 
        f \partial &\mapsto f(t)\partial + \sum_{i = 0}^{n-1} \frac{f^{(i+1)}(t)}{(i+1)!}v_i,\label{eq:psinW}, 
    \end{align}
    where the latter two are defined for all $n \geq 0$. These maps induce graded algebra homomorphisms $\Psi^{\W{-1}}_{\infty} : \Ua(\W{-1}) \to {T}_\infty$, $\Psi^{\W{-1}}_{n} : \Ua(\W{-1}) \to {T}_n$ and $\Psi^{W}_{n} : \Ua(W) \to \widetilde{T}_n$.
\end{theorem}
\begin{proof}
    Note that since $f \in \kk[t]$, $\Psi^{\W{-1}}_{\infty}(f\del)$ is a fintie sum. 
    The proof for $W$ is similar to that for $\W{-1}$. Thus we only provide the details for $\W{-1}$. To treat the finite and infinite case the same, let $\g$ be either $\W{0}$ or $\g_n$ in this proof. 

    Let $\mathcal{S} = \kk[t, t^{-1}, y_0] \otimes_{\kk} \Sa(\g)$. We put the ``order'' grading on $\mathcal{S}$ by setting $\deg t = 0, \deg y_0 = \deg v =1$ for $v \in \g$. Now, we filter $\mathcal{T} = A_1 \otimes \Ua(\g)$ by the increasing ``order'' filtration $\mathcal{F}_\bullet$ by setting $\deg t = 0$ and $\deg \partial = \deg v = 1$ for $v \in \g$. Thus, $\mathcal{F}_0 \mathcal{T} = \kk[t],  \mathcal{F}_1\mathcal{T} = \kk[t] \oplus \kk[t]\del \oplus \kk[t]\otimes \g, \dots$ as $\kk[t]$-modules. Then the associated graded of $\mathcal{T}$ with respect to this filtration is $\mathcal{S}$, and the associated graded construction induces a Poisson bracket on $\mathcal{S}$ with $\{y_0, t \} =1$ and the Kostant-Kirillov Poisson bracket on $\Sa(\g)$. This coincides with the Poisson structure on $\mathcal{S}$ that makes $\Phi^{\W{-1}}_n$ and $\Phi^{\W{-1}}_\infty$ Poisson homomorphisms in Corollary~\ref{cor:phiWn}.
    
    Let $\mathcal{F}'_1 \mathcal{T} = \kk[t]\del \oplus \kk[t]\otimes \g$, which is a distinguished complement of $\mathcal{F}_0 \mathcal{T} \subset \mathcal{F}_1 \mathcal{T}$. Note that $\mathcal{F}'_1 \mathcal{T}$ is closed under commutators. The associated graded map is an isomorphism of $\kk[t]$-modules $\gr:\mathcal{F}'_1 \mathcal{T} \to \mathcal{S}_1$, and for $a, b \in \mathcal{F}'_1 \mathcal{T}$:
    \begin{equation}
        \gr(ab-ba) = \{\gr(a), \gr(b)\},
    \end{equation} 
    i.e., the commutator bracket on $\mathcal{F}'_1 \mathcal{T}$ agrees with the Poisson bracket on $\mathcal{S}_1$. 

    Thus, we have a following commutative diagram of $\kk[t]$-module maps. 
    \[\begin{tikzcd}
        W \arrow[rd, "\Phi^{\W{-1}}_\infty", swap] \arrow[r, "\Psi^{\W{-1}}_\infty"] & \mathcal{F}'_1 \mathcal{T} \arrow[d, "\gr"] \\
        & \mathcal{S}_1
    \end{tikzcd}\]
    From Corollary~\ref{cor:phiWn}, $\Phi^{\W{-1}}_n$ and $\Phi^{\W{-1}}_\infty$ are Lie algebra homomorphisms, and $\gr$ is a Lie algebra isomorphism. Thus $\Psi^{\W{-1}}_\infty = \gr^{-1} \circ \Phi^{\W{-1}}_\infty$ and $\Psi^{\W{-1}}_n = \gr^{-1} \circ \Phi^{\W{-1}}_n$ are Lie algebra homomorphisms on $\W{-1}$. Thus, they induce algebra homomorphisms $\Ua(\W{-1}) \to {T}_\infty$ and $\Ua(\W{-1}) \to {T}_\infty$. If we set $\deg e_n = n$, $\deg t = 1, \deg v_i = i$, then $\Psi^{\W{-1}}_\infty$ and $\Psi^{\W{-1}}_n$ are graded.
\end{proof}

Notice that $\Psi^{\W{-1}}_n$ factors through $\Psi^{\W{-1}}_\infty$ via the obvious homomorphism $A_1 \otimes \Ua(\W{0}) \to A_1 \otimes \Ua(\g_n)$. In that
sense, $\Psi^{\W{-1}}_\infty$ is a universal homomorphism for all $\Psi^{\W{-1}}_n$. 

\begin{definition}\label{def:gboldn}
    Let $\n = (n_1, \dots, n_\ell) \in \ZZ^{\ell}_{\geq 1}$. We define
    \begin{align}
        \g_\n = \g_{n_1} \oplus \dots \oplus \g_{n_\ell}.
    \end{align}
    Similarly to Definition~\ref{def:ringt}, we define $T_\n = A_\ell \otimes \Ua(\g_\n)$ and $\widetilde{T}_\n = \tilde{A}_\ell \otimes \Ua(\g_\n)$. Note that $T_\n = T_{n_1} \otimes \dots \otimes T_{n_\ell}$ and $\widetilde{T}_\n = \widetilde{T}_{n_1} \otimes \dots \otimes \widetilde{T}_{n_\ell}$.
\end{definition}

\begin{proposition}\label{prop: psi multi}
    Let $\n = (n_1, \dots, n_\ell) \in \ZZ^{\ell}_{\geq 1}$. Then the linear maps
    \begin{align*}
        \Psi^{\W{-1}}_{\n}: \W{-1} &\rightarrow T_\n, \quad  
        f \del \mapsto \Psi^{\W{-1}}_{n_1} (f \del) \otimes 1 \otimes \dots \otimes 1 + \dots + 1 \otimes 1 \otimes \dots \otimes \Psi^{\W{-1}}_{n_\ell} (f\del),\\
        \Psi^{W}_{\n}: W &\rightarrow \widetilde{T}_\n, \quad  
        f \del \mapsto \Psi^{W}_{n_1} (f \del) \otimes 1 \otimes \dots \otimes 1 + \dots + 1 \otimes 1 \otimes \dots \otimes \Psi^{W}_{n_\ell} (f\del),
    \end{align*}
    induce graded algebra homomorphisms $\Psi^{\W{-1}}_{\n}: \Ua(\W{-1})\rightarrow T_\n$ and $\Psi^W_{\n}: \Ua(W)\rightarrow \widetilde{T}_\n$.
\end{proposition}
\begin{proof}
    We prove only the second statement. Since $\Ua(W)$ is a Hopf algebra, the coproduct $\Delta^{\ell}: \Ua(W) \to \Ua(W)^{\otimes \ell}$ is an algebra homomorphism. Thus $\Psi_{\n} = (\Psi_{n_1} \otimes \dots \otimes \Psi_{n_\ell}) \circ \Delta^{\ell}$ is an algebra homomorphism. 
\end{proof}

\subsection{Master homomorphisms $\Psi^{\W{-1}}_{\infty}$ and $\Psi^{W}_{\infty}$}\label{subsec:psiinfty}
We have, so far, constructed the master homomorphism $\Psi^{\W{-1}}_{\infty}$ which can be seen as the universal homomorphism for all $\Psi^{\W{-1}}_{n}$. It is natural to ask if such a master homomorphism $\Psi^{W}_{\infty}$ exists through which all $\Psi^{W}_{n}$ factor. In this subsection, we show such a map exists, and moreover, prove that $\Psi^{\W{-1}}_{\infty}$ and $\Psi^{W}_{\infty}$ are injective. These facts, together with Corollary \ref{cor:psiuniversal}, imply that $\Psi^{\W{-1}}_{\infty}$ and $\Psi^{W}_{\infty}$ satisfy another universal property that they are master homomorphisms for all one-point local functions, which we will explain later in Subsection \ref{subsec:annpsi}. For $W$, for negative powers of $t$, the sum on the right-hand side of~\eqref{eq:psiW-1infty} may be infinite, thus, one must work in the appropriate completion of the ring $\widetilde{T}_n$. 

\begin{definition}
    For $n' > n$, recall the projection $\tau_{n}: \g_{n'} \to \g_n$ defined in Definition~\ref{def:ringt}. We have the inverse system of rings $(\widetilde{T}_n)_{n \in \ZZ_{\geq 1}}$ with ring projections $\id \otimes \tau_{n}: \widetilde{T}_{n'} \to \widetilde{T}_{n}$. Thus, we can form the inverse limit 
    \begin{align*}
        \widehat{T}_\infty = \varprojlim \widetilde{T}_n.
    \end{align*}
\end{definition}

\begin{proposition}
    The map $\Psi^{W}_{\infty}$
    \begin{align}\label{eq:psiinfty}
    \Psi^{W}_{\infty} : \Ua(W) \rightarrow \widehat{T}_\infty, \quad f \partial \mapsto f\partial + \sum_{i = 0}^\infty \frac{f^{(i+1)}}{(i+1)!}v_i, 
\end{align}
    is an algebra homomorphism. 
\end{proposition}
\begin{proof}
    We have the family of homomorphisms $(\Psi^W_n)_{n \in \ZZ_{\geq 1}}$ such that $\Psi^W_n = (\id \otimes \tau_{n}) \circ  \Psi^W_{n'}$. Thus, by the universal property of $\widehat{T}_\infty$, there exists an algebra homomorphism $\Psi^{W}_{\infty} : \Ua(W) \rightarrow \widehat{T}_\infty$. It is easy to see that $\Psi^{W}_{\infty}$ has the form \eqref{eq:psiinfty}.
\end{proof}

\begin{lemma}\label{cor: psi infty}
    The Poisson homomorphism $\Phi^{\W{-1}}_{\infty}$ defined in \eqref{eq:phi-1inf2} is injective.
\end{lemma}

\begin{proof}
    Consider the ring homomorphism
    \begin{align*}\label{pl:kappa}
        \kappa^{\W{-1}} : \kk[t, y_0] \otimes \Sa(W_{\geq 0}) \to \Sa(W_{\geq 0}), \quad t \mapsto 1, y_0 \mapsto 0, \restr{\kappa^{\W{-1}}}{\Sa(W_{\geq 0})} = \id.
    \end{align*}
    Note that $\kappa^{\W{-1}}$ is only an algebra homomorphism and is not a Poisson homomorphism. We claim that $\kappa^{\W{-1}} \circ \Phi^{\W{-1}}_{\infty}$ is surjective onto $\Sa(\W{0})$ and it suffices to show that $\W{0} \subseteq \im (\kappa^{\W{-1}} \circ \Phi^{\W{-1}}_{\infty})$. For any $w \in W_{\geq 0}$, we write $w = \sum_{i=0}^n w_i e_i$ for some $n \in \NN$ and $w_i \in \kk$. Let
    \begin{align*}
        f = \sum_{i=0}^{n} w_i (t-1)^{i+1} \in \kk[t],
    \end{align*} 
    then we note that $\frac{f^{(i+1)}(1)}{(i+1)!} = w_i$ for $i=0, \dots, n$. Thus 
    \begin{align*}
        \kappa^{\W{-1}} \circ \Phi^{\W{-1}}_{\infty}(f\partial) = \sum_{i=0}^n \frac{f^{(i+1)}(1)}{(i+1)!} e_i = \sum_{i=0}^n w_i e_i = w.
    \end{align*}
    So $\Sa(W_{\geq 0}) \subseteq \im (\kappa^{\W{-1}} \circ \Phi^{\W{-1}}_{\infty})$. Since $\Sa(W_{\geq 0})$ has infinite GK dimension, $\im \Phi^{\W{-1}}_{\infty}$ has infinite GK dimension.

    If $\ker \Phi^{\W{-1}}_{\infty} \neq 0$, then by ~\cite{iyudu2020enveloping}*{Theorem 1.3} $\GK (\im \Phi^{\W{-1}}_{\infty}) < \infty$. Thus, $\ker \Phi^{\W{-1}}_{\infty} = 0$.
\end{proof}

\begin{proposition}
    $\Psi^{\W{-1}}_{\infty}$ and $\Psi^{W}_{\infty}$ are injective.
\end{proposition}
\begin{proof}
    We first show that $\Psi^{\W{-1}}_{\infty}$ is  injective. Consider the order filtration on $\Ua(\W{-1})$, where we define $\deg x = 1$ for $x \in \W{-1}$. Then $\Sa(\W{-1})$ is the associated graded of $\Ua(\W{-1})$ with respect to the above filtration. From the proof of Theorem~\ref{theo:psin}, we have the following commutative diagram.
    \[\begin{tikzcd}
        \Ua(\W{-1}) \arrow[d, "\gr" '] \arrow[r, "\Psi^{\W{-1}}_{\infty}"] & T_\infty \arrow[d, "\gr"] \\
        \Sa(\W{-1}) \arrow[r, "\Phi^{\W{-1}}_{\infty}"] & \mathcal{S} 
    \end{tikzcd}\]

    Let $z \neq 0 \in \Ua(\W{-1})$. Since the filtration is exhaustive, $\bigcup_{k} \mathcal{F}_k = \Ua(\W{-1})$, we can take the minimal $k$ such that $z \in \mathcal{F}_k$. Then $\gr(z) \in \Sa(\W{-1})_k \backslash \{0\}$, so $\Phi^{\W{-1}}_{\infty} (\gr(z)) \neq 0$ by Theorem~\ref{cor: psi infty}. Hence $\Psi^{\W{-1}}_{\infty}(z) \neq 0$ and $\Psi^{\W{-1}}_{\infty}$ is injective. Thus $\Psi^{\W{-1}}_{\infty}(\Ua(\W{-1}))$ has infinite GK-dimension. 

    We have the following commutative diagram 
    \[\begin{tikzcd}
        \Ua(\W{-1}) \arrow[d, hookrightarrow] \arrow[r, "\Psi^{\W{-1}}_{\infty}"] & T_\infty \arrow[d, hookrightarrow] \\
        \Ua(W) \arrow[r, "\Psi^{W{-1}}_{\infty}"] & \widehat{T}_\infty
    \end{tikzcd}\]
    Thus $\im \Psi^{W}_{\infty}$ also has infinite GK-dimension. If $\ker \Psi^{W}_{\infty} \neq 0$, then by ~\cite{iyudu2020enveloping}*{Theorem 1.3} $\GK \im \Psi^{W}_{\infty} < \infty$, which yields a contradiction. Hence $\Psi^{W}_{\infty}$ is injective.
\end{proof}

The maps $\Psi^{\W{-1}}_{\n}$ and $\Psi^{W}_{\n}$ are important for our orbit method and as the same time, they are interesting on their own. We conjecture that any map from $\Ua(\W{-1}) $ to some Weyl algebra always factor through some homomorphism $\Psi^{\W{-1}}_{\n}$.  
\begin{conjecture}\label{conj:mapsweylalgfactor}
    Let $\Psi_n$ denote the ring homomorphism $\Psi^{\W{-1}}_n$. For any ring homomorphism $\varphi: \Ua(\W{-1}) \to A_k$, there exists $\ell < k$ and $\n = (n_1, \dots, n_\ell) \in \NN^\ell$ such that $\varphi$ factors through $\Psi_{\n} = (\Psi_{n_1} \otimes \dots \otimes \Psi_{n_\ell}) \circ \Delta^{\ell}$ in the sense that there exists a homomorphism $\overline{\varphi}$ such that the following diagram commutes.
    \[\begin{tikzcd}
        \Ua(\W{-1}) \arrow[rd, "\varphi", swap] \arrow[r, "\Delta^\ell"]  &\Ua(\W{-1})^{\otimes \ell} \arrow[r, "\Psi_{n_1} \otimes \dots \otimes \Psi_{n_\ell} "] & [4em] {A}_\ell  \otimes {\Ua(\g_\n)} \arrow[ld, "\bar{\varphi}"] \\
        & A_k &
    \end{tikzcd}\]

    Also, this holds for $W$ where we replace $A_k$, $A_\ell$ and $\Psi^{\W{-1}}_{n_i}$ by $\tilde{A}_k$, $\tilde{A}_\ell$ and $\Psi^{W}_{n_i}$ respectively.
\end{conjecture}

\section{Annihilators of local representations}\label{sec:annihilators}
As before, by Propositions~\ref{prop:repWW-1} and~\ref{prop:repWVir}, we are only concerned with canonical local representation of $W$. We also denote $\Psi^W_\n$ by $\Psi_\n$ unless otherwise specified. In subsection \ref{subsec:annpsi}, using the map $\Psi_\n$, we establish a connection between the annihilators of local representations of $W$ and primitive ideals of $\Ua(\g_\n)$, which sets up the bridge to ``pullback the orbit method'' from $\g_\n$ to $W$. For $\n, \widehat{\n} \in \NN^\ell$, we define a partial order
\begin{align}\label{eq:biggertuple}
    \n \leq \widehat{\n} \Leftrightarrow n_i \leq \widehat{n}_i \text{ for all } i = 1, \dots, \ell.
\end{align}
We furthermore show that $\Psi_{\widehat{\n}}$ is the ``universal map'' for local functions of order $\leq \widehat{\n}$ in the sense that $\ker \Psi_{\widehat{\n}} = \bigcap_{\substack{\chi: \text{ local function} \\ \text{ of order } \leq \widehat{\n}}} Q_\chi$. In subsection \ref{subsec:even}, we prove a stronger result that for local functions $\chi$ of positive totally even order $\n$, i.e., $n_i = 2m_i >0$ for all $i$, then moreover $\ker \Psi_{\n} = Q_\chi$. 

\subsection{Annihilators of canonical local representations via the map $\Psi_\n$}\label{subsec:annpsi}
Recall that $\Loc^{\leq {\widehat{\n}}}$ is the set of local functions on $W$ of order at most $\widehat{\n}$. In this subsection, we will construct a projection
\begin{equation*}
    \pi^{{\n}}: \Loc^{\leq {\widehat{\n}}} \to \g^*_{{\n}}, \quad \chi \mapsto \widehat{\chi},
\end{equation*}
where $\widehat{\chi}$ is defined in Definition \ref{not:barchi}. We then show the main equation that for a local function $\chi$ of order $\n$, then 
\begin{equation}\label{eq:Qchibarchi}
    Q_\chi = (\Psi_{\widehat{\n}})^{-1} (\widetilde{A}_\ell \otimes Q_{\widehat{\chi}}), 
\end{equation}
where $Q_{\widehat{\chi}}$ is the primitive ideal of $\Ua(\g_\n)$ defined in Definition \ref{def:orbitrep}. We further show $\Psi_{\widehat{\n}}$ is the ``universal map'' for local functions of order $\leq \widehat{\n}$ in the sense that 
\begin{align*}
    \ker \Psi_{\widehat{\n}} = \bigcap_{\substack{\chi: \text{ local function} \\ \text{ of order }\n \leq \widehat{\n}}} Q_\chi.
\end{align*}

We first consider the case of one-point local function $\chi = \chi_{x; \alpha_0, \dots, \alpha_n}$ of order $n$ on $W$. Recall the definition of $\chi'$ in Notation~\ref{not:chi'}. Fix $\widehat{n} \geq n$. We will show that there exists a simple module $L_{\chi'}$ over $\widetilde{T}_{\widehat{n}}$ such that the canonical local representation $M'_\chi \cong \res{\Psi_{\widehat{n}}} L_{\chi'}$ as representations of $W$. The relationship between two annihilators is thus
\begin{align*}
    Q_\chi = \Ann_{\Ua(W)} M'_\chi = \Psi_{\widehat{n}}^{-1}(\Ann_{\widetilde{T}_{\widehat{n}}} L_{\chi'}).
\end{align*}

\begin{definition}\label{def:Ichi}
    Let $\chi = \chi_{x; \alpha_0, \dots, \alpha_n}$ be a one-point local function on $W$ of order $n$ and let $m = \left\lfloor \frac{n}{2}\right\rfloor$. Fix $\hat{n} \geq n$. Recall the canonical polarization $\p_\chi$ defined in Proposition~\ref{prop:canpol}, the basis $u_j$ of $W/\p_\chi$ defined in Definition~\eqref{eq:tildeudef} and $v_j = e_j + \W{\hat{n}}$. Recall the definition of $\chi'$ in Notation~\ref{not:chi'}. Motivated by the one-dimensional twisted representation $\kk'_\chi$ of $\p_\chi$, where $f \del \cdot 1_\chi = \chi(f\del) + \rho_{\p_\chi, W} (f\del) =  \chi'(f\del)1_\chi$, we define a left ideal of $\widetilde{T}_{\widehat{n}}$ by 
    \begin{align}
        I_{\chi'} = \begin{cases}
            \widetilde{T}_{\widehat{n}}(t-x, v_{m} - \chi'(u_m), \dots, v_{\widehat{n}-1} - \chi'(u_{\widehat{n}-1})) &\text{ if } m > 0, \\ 
            \widetilde{T}_{\widehat{n}}(t-x, v_{0} - \chi'(u_0) - 1, \dots, v_{\widehat{n}-1} - \chi'(u_{\widehat{n}-1})) &\text{ if } m=0.
        \end{cases}
    \end{align}
    Let $L_{\chi'} = \widetilde{T}_{\widehat{n}}/I_{\chi'}$. The case $m =0$ is special as $\Psi_n(u_{0}) = (t-x)\del+ v_0 = \del(t-x) - 1 + v_0$. 
\end{definition}

We first show that $\res{\Psi_{\widehat{n}}} L_{\chi'} \cong M'_{\chi}$ as representations of $W$. 
\begin{proposition}\label{prop:isoL}
    There exists an isomorphism of $\Ua(W)$-modules
\begin{align*}
    \theta: M'_{\chi} \to \res{\Psi_{\widehat{n}}} L_{\chi'}, \text{ such that } \theta(1_\chi) = 1+ I_{\chi'}.
\end{align*}
\end{proposition}
\begin{proof}
    By Proposition~\ref{prop:untwistrep}, we have $M'_\chi \cong M_{\chi'}$. By Lemma~\ref{lem:rewriterep}, we can write $M_{\chi'}$ as $\Ua(W)/J_{\chi'}$, where $J_{\chi'}$ is the left ideal generated by 
    \begin{equation}
        u_{m} - \chi'(u_m), \dots, u_{2m} - \chi'(u_{2m}), [\p_\chi, \p_\chi].
    \end{equation}
    From the algebra homomorphism $\Psi_{\widehat{n}}: \Ua(W) \to \widetilde{T}_{\widehat{n}}$, we have a $\Ua(W)$-module homomorphism
    \begin{align*}
        \widetilde{\theta}: \Ua(W) \rightarrow L_{\chi'}, \quad p \mapsto \Psi_{\widehat{n}}(p) + I_{\chi'}.
    \end{align*}
    We claim that $\widetilde{\theta}$ factors through the canonical map $\Ua(W) \to M'_\chi$ to give a well-defined homomorphism $\theta: M'_\chi \to L_{\chi'}$ such that $\theta(p+J_{\chi'}) = \Psi_{\widehat{n}}(p) + I_{\chi'}$. 
    
    We must check that $J_{\chi'} \leq \ker \widetilde{\theta}$ and it suffices to check on generators. Assume $m > 0$ and note that $n < 2m+2$, we have for any $q \in \kk[t, t^{-1}]$
    \begin{align*}
        \widetilde{\theta}((t-x)^{2m+2} q\del) &= \Psi_{\widehat{n}} ((t-x)^{2m+2} q\del) + I_{\chi'} = (t-x)^{2m+2} q \del + \sum_{i=0}^{\widehat{n}-1} \frac{\bigl((t-x)^{2m+2} q\bigr)^{(i+1)}}{(i+1)!}v_i + I_{\chi'}.
    \end{align*}
    This is 0 as $\widetilde{T}_{\widehat{n}}(t-x) \subseteq I_{\chi'}$ and $v_{\geq n} \in I_{\chi'}$. Moreover for $j = 0,\dots, m$
    \begin{align*}
        \widetilde{\theta}((t-x)^{m+1+j} \del) &= \Psi_{\widehat{n}} ((t-x)^{m+1+j} \del) + I_{\chi'} =(t-x)^{m+1+j}  \del + \sum_{i=0}^{m+j} \frac{\bigl((t-x)^{m+1+j}\bigr)^{(i+1)}}{(i+1)!}v_{i} + I_{\chi'}  \\ 
        &= v_{m+j} + I_{\chi'} \text{ as } \widetilde{T}_{\widehat{n}}(t-x) \subseteq I_{\chi'}.
    \end{align*}
    On the other hand if $m=0$, then $\widetilde{\theta}((t-x)\del) = (t-x)\del + v_0 + I_{\chi'}= v_0 - 1 + I_{\chi'}$. Hence $\widetilde{\theta}(J_{\chi'}) = 0$ in all cases, and $\theta$ is a well-defined $\Ua(W)$-homomorphism. We will now show that $\theta$ is an isomorphism by showing that 
    \begin{equation}\label{eq: for surjectivity and injectivity}
        \theta(u_{-1}^{k_{-1}} u_0^{k_0} \dots u_{m-1}^{k_{m-1}} 1_\chi) = \del^{k_{-1}}(v_{0}-1)^{k_0}\dots v_{m-1}^{k_{m-1}} + I_{\chi'} \text{ for all } k_{-1}, k_0, \dots, k_{m-1} \in \NN. 
    \end{equation}
    We proceed by downward induction on the lowest variable $u_j$ appearing in the word. The base case $j = -\infty$ (the word is empty) is trivial. Assuming the induction hypothesis, we break the induction step into 3 cases. For $j > 0$:
    \begin{align*}
        \theta(u_{j}^{k_j+1}\dots u_{m-1}^{k_{m-1}}  1_\chi) &= \Psi_{\widehat{n}}(u_{j}) \theta(u_{j}^{k_j}\dots u_{m-1}^{k_{m-1}}  1_\chi) = \Psi_{\widehat{n}}(u_{j}) v_{j}^{k_j}\dots v_{m-1}^{k_{m-1}} + I_{\chi'} \\ 
        &= \left ((t-x)^{j+1}  \del + \sum_{i=0}^{j-1} \frac{((t-x)^{j+1})^{(i+1)}}{(i+1)!}v_{i} + v_{j} \right )v_{j}^{k_j}\dots v_{m}^{k_{m}} + I_{\chi'} 
    \end{align*}
    and since $(t-x)$ commutes with all $v_i$ for $i \geq 0$, every term except the last one in the first factor is in $\widetilde{T}_{\widehat{n}}(t-x) \subseteq I_{\chi'}$. Then 
    \begin{align*}
        \theta(u_{j}^{k_j+1}\dots u_{m-1}^{k_{m-1}} 1_\chi) = v_{j}^{k_j+1}\dots v_{m-1}^{k_{m-1}} + I_{\chi'}.
    \end{align*}
    On the other hand, if $j=0$ then  
    \begin{align*}
        \theta(u_{0}^{k_0+1}\dots u_{m-1}^{k_{m-1}}1_\chi) &= \Psi_{\widehat{n}}(u_{0}) \theta(u_{0}^{k_0}\dots u_{m-1}^{k_{m-1}} 1_\chi) = \Psi_{\widehat{n}}(u_{0}) (v_0 -1)^{k_0} v_1^{k_{1}} \dots v_{m}^{k_{m}} + I_{\chi'}  \\ 
        &= \left ((t-x)\del + v_0 \right )(v_0 -1)^{k_0} v_1^{k_1} \dots v_{m-1}^{k_{m-1}} + I_{\chi'}\\
        &= (v_0 -1)^{k_0+1} v_1^{k_1} \dots v_{m-1}^{k_{m-1}} + I_{\chi'}.
    \end{align*}
    Lastly if $j=-1$, we have: 
    \begin{align*}
        \theta(u_{-1}^{k_{-1}+1}\dots u_{m-1}^{k_{m-1}}1_\chi) = \Psi_{\widehat{n}}(u_{-1}) \theta(u_{-1}^{k_{-1}}\dots u_{m-1}^{k_{m-1}}1_\chi) = \del \del^{k_{-1}} (v_0 -1)^{k_0} v_1^{k_1} \dots v_{m-1}^{k_{m-1}} + I_{\chi'}.
    \end{align*}    
    Thus, we have established~\eqref{eq: for surjectivity and injectivity}. As $\{ \del^{k_{-1}}(v_{0}-1)^{k_0}\dots v_{m-1}^{k_{m-1}} + I_{\chi'}\}$ are linearly independent, $\theta$ is injective. Since $\{\del^{k_{-1}}(v_{0}-1)^{k_0}\dots v_{m-1}^{k_{m-1}} + I_{\chi'}\}$ spans $L_{\chi'}$, $\theta$ is surjective. Hence $\theta$ is an isomorphism of $\Ua(W)$-modules. 
\end{proof}

\begin{corollary}\label{cor:relatingprimtive}
    $L_{\chi'}$ is a simple $\widetilde{T}_{\widehat{n}}$-module and $\Ann_{\widetilde{T}_{\widehat{n}}}L_{\chi'}$ is a primitive ideal. Moreover, 
    \begin{align*}
        Q_\chi = \Ann_{\Ua(W)} M'_\chi = (\Psi_{\widehat{n}})^{-1}(\Ann_{\widetilde{T}_{\widehat{n}}}L_{\chi'}).
    \end{align*}
\end{corollary}
\begin{proof}
     If $\chi \neq \chi_{x; \alpha_0, \frac{1}{2}}$, then by Theorem \ref{theo:simplicityone}, $M'_\chi$ is irreducible representation of $W$, thus $L_{\chi'}$ is a simple $\widetilde{T}_{\widehat{n}}$-module. Otherwise, if $\chi = \chi_{x; \alpha_0, \frac{1}{2}}$, then it can be easily seen that $L_{\chi'}$ is an irreducible representation of $\widetilde{T}_{\widehat{n}}$ from its generators.
\end{proof}

Thus, our main focus now is to analyze $\Ann_{\widetilde{T}_{\widehat{n}}}L_{\chi'}$. To that end, we define an element $\widehat{\chi}$ of $\g^*_{n}$ ``corresponding to'' $\chi$ as follows. 

\begin{definition}\label{not:barchi}
    For a one-point local function $\chi = \chi_{x; \alpha_0, \dots, \alpha_n}$ on $W$, recall the element $\chi' \in W^*$ defined in Notation~\ref{not:chi'}. Let $\bar{\chi}$ be the corresponding element of $\chi$ in $\g^*_n$ that is $\bar{\chi} (v_i) = \chi'(u_i)$. We then extend the definition to a multi-point local function $\chi = \chi_1 + \dots + \chi_\ell$ on $W$ to obtain $\bar{\chi} = \bar{\chi}_1 + \dots + \bar{\chi}_\ell \in \g^*_\n$. 

    For $\widehat{\n} \geq \n$ (defined in~\eqref{eq:biggertuple}), we have the canonical projection $\tau_{\widehat{\n}}: \g_{\widehat{\n}} \to \g_\n$. We denote the pullback $\tau_{\widehat{\n}}^*(\bar{\chi}) \in \g^*_{\widehat{\n}}$ by $\widehat{\chi}$. Note that we have defined a map 
    \begin{equation}\label{eq:mapchibar}
        \pi^{\widehat{\n}}: \Loc^{\leq \widehat{\n}} \to \g^*_{\widehat{\n}}, \quad \chi \mapsto \widehat{\chi}
    \end{equation}
    and it is easy to see that this map is surjective. 

    Furthermore, we denote by $Q_{\bar{\chi}}$ and $Q_{\widehat{\chi}}$ the primitive ideal of $\Ua(\g_{\n})$ and $\Ua(\g_{\widehat{\n}})$ corresponding to $\bar{\chi}$ and $\widehat{\chi}$ respectively under the Dixmier map.
\end{definition}

For a ring $R$, let $\mathcal{I}(R)$ denote its ideal lattice. Note that by ~\cite{mcconnell2001noncommutative}*{Lemma 9.6.9}, ~\cite{dixmier1996enveloping}*{Lemma 4.5.1}, since $\widetilde{A}_\ell$ is a simple ring with trivial center the embedding $\iota: \Ua(\g_{\widehat{\n}}) \to \widetilde{T}_{\widehat{\n}}$
induces a $1-1$ correspondence between the two ideal lattices
\begin{equation}\label{eq:isoideal}
    \iota_*: \mathcal{I}(\Ua(\g_{\widehat{\n}})) \rightarrow \mathcal{I}(\widetilde{T}_{\widehat{\n}}), \quad I \mapsto \widetilde{A}_\ell \otimes I. 
\end{equation}
Furthermore, $I$ is maximal (prime) if and only if $\widetilde{A}_\ell \otimes I$ is maximal (prime) and $\iota_*$ is a homeomorphism between $\Spec (\Ua(\g_{\widehat{\n}}))$ and $\Spec(\widetilde{T}_{\widehat{\n}})$ (with Zariski topology) as it respect inclusions of ideals in both directions. We could show that $\iota_*$ is furthermore a homeomorphism between $\Prim \Ua(\g_{\widehat{\n}})$ and $\Prim \widetilde{T}_{\widehat{\n}}$ by using the fact that both rings satisfy the Dixmier-Moeglin equivalence ~\cite{bell2016dixmiermoeglin}. 

In our case, we can show a stronger result that the primitive ideal $\Ann_{\widetilde{T}_{\widehat{n}}} L_{\chi'}$ can be written as
\begin{equation}
    \Ann_{\widetilde{T}_{\widehat{n}}} L_{\chi'} = \widetilde{A}_1 \otimes Q_{\widehat{\chi}}.
\end{equation}

\begin{proposition}\label{prop:polbarchi}
    Let $\p_{\widehat{\chi}}$ be the subalgebra $\W{m}/\W{\widehat{n}}$ of $\g_{\widehat{n}}$. Then $\p_{\widehat{\chi}}$ is a polarization of $\widehat{\chi}$.  
\end{proposition}
\begin{proof}
    Note that $[\p_{\widehat{\chi}}, \p_{\widehat{\chi}}]= \W{2m+2}/\W{\widehat{n}}$ and $\restr{\widehat{\chi}}{\W{2m+2}/\W{\widehat{n}}} = 0$, so  $\p_{\widehat{\chi}}$ is an isotropic subspace of $\g_{\widehat{n}}$ with respect to $B_{\widehat{\chi}}$. We now show that it is maximal isotropic subspace of $\g_{\widehat{n}}$. Suppose there is some nonzero element $\gamma_0 v_0 + \dots \gamma_{m-1} v_{m-1} \in \p_{\widehat{\chi}}$ and let $i<m$ be the smallest integer with $\gamma_i \neq 0$. Note that $v_{n-1} \in \p_{\widehat{\chi}}$, and 
    \begin{align}
        [\gamma_i v_{i} + \dots + \gamma_{m-1} v_{m-1}, v_{n-1-i}] = \gamma_i (n-1-2i) v_{n-1} + \text{ a linear combination of $v_{\geq n}$}. 
    \end{align}  
    Then $\widehat{\chi}([\gamma_i v_{i} + \dots + \gamma_{m-1} v_{m-1}, v_{n-1-i}]) = \gamma_i (n-1-2i) \widehat{\chi}(v_{n-1}) \neq 0$. Thus, $\p_{\widehat{\chi}}$ is a polarization of $\widehat{\chi}$.
\end{proof}
We call $\p_{\widehat{\chi}}$ the \emph{canonical polarization} of $\widehat{\chi}$ and show that for these polarization, we do not need to take the one-dimensional twist $\rho_{\p_{\widehat{\chi}}, \g_{\widehat{n}}}$ defined in Definition~\ref{def:onedimtwist}.

\begin{proposition}
    Let $\widehat{\chi} \in \g^*_{\widehat{n}}$ and let  $\p_{\widehat{\chi}}$ be the canonical polarization of $\widehat{\chi}$ in Proposition \ref{prop:polbarchi}. Recall the representations $M'_{\widehat{\chi}} = M'_{\widehat{\chi}, \p_{\widehat{\chi}}}$ in Definition~\ref{def:orbitrep}. Then $Q_{\widehat{\chi}} = \Ann_{\Ua(\g_{\widehat{n}})}M'_{\widehat{\chi}}$. 
    
    Let $J_{\widehat{\chi}} =  \Ua(\g_{\widehat{n}})(v_m - \widehat{\chi}(v_m), \dots, v_{\widehat{n}} - \widehat{\chi}(v_{\widehat{n}}))$. 
    Then $M'_{\widehat{\chi}} \cong \Ua(\g_{\widehat{n}})/J_{\widehat{\chi}}$, and is a simple $\Ua(\g_{\widehat{n}})$-module. 
\end{proposition}
\begin{proof}
    Note that $Q_{\widehat{\chi}} = \Ann_{\Ua(\g_{\widehat{n}})}M'_{\widehat{\chi}}$ and $M'_{\widehat{\chi}}$ is irreducible follow from Theorem~\ref{theorem: Kirillov method}. We want to show that $\rho_{\p_{\widehat{\chi}},\g_{\widehat{n}}} = 0$.  
    If $m = 0$, then $\p_{\widehat{\chi}} = \g_{\widehat{n}}$ and $\rho_{\p_{\widehat{\chi}},\g_{\widehat{n}}} = 0$. Otherwise, if $m > 0$, then  $\p_{\widehat{\chi}} = \W{m}/\W{\widehat{n}}$ is an ideal of $\Ua(\g_{\widehat{n}})$, and by ~\cite{dixmier1996enveloping}*{5.5.1}, $\rho_{\p_{\widehat{\chi}},\g_{\widehat{n}}} = 0$. Then by Lemma~\ref{lem:rewriterep}, $M'_{\widehat{\chi}} \cong \Ua(\g_{\widehat{n}})/J_{\widehat{\chi}}$.
\end{proof}

We are now prepared to show the main equation~\eqref{eq:Qchibarchi} that relate annihilators of local representations of $W$ to primitive ideals of $\Ua(\g_{\widehat{\n}})$. We first show the one-point case in Proposition~\ref{prop:Qchihatchione} and for a general local function in Proposition~\ref{prop:Qchihatchimulti}.
\begin{proposition}\label{prop:Qchihatchione}
    Let $\chi = \chi_{x; \alpha_0, \dots, \alpha_n}$ be a one-point local function on $W$ where $x, \alpha_n \neq 0$. Let $\widehat{n} \geq n$ and $\widehat{\chi}$ be the corresponding element in $\g^*_{\widehat{n}}$ in Definition~\ref{not:barchi}. Recall the definition of $Q_\chi$ in Definition~\ref{def:canrep} and $Q_{\widehat{\chi}}$ in Definition~\ref{not:barchi}. Then 
    \begin{align*}
        Q_\chi = (\Psi_{\widehat{n}})^{-1} (\widetilde{A}_1 \otimes Q_{\widehat{\chi}}).
    \end{align*}
\end{proposition}
\begin{proof}
    By Corollary~\ref{cor:relatingprimtive}, it suffices to show that $\Ann_{\widetilde{T}_{\widehat{n}}}L_{\chi'} = \widetilde{A}_1 \otimes Q_{\widehat{\chi}}$. Let $\widehat{Q} = \Ann_{\widetilde{T}_{\widehat{n}}}L_{\chi'} \cap \Ua(\g_{\widehat{n}})$. By~\eqref{eq:isoideal}, $\Ann_{\widetilde{T}_{\widehat{n}}}L_{\chi'} = \widetilde{A}_1 \otimes \widehat{Q}$. We claim that $\widehat{Q} = Q_{\widehat{\chi}}$.
    
    It is obvious that $\widehat{Q} \subseteq Q_{\widehat{\chi}}$ because if $z \in \Ua(\g_{\widehat{n}}) \cap \Ann_{\widetilde{T}_{\widehat{n}}}L_{\chi'}$ then $z \in \Ann_{\Ua(\g_{\widehat{n}})}M'_{\widehat{\chi}} = Q_{\widehat{\chi}}$. For the other direction, let $z \in Q_{\widehat{\chi}}$. Then $z$ commutes with $\del \in \widetilde{A}_1$, thus we have
    \begin{align*}
        z (\del^{k_{-1}} v_0^{k_0} \dots v_{m-1}^{k_{m-1}} + I_{\chi'}) = \del^{k_{-1}} z (v_0^{k_0} \dots v_{m-1}^{k_{m-1}} + I_{\chi'}) = 0. 
    \end{align*}
    since $z v_0^{k_0} \dots v_{m-1}^{k_{m-1}} + J_{\widehat{\chi}} = 0 \Rightarrow z v_0^{k_0} \dots v_{m-1}^{k_{m-1}} \subseteq J_{\widehat{\chi}} \subseteq I_{\chi'}$. Thus $z \in \Ann_{\widetilde{T}_{\widehat{n}}}L_{\chi'}$ and $z \in \widehat{Q}$ . Therefore $\widehat{Q} = Q_{\widehat{\chi}}$. 
\end{proof}

\begin{proposition}\label{prop:Qchihatchimulti}
    Let $\chi$ be a multi-point local function on $W$ following Notation~\ref{not:localfunc}. Let $\widehat{\n} \in \NN^\ell$ such that $\widehat{n}_i \geq n_i$ and $\widehat{\chi}$ be the corresponding element in $\g^*_{\widehat{\n}}$ defined in Definition~\ref{not:barchi}. Recall the definition of $Q_\chi$ in Definition~\ref{def:canrep} and $Q_{\widehat{\chi}}$ in in Definition~\ref{not:barchi}. Then 
    \begin{align*}
        Q_\chi = (\Psi_{\widehat{\n}})^{-1} (\widetilde{A}_\ell \otimes Q_{\widehat{\chi}}).
    \end{align*}
\end{proposition}
\begin{proof}
    Using the multi-point version of Definition~\ref{def:Ichi}, let $I_{\chi'} = \widetilde{T}_{\widehat{\n}} (I_{\chi'_1} + \dots  I_{\chi'_\ell})$ be a left ideal of $\widetilde{T}_{\widehat{\n}}$ and $L_{\chi'} = \widetilde{T}_{\widehat{\n}}/I_{\chi'}$. Then we have an isomorphism 
    \begin{equation}\label{eq:isoLmulti}
        L_{\chi'} \cong L_{\chi'_1} \otimes \dots \otimes L_{\chi'_\ell}.
    \end{equation}
    Since $M'_{\chi_i} \cong \res{\Psi_{\widehat{n}_i}} L_{\chi'_i}$, by Proposition~\ref{prop:isoL}, we have $M'_{\chi_1} \otimes \dots \otimes M'_{\chi_\ell} \cong \res{\Psi_{\widehat{n}_1}} L_{\chi'_1} \otimes \dots \otimes \res{\Psi_{\widehat{n}_\ell}}L_{\chi'_\ell}$. Note that $\res{\Psi_{\widehat{n}_1}} L_{\chi'_1} \otimes \dots \otimes \res{\Psi_{\widehat{n}_\ell}}L_{\chi'_\ell} \cong \res{\Psi_{\widehat{n}_1} \otimes \dots \Psi_{\widehat{n}_1}} L_{\chi'_1} \otimes \dots \otimes L_{\chi'_\ell}$. Thus, by Proposition~\ref{prop:tensor1point}, $M'_\chi  \cong \res{\Psi_{\widehat{\n}}} L_{\chi'}$. So, $Q_\chi = \Ann_{\Ua(W)} M'_\chi = (\Psi_{\widehat{\n}})^{-1}(\Ann_{\widetilde{T}_{\widehat{\n}}}L_{\chi'})$. We proceed similarly to Proposition~\ref{prop:Qchihatchione}.

    It suffices to show that $\Ann_{\widetilde{T}_{\widehat{\n}}}L_{\chi'} = \widetilde{A}_\ell \otimes Q_{\widehat{\chi}}$. Let $\widehat{Q} = \Ann_{\widetilde{T}_{\widehat{\n}}}L_{\chi'} \cap \Ua(\g_{\widehat{\n}})$. By~\eqref{eq:isoideal}, $\Ann_{\widetilde{T}_{\widehat{\n}}}L_{\chi'} = \widetilde{A}_\ell \otimes \widehat{Q}$. We claim that $\widehat{Q} = Q_{\widehat{\chi}}$.
    
    It is obvious that $\widehat{Q} \subseteq Q_{\widehat{\chi}}$ as if $z \in \Ua(\g_{\widehat{\n}}) \cap \Ann_{\widetilde{T}_{\widehat{\n}}}L_{\chi'}$ then $z \in \Ann_{\Ua(\g_{\widehat{\n}})}M'_{\widehat{\chi}}$. For the other direction, let $z \in Q_{\widehat{\chi}}$. As $z$ commutes with $\del_1, \dots, \del_\ell \in \widetilde{A}_\ell$, then for any word $\varsigma_i \in \Ua(\g_{\widehat{n}_i}), k_i \in \NN$, 
    \begin{align*}
        z \big((\del_1^{k_1} \varsigma_1 + I_{\chi'_1}) \otimes \dots \otimes (\del_\ell^{k_\ell} \varsigma_\ell + I_{\chi'_\ell})\big) 
        &= \del_1^{k_1}\dots \del_\ell^{k_\ell}  z  \big((\varsigma_1 + I_{\chi'_1}) \otimes \dots \otimes(\varsigma_\ell + I_{\chi'})\big ) \\
        &=\del_1^{k_1}\dots \del_\ell^{k_\ell}  z  (\varsigma_1 \dots \varsigma_\ell + I_{\chi'}), 
    \end{align*}
    where the last equation is induced by the isomorphism \eqref{eq:isoLmulti}. As $z \varsigma_1 \dots \varsigma_\ell + J_{\widehat{\chi}} = 0$, then $z \varsigma_1 \dots \varsigma_\ell \subseteq J_{\widehat{\chi}} \subseteq I_{\chi'}$, and thus $z \in \Ann_{\widetilde{T}_{\widehat{\n}}}L_{\chi'}$, i.e., $z \in \widehat{Q}$. Therefore $\widehat{Q} = Q_{\widehat{\chi}}$. 
\end{proof}

We immediately obtain two useful corollaries that all primitive ideals of $\Ua(\g)$ of the form $Q_\chi$ are completely prime ideals, and furthermore, they are kernels to some homomorphism from $\Ua(W)$ to some Weyl algebra. 

\begin{corollary}\label{cor:comprime}
    Let $\g$ be $\W{-1}, W$ or $Vir$, and $\chi$ be a local function on $\g$. Then $Q_\chi$ is a completely prime ideal. Thus $\Ua(\g)$ satisfies ACC on primitive ideals of the form $Q_\chi$. 
\end{corollary}
\begin{proof}
    By ~\cite{dixmier1996enveloping}*{Theorem 3.7.2.} the primitive ideal $Q_{\widehat{\chi}}$ is a completely prime ideal of $\Ua(\g_\n)$. Then $\tilde{A}_\ell \otimes \Ua(\g_\n)/Q_{\widehat{\chi}}$ is an iterated Ore extension of a domain, and thus is still a domain. Hence $\tilde{A}_\ell \otimes Q_{\widehat{\chi}}$ is a completely prime ideal. Thus, $Q_\chi = \Psi^{-1}_{\n}(\tilde{A}_\ell \otimes Q_{\widehat{\chi}})$ is a completely prime ideal. By ~\cite{iyudu2020enveloping}*{Proposition 6.4.}, since $\Ua(\g)$ satisfies ACC on completely prime ideals, it satisfies ACC on primitive ideals of the form $Q_\chi$. 
\end{proof}

\begin{corollary}\label{cor:annkerweyl}
    Let $\g$ be $\W{-1}, W$ or $Vir$, and $\chi$ be a local function on $\g$. Then there exists $k \in \NN$ and $\varphi: \Ua(\g) \to \tilde{A}_k$ such that $Q_\chi = \ker \varphi$.
\end{corollary}
\begin{proof}
    From Proposition~\ref{prop:Qchihatchimulti}, $Q_\chi = (\Psi_{{\n}})^{-1} (\widetilde{A}_\ell \otimes Q_{\bar{\chi}})$. Since $\g_\n$ is a solvable Lie algebra, from ~\cite{joseph1974ProofOT}, there exists $\bar{k} \in \NN$ and an isomorphism $\overline{\varphi}: \Ua(\g_\n)/Q_{\bar{\chi}} \to A_{\bar{k}}$. We have $\varphi := \overline{\varphi} \circ \Psi_{\n}: \Ua(\g) \to \tilde{A}_\ell \otimes A_k$ and $Q_\chi = \ker \varphi$.
\end{proof}

We will now show that $\Psi_{\widehat{\n}}$ is the ``universal homomorphism'' for all local functions of order $\leq \n$ (defined in~\eqref{eq:biggertuple}) in the sense that $\ker \Psi^W_{\widehat{\n}} = \bigcap_{\substack{\chi: \text{ local function on } W \\ \text{ of order } \leq \widehat{\n}}} Q_\chi$. 
\begin{corollary}\label{cor:psiuniversal}
    The map $\Psi^W_{\widehat{\n}}$ and $\Psi^{\W{-1}}_{\widehat{\n}}$ are universal in the sense that 
    \begin{align}
        \ker \Psi^W_{\widehat{\n}} &= \bigcap_{\chi: \text{ local function on } W \text{ of order } \leq \widehat{\n}} Q_\chi \text{ and }\\
        \ker \Psi^{\W{-1}}_{\widehat{\n}} &= \bigcap_{\chi: \text{ local function on } \W{-1} \text{ of order } \leq \widehat{\n}} Q_\chi.
    \end{align}
\end{corollary}
\begin{proof}
    As before, it suffices to show only the statement for $W$. Let $\chi$ be a local function on $W$ of order $\n$ following Notation~\ref{not:localfunc} and assume $\n \leq \widehat{\n}$. Recall $\widehat{\chi} \in \g^*_{\widehat{\n}}$ in Definition~\ref{not:barchi}. By Proposition~\ref{prop:Qchihatchimulti}, $Q_\chi = (\Psi_{\widehat{\n}})^{-1} (\widetilde{A}_\ell \otimes Q_{\widehat{\chi}})$. The surjective map $\pi^{\widehat{\n}}: \Loc^{\leq \widehat{\n}} \to \g^*_{\widehat{\n}}$ in Definition~\ref{not:barchi} yields
    \begin{align}
        \bigcap_{\chi: \text{ local function on } W \text{ of order } \leq \widehat{\n}} Q_\chi = (\Psi_{\widehat{\n}})^{-1} (\widetilde{A}_\ell \otimes \bigcap_{\widehat{\chi}\in \g^*_{\widehat{\n}}} Q_{\widehat{\chi}}).
    \end{align}
    As $\g_{\widehat{\n}}$ is a solvable finite-dimensional Lie algebra, the strong Dixmier map $\Dx^{\g_{\widehat{\n}}}$ is surjective, thus all primitive ideals of $\Ua(\g_{\widehat{\n}})$ are of the form  $Q_{\widehat{\chi}}$ for some $\widehat{\chi} \in \g^*_{\widehat{\n}}$. Therefore, $\bigcap_{\widehat{\chi}\in \g^*_{\widehat{\n}}} Q_{\widehat{\chi}}$ is the Jacobson radical $\Ua(\g_{\widehat{\n}})$, which is $0$ by ~\cite{behr1986}.
\end{proof}

\subsection{Annihilators of canonical local representations of totally even order}\label{subsec:even}
In this subsection, we consider local functions of totally even order $\n$, i.e. $n_i = 2m_i$ for all $i$, and calculate the annihilators of their local representations. By constructing algebra homomorphisms from $\Ua(W)$ to various localized Weyl algebras, we can realize these local representations as restricted Weyl-modules under those homomorphisms.
Firstly, we consider the case of a one-point local function $\chi = \chi_{x; \alpha_0, \dots, \alpha_n}$, where $n = 2m > 0$ and $\alpha_n, x \neq 0$. The main idea is defining an algebra embedding $\varphi_m: \Ua(\g_{2m}) \hookrightarrow A_{m}$, where $\g_{2m} = \W{0}/\W{2m}$ is the solvable subquotient of $W$ introduced in Notation~\ref{def:ringt} and $A_{m}$ is the $m$-th Weyl algebra introduced in Notation~\ref{def:Weylalg}. In this section, we mainly work with $W$, thus, we will abbreviate the notation $\Psi^W_{n}$ to $\Psi_n$ unless otherwise specified. 

It follows from Joseph ~\cite{joseph1974ProofOT} that the quotient field of $\Ua(\g_{2m})$ will be isomorphic to $D_{m}$, the quotient field of the Weyl algebra $A_{m}$. Abusing notation, we define $\tilde{A}_{m+1} \coloneqq \kk[t, t^{-1}, \del]\otimes A_m$. We provide an explicit proof and embedding in our case, which will be important to construct an explicit map $\overline{\Psi}_{2m}: \Ua(W) \to \tilde{A}_{m+1}$ and a module $N_\chi$ over $\tilde{A}_{m+1}$ such that 
\begin{align*}
    M'_\chi \cong \res{\overline{\Psi}_{2m}}N_\chi.
\end{align*}
As a corollary, we obtain that the annihilator of a local representation of even order only depends on its pseudo-orbit. Moreover, given two one-point local functions of even order, we completely characterize inclusion of the annihilators their local representations. Finally, we generalize the construction to multipoint local function $\chi$, where each component $\chi_i$ of $\chi$ has even order.

We first show an explicit embedding of $\Ua(\g_{2m})$ in $A_{m}$, which in turns yields a ring homomorphism $\overline{\Psi}_{2m}: \Ua(W) \to \tilde{A}_{m+1}$. 
\begin{proposition}\label{prop: weyl algebra embedding}
    Recall the notation of the $m$-th Weyl algebra $A_m$ defined in Notation~\ref{def:ringt}. There is an algebra embedding $\varphi_m: \Ua(\g_{2m}) \to A_{m}$ given by 
    \begin{align}\label{eq:weylmap}
        \varphi_m: \g_{2m} &\to A_{m} = \kk[s_0, \dots, s_{m-1}, \del_0, \dots, \del_{m-1}],\nonumber \\ 
        v_{m+j} &\mapsto s_{j}, \quad v_{j} \mapsto \sum_{i = 0}^{m-j-1} (m+i-j) s_{i+j} \del_i, \quad \text{ for } j = 0, \dots, m-1. 
    \end{align}
\end{proposition}
\begin{proof} 
    Since $v_{m}, \dots, v_{2m-1}$ commute in $\g_{2m}$, they generate an abelian subalgebra $\mathfrak{s}$ of $\g_{2m}$. Thus, if we let $s_j = \varphi_m(v_{m+j})$ for $j = 0,\dots,m-1$, then $\varphi_m(\Ua(\mathfrak{s})) = \kk[s_0, \dots, s_{m-1}] \subset A_m$.
    
    Now $\mathfrak{s}$ is stable under the action of $v_0, \dots, v_{m-1}$ and so it is a Lie ideal of $\g_{2m}$. For any $0 \leq i \leq m-1$, the adjoint action of $v_i$ on $\Ua(\g_{2m})$, which is by definition a derivation, restricts to a derivation of $\Ua(\mathfrak{s})$. We have
    \begin{align*}
        \ad(v_j)|_{\Ua(\mathfrak{s})} = \sum_{i = 0}^{m-j-1} (m+i-j) v_{m+i+j} \frac{\del}{\del v_{m+i}} \text{ for } j = 0, \dots, m-1, 
    \end{align*}
    giving~\eqref{eq:weylmap}. Thus $\varphi_m$ is an algebra homomorphism.
    
    Let $F$ be the quotient division ring of $\im \varphi$, which exists by Goldie's Theorem. Since $s_i \in F$ and $\varphi_m(v_{m-1}) = s_{m-1} \del_0$, thus $\del_0 = \varphi_m(v_{2m-1})^{-1}\varphi_m(v_{m-1}) \in F$. As $\varphi_m(v_{m-2}) = 2s_{m-2} \del_0 + 3s_{m-2} \del_1$, then similarly $\del_1 \in F$. Iteratively, we can show that $\del_j \in F$ for $j = 0, \dots, m-1$. Thus $D_{m} = F$. 
    
    Note that ~\cite{zhang1996}*{Theorem 1.1 and Theorem 1.2} imply that for every subalgebra $B$ of $A_m$ such that the quotient division ring of $B$ equals $D_m$, then $\GK B = \GK A_m= 2m$. Hence $\GK \im \varphi_m = 2m$. Suppose for the contrary that $\varphi_m$ is not injective. Since $\Ua(\g_{2m})$ is a domain, then by ~\cite{krause2000growth}*{Proposition 3.5}, $\GK \im \varphi_m < \GK \Ua(\g_{2m}) = 2m$, we thus obtain a contradiction. Hence $\varphi_m$ is injective. 
\end{proof}
With $\Psi_0: \Ua(W) \to \tilde{A}_1$, we thus have a family of algebra homomorphisms from $\Ua(W)$ to any localized Weyl algebra.
\begin{corollary}\label{cor: weyl map even}
    For all $m \in \NN$, there is an algebra homomorphism $\overline{\Psi}^{W}_{2m} = (\id \otimes \varphi_m) \circ \Psi^{W}_{2m} $
    \begin{align}\label{eq: map even}
        \overline{\Psi}^{W}_{2m}: \Ua(W) &\to \tilde{A}_{m+1}, \nonumber \\
        f\del &\mapsto f(t) \del + f'(t) \bigl ( \sum_{i = 0}^{m-1} (m+i) s_{i} \del_i \bigr ) + \dots + \frac{f^{(2m)}(t)}{(2m)!}s_{m-1}.
    \end{align}
    \qed
\end{corollary}

Let $\chi = \chi_{x; \alpha_0, \dots, \alpha_{2m}}$ be a one-point local function on $W$ of even order $2m >0$, and let $\p_\chi = W((t-x)^{m+1})$ be the canonical polarization of $\chi$ and $M'_\chi$ be the canonical local representation of $W$. Recall the definition of $\chi'$ in Notation~\ref{not:chi'}. Since $2m' > 1$, $\chi' = \chi$ and by Proposition~\ref{prop:untwistrep}, $M'_\chi \cong M_\chi$. We now construct a module $N_\chi$ over $\tilde{A}_{m+1}$ such that $M_\chi \cong \res{\overline{\Psi}_{2n}} N_\chi$. By Proposition~\ref{prop: weyl algebra embedding}, $\widetilde{T}_{2m} = \kk[t, t^{-1}, \del] \otimes \Ua(\g_{2m})$ is a subalgebra of $\tilde{A}_{m+1}$. Recall the ideal $I_{\chi}$ defined in Definition~\ref{def:Ichi}. Consider the left ideal $ \tilde{A}_{m+1} I_\chi$ and let $N_\chi = \tilde{A}_{m+1}/ \tilde{A}_{m+1} I_\chi$, which is a simple left $\tilde{A}_{m+1}$-module. 
The homomorphism $\overline{\Psi}_{2m}$ induces a $W$-action on $N_\chi$ given by
\begin{equation}\label{eq: restricted module even}
    f \del \cdot (1+  \tilde{A}_{m+1} I_\chi) = \overline{\Psi}_{2m}(f \del) +   \tilde{A}_{m+1} I_\chi.
\end{equation} 
We will now show $M_\chi$ is isomorphic to $N_\chi$ as a representation of $W$.

\begin{proposition}\label{prop: second order hom even}
    Let $\chi = \chi_{x; \alpha_0, \dots, \alpha_{2m}}$ be a one-point local function of order $2m >0$ on $W$. There exists an isomorphism of $\Ua(W)$-modules $\theta: M'_\chi \to N_\chi$ such that $\theta(1_\chi) = 1+  \tilde{A}_{m+1} I_\chi$, where $1_\chi$ is the canonical generator of $M_\chi$.
\end{proposition}
\begin{proof}
    Since $2m > 1$, $\chi' = \chi$ and by Proposition~\ref{prop:untwistrep} $M'_\chi \cong M_\chi$. Using Proposition~\ref{prop:untwistrep}, we write $M_\chi = \Ua(W)/J_\chi$, where
    \begin{align}
        J_\chi = \Ua(W)(u_{m} - \chi(u_{m}), \dots, u_{2m} - \chi(u_{2m}), [\p_\chi, \p_\chi]).
    \end{align}
    The homomorphism $\overline{\Psi}_{2m}$ induces a $\Ua(W)$-module map
    \begin{align*}
        \tilde{\theta}: \Ua(W) \rightarrow N_\chi, \quad 
        p \mapsto \overline{\Psi}_{2m} (p) +   \tilde{A}_{m+1} I_\chi.
    \end{align*}
    We claim that $\tilde{\theta}(J_\chi) = 0$. If $q \in \kk[t, t^{-1}]$, then $\overline{\Psi}_{2m} ((t-x)^{2m+1} q\del)$ equals
    \begin{align*}
        (t-x)^{2m+1}q\del + \Bigl (\sum_{i = 0}^{m-1} (m+i) s_{i} \del_i \Bigr ) \Bigl ((t-x)^{2m+1} q\Bigr )' + \dots + s_{m-1}\frac{\Bigl ((t-x)^{2m+1} q \Bigr )^{(2m)}}{(2m)!}.
    \end{align*}
    As every term of $\overline{\Psi}_{2m} ((t-x)^{2m+1} q\del)$ is in $\tilde{A}_{m+1}(t-x) \subseteq \tilde{A}_{m+1}I_\chi$, thus $ \overline{\Psi}_{2m} ((t-x)^{2m+1} q\del) \in   \tilde{A}_{m+1} I_\chi$ and $\tilde{\theta}([\p_\chi, \p_\chi]) = 0$. Now for $j = 0,\dots, m-1$, we have
    \begin{align*} 
        &\tilde{\theta}((t-x)^{m+1+j} \del) = \overline{\Psi}_{2m} ((t-x)^{m+1+j} \del) +   \tilde{A}_{m+1} I_\chi \\
        &= (t-x)^{m+1+j} \del + 
        \bigl (\sum_{i = 0}^{m-1} (m+i) s_{i} \del_i \bigr ) 
        \bigl ( (t-x)^{m+1+j} \bigr )' + \dots +  s_{j}\frac{\bigl ( (t-x)^{m+1+j} \bigr )^{(m+j+1)}}{(m+j+1)!}+   \tilde{A}_{m+1} I_\chi  \\ 
        &= s_{j} +   \tilde{A}_{m+1} I_\chi,
    \end{align*}
    since every term except the last one is in $\tilde{A}_{m+1}(t-x) \subseteq \tilde{A}_{m+1}I_\chi$. Hence $\tilde{\theta}(J_\chi) = 0$, so $\tilde{\theta}$ factors through the canonical map $\Ua(W) \to M_\chi$ to give a well-defined homomorphism $\theta: M_\chi \to N_\chi$.

    Define a lexicographic order on the basis $S\coloneqq \{\del^{k_{-1}} \del_{m-1}^{k_{m-1}} \dots \del_0^{k_0} | k_i \in \ZZ \}$ of $N_\chi$ by fixing $\del_{m-1} > \dots > \del_0 > \del$. We will show by downward induction on the smallest variable appearing in $u_{-1}^{k_{-1}} u_0^{k_0} \dots u_{m-1}^{k_{m-1}}  1_\chi$ that 
    \begin{equation}\label{eq: for surjectivity and injectivity Weyl module}
        \theta(u_{-1}^{k_{-1}} u_0^{k_0} \dots u_{m-1}^{k_{m-1}}  1_\chi) = c\del^{k_{-1}}\del_{m-1} ^{k_0}\dots \del_{0}^{k_{m}} + [\text{terms of lower order}] +   \tilde{A}_{m+1} I_\chi, 
    \end{equation}
    for $c \in \kk^\times$. The base case is obvious as $\theta(1_\chi) = 1+  \tilde{A}_{m+1} I_\chi$.
    
    Assuming the induction hypothesis, we first consider when there is no $u_{-1}$. Then for $j \geq 0$
    \begin{align*}
        \theta(u_{j}^{k_j+1}\dots u_{m-1}^{k_{m-1}} 1_\chi) 
        &= \overline{\Psi}_{2m}(u_{j})(c\del_{m-1-j}^{k_j}\dots \del_{0}^{k_{m-1}} +   \tilde{A}_{m+1} I_\chi  + [\text{terms of lower order}]).
    \end{align*}
    Since  
    \begin{align*}
        \overline{\Psi}_{2m}(u_j) = (t-x)^{j+1}\del + \binom{j+1}{1} \bigl (\sum_{i=0}^{m-1}(m+j)s_i\del_i \bigr )(t-x)^{j} + \dots + \bigl (\sum_{i=0}^{m-j-1}(m+i-j)s_{i+j}\del_i \bigr ),
    \end{align*}
    and $(t-x)$ commutes with $s_0, \dots,  s_{m-1}, \del_0, \dots, \del_{m-1}$, then $\overline{\Psi}_{2m}(u_{j}) - \bigl (\sum_{i=0}^{m-j-1}(m+i-j)s_{i+j}\del_i \bigr ) \in \tilde{A}_{m+1} I_\chi$. 
    Thus 
    \begin{align*}
        \theta(u_{j}^{k_i+1}\dots u_{m-1}^{k_{m-1}} 1_\chi) &= \text{LT}\bigl (c\overline{\Psi}_{2m}(u_{j})(\del_{m-1-i}^{k_j}\dots \del_{0}^{k_{m}}) +   \tilde{A}_{m+1} I_\chi \bigr ) + [\text{terms of lower order}] \\
        &= c(2m-1-j)s_{m-1}\del_{m-1-j}^{k_j+1}\dots \del_{0}^{k_{m}} +   \tilde{A}_{m+1} I_\chi + [\text{terms of lower order}]\\
        &= c(2m-1-j) (2m+1)! \alpha_{2n} \del_{m-1-j}^{k_j+1}\dots \del_{0}^{k_{m}} +   \tilde{A}_{m+1} I_\chi + [\text{terms of lower order}].
    \end{align*}
    Then~\eqref{eq: for surjectivity and injectivity Weyl module} follows as $\overline{\Psi}_{2m}(u_{-1}) = \del$. It follows from~\eqref{eq: for surjectivity and injectivity Weyl module} that $\{ \theta(u_{-1}^{k_{-1}}\dots u_{m-1}^{k_{m-1}} 1_\chi) \}$ are linearly independent, and $\theta$ is injective. Furthermore, since $S$ spans $N_\chi$, $\theta$ is surjective.
\end{proof}

As a result, we realize the kernel of $\Psi_{2m}$ as the annihilator of local representation of local function of order $2m$, motivating more interesting properties of the maps $\Psi_n$. 
\begin{corollary}\label{cor:kereven}
    Let $\chi = \chi_{x; \alpha_0, \dots, \alpha_{2m}}$ be a one point local function on $W$ of order $2m >0$. Then 
    \begin{align*}
        Q_\chi = \ker \Psi_{2m}.
    \end{align*}
\end{corollary}
\begin{proof}
    By Proposition~\ref{prop: second order hom even},  $\Ann_{\Ua(W)} M_{\chi} = \overline{\Psi}^{-1}_{2m}(\Ann_{\tilde{A}_{m+1}} N_{\chi})$. Since $N_\chi$ is a faithful module over $\tilde{A}_{m+1}$, $\Ann_{\Ua(W)} M_{\chi} = \ker \overline{\Psi}_{2m}$. As $\varphi_m$ is injective, $\id \otimes \varphi_m$ is injective and thus $\ker \overline{\Psi}_{2m} = \ker \Psi_{2m}$. 
\end{proof}
Thus, as a corollary, we immediately have that $Q_\chi$ is independent of the pseudo-orbit $\orbit{\chi}$ when $\chi$ is a one-point local function of order $2m>0$.
\begin{corollary}
    Let $\g = \W{-1}, W$ or $Vir$ and let $\chi$ be a one-point local function on $\g$ of order $2m>0$. If $\eta \in \orbit{\chi}$, then $Q_\chi = Q_\eta$. 
\end{corollary}
\begin{proof}
    By~\cite{petukhov2022poisson}*{Theorem 4.2.1.}, $\eta$ also has order $2m$. By Corollary~\ref{cor:kereven}, $Q_\chi = \ker \Psi_{2m} = Q_\eta$.
\end{proof}

Moreover, Corollary~\ref{cor:kereven} yields an immediate result on the inclusion of annihilators of local representations of even-order local functions. 
\begin{corollary}\label{cor:Qinclusioneven}
    Let $\g = \W{-1}, W$ or $Vir$ and let $\chi, \eta$ be two one-point local functions of positive even order. If $\chi \in \overline{\orbit{\eta}}$, then $Q_\eta \subseteq Q_\chi$.
\end{corollary}
\begin{proof}
    Let $2m$, $2\widehat{m}$ be the order of $\chi, \eta$ respectively. By ~\cite{petukhov2022poisson}*{Corollary 4.2.9.}, we have $2\widehat{m} > m > 0$. The following diagram thus commutes
    \[\begin{tikzcd}
        \Ua(W) \arrow[rd, "\Psi_{2m}", swap] \arrow[r, "\Psi_{2\widehat{m}}"] & \tilde{A}_1 \otimes \Ua(\g_{2\widehat{m}}) \arrow[d, "\id \otimes \tau_{2m}"] \\
        & \tilde{A}_1 \otimes \Ua(\g_{2m})
    \end{tikzcd}\]
    where $\tau_{2m}: \Ua(\g_{2\widehat{m}}) \to \Ua(\g_{2m})$ is the canonical projection sending $v_{j}$ to 0 for $j \geq 2m$. Hence 
    \begin{align*}
        Q_\eta = \ker \Psi_{2\widehat{m}} \subseteq \ker \Psi_{2m} = Q_\chi.
    \end{align*}
\end{proof}

We can take the multi-point version of the map $\overline{\Psi_{2m}}$ as follows. 
\begin{theorem}\label{theo:even}
    Let $\g = \W{-1}, W$ or $Vir$ and let $\chi$ be a local function on $\g$ of positive totally even order, so all components $\chi_i$ have even order $2m_i >0$. Then there is an algebra homomorphism
    \begin{align}
        \overline{\Psi}_{2\m} = ( \overline{\Psi}_{2m_1} \otimes \dots \otimes  \overline{\Psi}_{2m_\ell}) \circ \Delta^{\otimes \ell}: \Ua(W) \to \tilde{A}_\ell \otimes A_{m_1 + \dots + m_\ell}. 
    \end{align}
    Moreover
    \begin{align}
        Q_\chi = \ker \overline{\Psi}_{2\m} =\ker {\Psi}_{2\m},
    \end{align}
    and thus only depends on the pseudo-orbit of $\chi$. 
    \qed
\end{theorem} 
We also have an analog of Corollary~\ref{cor:Qinclusioneven}.
\begin{corollary}
    Let $\g = \W{-1}, W$ or $Vir$ and $\chi, \eta$ be two local functions on $\g$ of totally even order $2\mathbf{m}$, $2\widehat{\mathbf{m}}$ respectively. If $2\widehat{m}_i > 2m_i> 0$ for all $i$, then $Q_\eta \subseteq Q_\chi$.
    \qed
\end{corollary}

For local functions on $W$ which do not have totally even order, it is more difficult to write an explicit map to the Weyl algebra, so we extend Theorem~\ref{theo:even} via a different method. 
\section{The Dixmier map}\label{sec: Dixmier map}
In this final section, we prove the rest of Theorem~\ref{theo:Dixmier1}, which we restate here. 
\begin{theorem}
    Let $\g$ be $\W{-1}, W$ or $Vir$ and let $\chi, \eta$ be local functions on $\g$. If $\orbit{\chi} = \orbit{\eta}$, then $Q_\chi = Q_\eta$. 
    
    Therefore, the Dixmier map $\Dx^\g: \g^* \to \Prim \Ua(\g)$ defined in~\eqref{eq:Dixmier} factors through $\g^* \to \Pprim \Sa(\g)$ to give a well-defined strong Dixmier map $\widetilde{\Dx}^\g: \Pprim \Sa(\g) \to \Prim \Ua(\g)$, as defined in~\eqref{eq:dxWW-1} and~\eqref{eq:dxVir}. 
\end{theorem}

Let $\g$ be $\W{-1}, W$ or $\Vir$ and let $\chi$ be a local function on $\g$ of order $\n$. Recall $\g_\n = \g_{n_1} \oplus \dots \oplus \g_{n_\ell}$ defined in Definition~\ref{def:gboldn} and $\bar{\chi} \in \g^*_\n$ defined in Notation~\ref{not:barchi}. In subsection~\ref{subsec:adjointgroup}, we characterize the adjoint algebraic group $G_\n$ of $\g_\n$. Then in subsection~\ref{subsec:coorbit}, we compute the coadjoint orbit $G_\n \cdot \bar{\chi}$ and show that the map $\pi^{\leq \n}: \Loc^{\leq \n} \to \g^*_\n$ defined in~\eqref{eq:mapchibar} induces an bijection between pseudo-orbits in $\Loc^{\leq \n} \subseteq \g^*$ and coadjoint orbits in $\g^*_\n$.
Finally, we prove that the strong Dixmier maps $\widetilde{\Dx}^\g: \Pprim \Sa(\g) \to \Prim \Ua(\g)$ are well-defined.
\subsection{The adjoint algebraic group of $\g_n$}\label{subsec:adjointgroup}
In this subsection, we first characterize the adjoint group of the solvable Lie algebras $\g_n = \W{0}/\W{n} = \kk \{v_0, \dots, v_{n-1} \}$ and generalize it to multi-point version. The case $\g_0 = 0$ is tautological so we consider $n \geq 1$. We can write the Lie algebra $\g_n$ differently as  
\begin{equation}\label{eq:gniso}
    \g_n \cong  \frac{a\kk[a]}{a^{n+1}\kk[a]}\del_a
\end{equation}
by identifying $v_i = a^{i+1}\del$. For $s \in \kk[a]$ with $s = c_1 a + \dots c_n a^n$ with $c_1 \neq 0$, note that there is a unique algebra automorphism of $\kk[a]/a^{n+1}\kk[a]$ as follows
\begin{equation}\label{eq: G_n action}
    \End_{a \to s}: \kk[a] \to \kk[a], \quad a \mapsto s.
\end{equation}
Note that $s'$ is invertible in $\frac{\kk[a]}{a^{n+1}\kk[a]}$. We then extend $\End_{a \to s}$ to a related map $g_s$ by~\eqref{eq: G_n action del}, which we will show is a Lie algebra automorphism of $\g_n$
\begin{equation}\label{eq: G_n action del}
    g_s: \g_n \to \g_n, \text{ induced by }  a \mapsto s, \del \mapsto \frac{1}{s'}\del.
\end{equation}

\begin{lemma}\label{lem:Endauto}
    Let $s = c_1 a + \dots + c_n a^n$ with $c_1 \neq 0$. Then $g_s$ is an automorphism of $\g_n$.
\end{lemma}
\begin{proof}
    The map $g_s$ is a Lie algebra homomorphism of $\g_n$ as for all $i, j \geq 1$
    \begin{align*}
        [g_s (a^i \del_a), g_s (a^j \del_a)] &= s^i \frac{1}{s'} \del s^j \frac{1}{s'} \del - s^j \frac{1}{s'} \del s^i \frac{1}{s'} \del \\
        &= (j-i)s^{i+j-1} \frac{1}{s'} \del = g_s([a^i \del_a, a^j \del_a]).
    \end{align*}
    We define an order on $\g_n$ by letting $v_0 > v_1 > \dots > v_{n-1}$, then 
\begin{align*}
g_s (a^i \del) = s^i \frac{1}{s'} \del = (c_1 a + \dots + c_n a^n)^i \frac{1}{c_1}(1- \frac{2c_2}{c_1}a + \dots )\del = c_1^{i-1}a^i\del + [\text{terms of lower order}].
\end{align*}
Since $g_s$ sends a basis of $\g_n$ to a basis of $\g_n$, it is an automorphism. 
\end{proof}

\begin{definition}\label{def:Gn}
    For $n \leq 1$, let $G_n = \id$. For $n \in \mathbb{Z}_{\geq 2}$, let $G_n = \{g_s| s = c_1 a + \dots c_n a^n \text{ with } c_1 \neq 0 \}$. Then $G_n$ is an algebraic subgroup of $\mathrm{GL}(\g_n)$ for all $n$.

    For $\n = (n_1, \dots, n_\ell)$, let $G_\n = G_{n_1} \times \dots \times G_{n_\ell}$. As $\g_\n = \g_{n_1} \oplus \dots, \oplus \g_{n_\ell}$, then it is obvious that $G_\n$ is an algebraic subgroup of $\mathrm{GL}(\g_\n)$.
\end{definition}

\begin{example}[Computations of the group $G_n$ for small $n$]\label{ex:piGn}
    For $\g_1 = \kk \{v_0\}$, we have $g_{c_1 a}(a \del) = c_1 a \frac{1}{c_1}\del = a\del$ then $G_1 = \{1\}$. 
        
    For $\g_2 = \kk \{ v_0, v_1 \}$, we have 
\begin{align*}
    g_{c_1 a + c_2 a^2}(a \del) &=   (c_1 a + c_2 a^2) \frac{1}{c_1 + 2c_2 a}\del = (c_1 a + c_2 a^2)\frac{1}{c_1}(1-\frac{2c_2}{c_1} a + \dots ) \del = a \del - \frac{c_2}{c_1} a^2 \del, \\
    g_{c_1 a + c_2 a^2}(a^2 \del) &=  (c_1 a + c_2 a^2)^2 \frac{1}{c_1 + 2c_2 a}\del = (c_1 a + c_2 a^2)^2 \frac{1}{c_1}(1-\frac{2c_2}{c_1} a + \dots ) \del = c_1 a^2 \del,
\end{align*}
then $G_2 =\left \{ \begin{bNiceArray}{cc}
1 & 0 \\
d_1 & c_1
\end{bNiceArray}: c_1 \in \kk^\times, d_1 \in \kk \right \}$.    
\end{example}

We want to show that $G_n$ is the adjoint algebraic group of $\g_n$. Recall that the \emph{adjoint algebraic group} of $\g_n$ is the smallest algebraic subgroup $G^{\ad} \leq \mathrm{GL}(\g_n)$\label{pl:adgrp} whose Lie algebra contains $\ad \g = \{ \ad x| x \in \g \}$ (see e.g. ~\cite{borho1973}*{12.2}). We first have a useful lemma outlining the structure of $G_n$.

\begin{lemma}\label{lem:G_n+1}
    For $n \geq 1$, let $s_{n+1} = c_1 a + \dots + c_{n+1} a^{n+1}$ where $c_1 \neq 0$ and let $s_n = s_{n+1} - c_{n+1}a^{n+1}$. Then 
    $g_{s_{n+1}} = 
    \begin{bNiceArray}{ccw{c}{1cm}c|c}[margin]
\Block{4-4}{g_{s_{n}}} & & & & 0 \\
& & & & 0 \\
 & & & &\Vdots \\
& & & & 0 \\
\hline
p_0 & \Cdots& & p_{n-1} & c_1^n
\end{bNiceArray}$, where $p_0 \in \kk[c_1^{-1}, c_1, c_2, \dots, c_{n+1}]$ and $p_1, \dots, p_{n-1} \in \kk[c_1^{-1}, c_1, c_2, \dots, c_{n}]$.
\end{lemma}
\begin{proof}
Let $f_{n+1} = \frac{1}{s'_{n+1}}$, $f_{n} = \frac{1}{s'_{n}}$, and note that $s'_{n+1} = s'_n + (n+1)c_{n+1}a^{n+1}$. By Taylor series expansion, we write $f_{n+1} = \sum_{i = 0} \frac{f^{(i)}_{n+1}(0)}{i!}a^i$ and $f_{n} = \sum_{i = 0} \frac{f^{(i)}_{n}(0)}{i!}a^i$. For all $i$, the maximal derivative of $s_{n+1}$ appearing in $f^{(i)}_{n+1}$ is $s^{(i+1)}_{n+1}$. Moreover, since we are evaluating at $a = 0$, $c_{n+1}a^{n+1}$ will not contribute to $f^{(i)}_{n+1}(0)$ for $i = 0, \dots, n-1$, thus 
    \begin{equation}
        f^{(i)}_{n+1}(0) = f^{(i)}_{n}(0) \text{ for all } i = 0, \dots, n-1.
    \end{equation}
Let $b_i = \frac{f^{(i)}_{n+1}(0)}{i!} = \frac{f^{(i)}_{n}(0)}{i!}$ for $i = 0, \dots, n-1$ and $b_n = \frac{f^{(n)}_{n+1}(0)}{n!}$. 
Consider now the term $b_n a^n$ in the Taylor series expansion of $f_{n+1}$, the maximal derivative of $s_{n+1}$ in $b_n$ is $s^{(n+1)}_{n+1}(0) = (n+1)!c_{n+1}$, thus $b_n \in \kk[c_1, c^{-1}_{1}, \dots, c_{n+1}]$.  

Because $s_{n+1}$ is a multiple of $a$, $g_{s_{n+1}} (a^{i}\del) = s_{n+1}^i f_{n+1} \del$, and we truncate from $a^{n+2}\del$, thus only powers up to $a^{n}$ in $f_{n+1}$ will be taken into account. Hence, we can rewrite $f_{n+1} = f_n + b_n a^{n}$. Then 
    \begin{align*}
        g_{s_{n+1}}(a^{i}\del) = s_{n+1}^i \bigl(\sum_{i = 0} b_i a^i \bigr) \del = (s_{n}+c_{n+1}a^{n+1})^i (f_n +  b_n a^{n}) \del = g_{s_{n}}(a^i \del) + p_i a^{n+1}\del,
    \end{align*}
    where $p_i$ is the coefficient of $a^{n+1}\del$ in $g_{s_{n+1}}(a^{i}\del)$. Moreover, note that when $i > 1$, $b_n a^n$ and $c_{n+1} a^{n+1}$ do not contribute to $p_i$ as $s_{n+1}$ is a multiple of $a$, and thus $p_1, \dots, p_{n-1} \in \kk[c_1, c^{-1}_{1}, \dots, c_n]$. We hence obtain the form of $g_{s_{n+1}} $ stated in the lemma. 
\end{proof}

\begin{corollary}\label{cor:dimpiG}
    $\dim G_n = \begin{cases}
        0 &\text{ if } n=0,1; \\ 
        n &\text{ otherwise. }
    \end{cases}$
\end{corollary}
\begin{proof}
    For $n=1$, from Example~\ref{ex:piGn}, we have $G_1 = \{1\}$. We then induct on $n \geq 2$, for the base case, from~\ref{ex:piGn}, $\dim G_2= 2$. Assume the induction hypothesis, from Lemma~\ref{lem:G_n+1}, $\dim G_{n+1} \leq \dim G_{n} +1$. Moreover, $p_0$ takes all values in $\kk$ as for any $c_{n+1} \in \kk$, consider $s_{n+1} = a + c_{n+1}a^{n+1}$, then
    \begin{align*}
        g_{s_{n+1}} (a\del) = \frac{a + c_{n+1}a^{n+1}}{1 + (n+1)c_{n+1}a^n} \del = (a + c_{n+1}a^{n+1}) \bigl( 1 - (n+1)c_{n+1}a^n + \dots \bigr) \del = a\del - nc_{n+1}a^{n+1}\del.
    \end{align*} 
    Thus $\dim G_{n+1} = \dim G_{n} +1 = n+1$. 
\end{proof}

Note that for $n > 1$, $G_n$ is an $n$-dimensional solvable algebraic group, and we claim that $G_n$ is the adjoint group of $\g_n$ for all $n$. Let the corresponding Lie algebra of $G_n$ be $\mathrm{lie}^{\leq n} \subseteq  \mf{gl}(\g_n)$, we first show that $\mathrm{lie}^{\leq n} = \ad(\g_n)$\label{pl:lieGn}.
\begin{lemma}\label{lem:lieGn}
    We have $\mathrm{lie}^{\leq n} = \ad(\g_n)$ as a Lie algebra where $\ad: \g_n \to \mathfrak{gl}(\g_n)$. 
\end{lemma}
\begin{proof}
    We follow the proof of ~\cite{petukhov2022poisson}*{Lemma 4.1.3.}. The Lie alegbra $\mathrm{lie}^{\leq n}$ is the tangent space to $G_n$ at the identity. Let $s = c_1 a + \dots + c_n a^n$ for $c_1 \neq 0$ and denote by $\xi_s$ the tangent direction at the identity defined by the line $a \to g_{a+hs(\cdot)}$ for $h \in \kk$. The action of $G_n$ derives to an action of $\xi_s$, so $\mathrm{lie}^{\leq n}$ consists of $\xi_s$ where $s = c_1 a + \dots + c_n a^n$ for $c_1 \neq 0$. It is enough to check that $\xi_s \cdot f\del = [s\del, f \del]$ for all $f\del \in \g_n$. We have
    \begin{align*}
        \xi_s \cdot f\del & =  (\xi_s \cdot f) \del + f( \xi_s \cdot \del)                             \\
                          & = (\del_h f(a+hs))|_{h=0} \del + f (\del_h( \frac{1}{1 + hs'}\del))|_{h=0} \\
                          & = f's\del - fs'\del = [s\del, f\del].
    \end{align*}
    Thus $\xi_s = \ad(s\del)$ and the lemma follows. 
\end{proof}

\begin{proposition}\label{prop:adjointgroup}
    $G_n$ is the adjoint group of $\g_n$.
\end{proposition}
\begin{proof}
    For $n=0, 1$, $\ad(\g_n) = 0$ and $G_n = \id$ thus $G_n$ is the algebraic group of $\g_n$ so we assume $n\geq 2$. Then $\dim \ad(\g_n) = n$, thus any algebraic group $G$ whose Lie algebra contains $\ad(\g_n)$ must have dimension at least n. From Lemma~\ref{lem:lieGn}, $\mathrm{lie}^{\leq n} = \ad(\g_n)$. So by Corollary~\ref{cor:dimpiG}, $\dim G_n = n$, thus $G_n$ is the smallest algebraic subgroup of $\mathrm{GL}(\g_n)$ whose Lie algebra contains $\ad(\g_n)$, hence $G_n$ is the adjoint group of $\g_n$. 
\end{proof}
The following corollary follows immediately and we have constructed the adjoint alegbraic group of $\g_\n$.
\begin{corollary}\label{cor:adjointgroup}
    $G_\n$ is the adjoint group of $\g_\n$.
    \qed
\end{corollary}

\begin{remark}
    We note a discrepancy between the action of $G_n$ on $\g^*_n$ and $\DLoc_{x}^{\leq n}$ on $\Loc_x^{\leq n}$ defined in ~\cite{petukhov2022poisson}. In particular, note that $\g^*_n$ or $G_n$ is of dimension $n$, whereas $\Loc^{\leq n}$ or $\DLoc^{\leq n}$ is of dimension $n+2$. This corresponds to the fact that if $\chi  = \chi_{x; \alpha_0, \dots, \alpha_n}$ is a local function on $W$, then in forming $\bar{\chi} \in \g^*_n$, we forget about the coefficient $\alpha_0 \in \kk$ as well as the base point $x \in \kk^\times$.
\end{remark}

\subsection{Coadjoint orbits of \texorpdfstring{$\g_\n$}{g n}}\label{subsec:coorbit}
Let $\chi$ be a local function on $\W{-1}, W$ or $Vir$ of order $\n$ and recall the definition of $\bar{\chi}$ in Notation~\ref{not:barchi}. In this subsection, we characterize the coadjoint orbit of $\bar{\chi}$ and show that the map $\pi^{\leq \n}: \Loc^{\leq \n} \to \g^*_\n$ defined in \eqref{eq:mapchibar} induces a bijection between pseudo-orbits in $\Pprim \Sa(\g)$ and coadjoint orbits in $\g^*_\n$. Our result and our proof are similar to ~\cite{petukhov2022poisson}*{Theorem 4.2.1}. We first consider the one-point local function case.

\begin{proposition}\label{prop:Gnorbit}
    Recall the Lie algebra $\g_n = \W{0}/\W{n}$ and its basis $v_i = e_i + \W{n}$. Let $\beta_1,\dots,\beta_n \in \kk$ with $\beta_n \neq 0$ and let $\bar{\chi} = \bar{\chi}_{\beta_1, \dots, \beta_n} \in \g^*_n$ denote the element that maps $v_i \mapsto (i+1)!\beta_{i+1}$. Denote $m = \left\lfloor \frac{n}{2}\right\rfloor$.
    \begin{enumerate}
        \item[(1)] If $n=2m$ then $G_n \cdot \bar{\chi} = G_n \cdot v^*_{n-1}$, and $\dim G_n \cdot \bar{\chi} = n$.
        \item[(2)] If $n=1$, then $G_1 \cdot \bar{\chi} =  \{\bar{\chi}\}$.
        \item[(3)] If $n = 2m+1>1$ then there is $\beta \in \kk$ such that $G_n \cdot \bar{\chi} = G_n \cdot (v^*_{n-1} + \beta v^*_{m})$, and $\dim G_n \cdot \bar{\chi} = n - 1$.
        \item[(3')] Let $\beta_1, \beta_2 \in \kk$ and $m \in \ZZ_{\geq 1}$. Then
            \begin{align}
                G_n \cdot (v^*_{2m}+ \beta_1 v^*_{m}) = G_n \cdot (v^*_{2m}+ \beta_2 v^*_{m}) \Leftrightarrow \beta_1 = \pm \beta_2.
            \end{align}
    \end{enumerate}
\end{proposition}
\begin{proof}
    The statement (2) for $G_1$ is obvious, thus, we only need to show (1) and (3). We follow the same proof in ~\cite{petukhov2022poisson}*{Theorem 4.2.1}. For $j = 1, \dots, n$ and $i \in \NN$ we have
    \begin{align}\label{eq: End on v_i}
        g_{a+\alpha a^{j+1}} (v_i) & = \frac{(a + \alpha a^{j+1})^{i+1}}{1 + \alpha(j + 1)a^j} \del = (a^{i+1} + \alpha(i - j)a
        ^{i+j+1} + [\text{terms in higher degree of a}])\del  \nonumber                                                                                       \\
                                            & = v_i + \alpha(i - j)v_{i+j} + \text{ a linear combination of } v_{> i+j},
    \end{align}
    where $a^k = 0$ and $v_{k-1}$ if $k > n$. Thus
    \begin{equation}\label{eq: End on e*_i}
        g_{ a+\alpha a^{j+1}} (v^*_i) = v^*_i + \alpha(i-j)v^*_{i-j} + \text{ a linear combination of } v^*_{< i-j},
    \end{equation}
    where $v^*_k = 0$ if $k <0$. Similarly
    \begin{equation}\label{eq: Dil scale on e*_i}
        g_{ \zeta a} (v^*_i) = \zeta^i v^*_i.
    \end{equation}
    We apply these transformations to
    \begin{align*}
        \bar{\chi}_{\beta_1, \dots, \beta_n} = \beta_1 v^*_0 + \dots  \beta_{n}(n!)v^*_{n-1},
    \end{align*}
    noting that the coefficient of $v^*_{n-1}$ is nonzero.

    By applying $g_{ a+ \alpha_1 a^2}$ with appropriate $\alpha_1$ we may cancel the coefficient of $v^*_{n-2}$, using~\eqref{eq: End on e*_i}. This
    does not affect the coefficient of $v^*_{n-1}$. Then by applying the appropriate $g_{ a + \alpha_2 a^3}$ we can cancel the coefficient of $v^*_{n-2}$ without changing the coefficient of $v^*_{n-2}$ or $v^*_{n-1}$. Repeating, we may cancel the coefficients of all $v^*_{n-1-k}$ with $1 \leq k \leq n-1$, unless $k = m$. {An important observation is the sequence of $\alpha_1, \dots., \alpha_{n-1}$ we need to apply here is exactly the prefix of the sequence of $\alpha_1, \dots, \alpha_n$ we need to apply in ~\cite{petukhov2022poisson}*{Theorem 4.2.1} for $G_n$ on ${\chi}_{0; 0, \dots, \beta_n} \in \Loc^{\leq n}_{x} \subseteq W_{\geq -1}^*$}.

    Thus in case (1) we obtain that $\omega = \beta_n v^*_{n-1} \in G_n \cdot \bar{\chi}$. In case (3) we obtain some
    \begin{equation}\label{eq: orbit odd}
        \omega = \alpha v^*_{\frac{n-1}{2}} + \beta_n v^*_{n-1} \in G_n \cdot \bar{\chi}.
    \end{equation}
    {And note that $\alpha$ in~\eqref{eq: orbit odd} is the same as $\alpha$ in ~\cite{petukhov2022poisson}*{The equation after 4.2.4}}. Now applying~\eqref{eq: Dil scale on e*_i} to $\omega$, we may rescale the coefficient of $v^*_{n-1}$ by any $\zeta^{n-1}$ (the same $\zeta$ in ~\cite{petukhov2022poisson}), so if $n \neq 1$, we can set the coefficients of $v^*_{n-1}$ to be 1. 

    Now we only need to show (3). This goes the same as ~\cite{petukhov2022poisson}*{Theorem 4.2.1}. By applying~\eqref{eq: Dil scale on e*_i}, if $\beta_1 = \pm \beta_2$ then
    \begin{align*}
        G_n \cdot (v^*_{2m} + \beta_1 v^*_{m}) = G_n \cdot (v^*_{2m} + \beta_2 v^*_{m}).
    \end{align*}
    On the other hand, pick $\beta_1, \beta_2 \in \kk$ and set $\bar{\chi}_1 = v^*_{2m} + \beta_1 v^*_{m}$, $\bar{\chi}_2 = v^*_{2m} + \beta_2 v^*_{m}$. Assume that $G_n \cdot \bar{\chi}_1 = G_n \cdot \bar{\chi}_2$, so there
    exists $s \in \kk[a], \deg s \leq n, s(0) = 0, s'(0) \neq 0$ such that  $g_{ a + s}(\bar{\chi}_1) =\bar{\chi}_2$.
    If $s = s'(0)a$, then the statement of (3') is straightforward. Assume that $s \neq s'(0)a$, then $s(a) \neq s'(0)a \mod a^d$ for some $d$, and we choose $d$ minimal.

    Then we have $d \leq 2m+1 = n$ then
    \begin{align*}
        s = s'(0) + \frac{s^{(d)}(0)}{d!}a^n \mod a^{d+1}.
    \end{align*}
    with $s^{(d)} \neq 0$ and set $\gamma = s'(0)$, $\tau = \frac{s^{(d)}(0)}{d!}$. Let $\g^*_{\leq 2m-d} = \{ \bar{\chi}_{\beta_1, \dots, \beta_n} \in \g^*_n| \beta_i = 0 \text{ for } i \geq 2m-d\}$. Further~\eqref{eq: End on e*_i},~\eqref{eq: Dil scale on e*_i} imply:
    \begin{align*}
        g_{ s}(\bar{\chi}_1) = \gamma^{-2m} v^*_{2m} + \gamma^{i -2j-1}\tau (2m-2d)e_{2m-d} \mod \g^*_{\leq 2m-d}
    \end{align*}
    This cannot give $\bar{\chi}_2$ unless $m= d$. In that case,
    \begin{align*}
        g_{ a+ \alpha t^{m+1} + ((m+1)\alpha^{2} - \alpha \frac{\beta}{2})t^{2m+1}}(v^*_{2m} + \beta v^*_{m}) = (v^*_{2m} + \beta v^*_{m})
    \end{align*}
    for every $\alpha \in \kk$. Thus, we can replace $s$ by
    \begin{equation}\label{eq: 4.2.6}
        a+ \alpha t^{m+1} + ((m+1)\alpha^{2} - \alpha \frac{\beta}{2})t^{2m+1}
    \end{equation}
    for arbitrary $\alpha$. Pick $\alpha = -\frac{s^{(d+1)}(0)}{(d+1)!(s'(0))^{d+1}}$. Then~\eqref{eq: 4.2.6} is equal to $s'(0)t \mod t^{d+1}$, and we have reduced this case to the previous one.
\end{proof}

From the proof of Proposition~\ref{prop:Gnorbit}, we have the following important corollary that outline the relationship between $\orbit{\chi}$ for $\chi$ a one-point local function on $\W{-1}, W$ or $\Vir$ and $G_n \cdot \bar{\chi}$. 
\begin{corollary}\label{cor:orbitrelation}
    Let $\g$ be $\W{-1}, W$ or $\Vir$ and $\chi, \eta$ be two one-point local functions on $\g$ of order $n$. Let $\bar{\chi}, \bar{\eta}$ be the corresponding elements of $\g^*_n$ defined in Notation~\ref{not:barchi}. Then  
    \begin{equation}\label{eq: orbit relation}
        \orbit{\chi} = \orbit{\eta} \Leftrightarrow G_n \cdot \bar{\chi} = G_n \cdot \bar{\eta}.
    \end{equation}
    Moreover, $\dim G_n \cdot \bar{\chi} = \dim \orbit{\chi} - 2$.
\end{corollary}
\begin{proof}
    Recall the definition of $\chi'$ and $\eta'$ defined in~\ref{not:chi'}. By Lemma~\ref{lem:obritchi'}, $\orbit{\chi} = \orbit{\eta} \Leftrightarrow \orbit{\chi'} = \orbit{\eta'}$. As in ~\cite{petukhov2022poisson}*{Theorem 4.3.1}, by shifting the base points, we can assume both $\chi'$ and $\eta'$ are based at 1 (note that the definition of $\bar{\chi}$ and $\bar{\eta}$ are not affected by their base points). Then $\orbit{\chi'} = \orbit{\eta'} \Leftrightarrow \DLocl{n}_1 \chi' = \DLocl{n}_1 \eta'$. The proof of Proposition~\ref{prop:Gnorbit} shows that this is {equivalent to} $G_n \cdot \bar{\chi} = G_n \cdot \bar{\eta}$. Lastly, the dimension equation follows from Proposition~\ref{prop:Gnorbit} and ~\cite{petukhov2022poisson}*{Theorem 4.3.1}.
\end{proof}

We have the similar version for multi-point local function.

\begin{corollary}\label{cor:orbitrelationmulti}
    Let $\g$ be $\W{-1}, W$ or $\Vir$ and $\chi, \eta$ be two one-point local functions on $\g$ of order $n$. Let $\bar{\chi}, \bar{\eta}$ be the corresponding elements of $\g^*_n$ defined in Notation~\ref{not:barchi}. Then  
    \begin{equation}\label{eq:orbitrelationmulti}
        \orbit{\chi} = \orbit{\eta} \Leftrightarrow G_\n \cdot \bar{\chi} = G_\n \cdot \bar{\eta}.
    \end{equation}
    Moreover, $\dim G_\n \cdot \bar{\chi} = \dim \orbit{\chi} - 2\ell$.
\end{corollary}
\begin{proof}
    By ~\cite{petukhov2022poisson}*{Theorem 4.3.1}, by rearranging the component one-point local functions $\orbit{\chi} = \orbit{\eta} \Leftrightarrow \orbit{\chi_i} = \orbit{\eta_i}$ for all $i$. By Corollary~\ref{cor:orbitrelation}, $\orbit{\chi_i} = \orbit{\eta_i} \Leftrightarrow G_{n_i} \cdot \bar{\chi}_i = G_{n_i} \cdot \bar{\eta}_i$ for all $i$. This is obviously equivalent to $G_\n \cdot \bar{\chi} = G_\n \cdot \bar{\eta}$. The dimension equation follows from $\dim \orbit{\chi} = \sum_i \dim \orbit{\chi_i}$, $ \dim G_\n \cdot \bar{\chi} = \sum_i \dim G_{n_i} \cdot \bar{\chi_i}$, and Corollary~\ref{cor:orbitrelation}.
\end{proof}

\subsection{The Dixmier map}\label{subsec:dixmiermap}

Finally, we put the pieces together and prove the existence and well-definedness of the strong Dixmier map.
\begin{proposition}\label{prop:anndeppseudo}
    Let $\chi$ and $\eta$ be two non-zero local functions on $W$. If they have the same pseudo-orbit, i.e., $\orbit{\chi} = \orbit{\eta}$, then $Q_\chi = Q_\eta$. 
\end{proposition}

\begin{proof}
    Let $\n, \widehat{\n}$ be the order of $\chi$ and $\eta$ respectively. By ~\cite{petukhov2022poisson}*{Theorem 4.3.1}, we can reorder their one-point component such that $\n = \widehat{\n}$. Recall the definition of $\bar{\chi}$ and $\bar{\eta}$ in Notation~\ref{not:barchi}. By Corollary~\ref{cor:orbitrelationmulti}, $G_\n \cdot \bar{\chi} = G_\n \cdot \bar{\eta}$. Since $\g_\n$ is a finite-dimensional solvable Lie algebra, by Theorem~\ref{theorem: Kirillov method}, we have $Q_{\bar{\chi}} = Q_{\bar{\eta}}$. Therefore, by Proposition~\ref{prop:Qchihatchimulti} 
    \begin{align*}
        Q_\chi = \Psi^{-1}_{\n}(\tilde{A}_\ell \otimes Q_{\bar{\chi}}) = \Psi^{-1}_{\n}(\tilde{A}_\ell \otimes Q_{\bar{\eta}}) = Q_\eta.
    \end{align*}
\end{proof}

\begin{theorem}\label{theo: Dixmier map}
    The following strong Dixmier map is well-defined. Let $\g$ be $\W{-1}$ or $W$, then
    \begin{equation}\label{eq: Dixmier Witt}
        \widetilde{\Dx}^{\g}: \Pprim S(\g) \to \Prim \Ua(\g): \quad 0 \mapsto 0, P(\chi) \neq 0 \mapsto Q_\chi.
    \end{equation}
    Lastly, we can lift the map to $Vir$ as follows
    \begin{align}\label{eq: Dixmier Vir}
        \widetilde{\Dx}^{\Vir}: \Pprim \Sa(\Vir) &\longrightarrow \Prim \Ua(\Vir): \nonumber \\
        P(\chi) = (z-\chi(z)) &\longmapsto (z-\chi(z)) \nonumber \\
        P(\chi) \neq (z-\chi(z)) &\longmapsto Q_\chi = \pi^{-1}(Q_{\Gamma(\chi)}),
    \end{align}
    where $\Gamma: \Ua(\Vir) \to \Ua(W)$ is the canonical projection $z \mapsto 0$ and $\Gamma(\chi)$ is the corresponding local function on $W$.
\end{theorem} 
\begin{proof}
    The statement for $W$ follows immediately from Proposition~\ref{prop:anndeppseudo}. For $\W{-1}$, let $\chi$ be a local function on $W$ and let $\iota$ be the inclusion $\W{-1} \hookrightarrow W$. By ~\cite{petukhov2022poisson}*{Corollary 4.3.11 and Proposition 4.3.6 a} there is a homeomorphism $\res{\Sa(\W{-1})}: \Pprim \Sa(W) \to \Pprim \Sa(\W{-1})$ defined by  
    \begin{align*}
        P(\chi) \cap \Sa(\W{-1}) = P(\chi|_{\W{-1}}). 
    \end{align*}
    Since $\res{\iota} M_\chi = M_{\chi|_{\W{-1}}}$, we have 
    \begin{align}
        Q_\chi \cap \Ua(\W{-1}) = Q_{\chi|_{\W{-1}}}.
    \end{align}
    Furthermore, for any local function $\chi$ on $\W{-1}$ which has 0 as a base point, we can consider a different local function $\tilde{\chi}$ on $\W{-1}$ where we replace that base point to a new base point not in the support of $\chi$. Since we have $\orbit{{\chi}} = \orbit{\tilde{\chi}} \subseteq W^*_{\geq -1}$, we have that $Q_{{\chi}} = Q_{\tilde{\chi}}$, thus the map $\widetilde{\Dx}^{\W{-1}}$ is well-defined. 
    
    For $\chi \in \Vir^*$, by ~\cite{iyudu2020enveloping}*{Theorem 1.4}, $(z-\chi(z))$ is a primitive ideal of $\Ua(\Vir)$. On the other hand by ~\cite{petukhov2022poisson}*{Theorem 3.3.1} we note that $P(\chi) \neq (z-\chi(z))$, then $\chi(z) = 0$ and $\chi$ is a local function on $\Vir$, thus $\Gamma(\chi)$ is a local function on $W$, and $Q_{\Gamma(\chi)} = \Ann_{\Ua(W)}M'_{\Gamma(\chi)}$ is a primitive ideal of $\Ua(W)$. Hence $\Gamma^{-1}(Q_{\Gamma(\chi)}) = \Ann_{\Ua(\Vir)}M'_{\chi}$ is a primitive ideal of $\Ua(\Vir)$.

    Moreover, by definition, for $\chi, \eta$ local function on $\Vir$ then $\chi(z) = \eta(z) = 0$. Thus for local function on $\Vir$, we have $\orbit{\chi} = \orbit{\eta} \subseteq \Vir^*$ is equivalent to $\orbit{\Gamma(\chi)} = \orbit{\Gamma(\eta)} \subseteq W^*$, thus $Q_{\Gamma(\chi)} = Q_{\Gamma(\eta)}$, hence $Q_\chi = Q_\eta$ as ideals of $\Ua(\Vir)$.
\end{proof}

\section{Inclusion of annihilators and further questions}\label{sec:inclusion}
Let $\g$ be either $\W{-1}, W$ or $\Vir$. In this section, we show that for one-point local functions $\chi, \eta$ on $\g$, then
\begin{align}
    P(\chi) \subsetneqq P(\eta) \Leftrightarrow Q(\chi) \subsetneqq Q(\eta).
\end{align}
We show that there are primitive ideals of $\Ua(\g)$ that are not the kernels of maps to the first Weyl algebra, answering in the negative a question of~\cite{conley2024}. We then state the analogous conjecture for multi-point local functions. Lastly, we motivate some further research question and their application with respect to the strong Dixmier map $\widetilde{\Dx}^{\Vir}$. 

We begin by characterizing orbit closures in $\g^*_n$.
\begin{proposition}\label{cor: Bruhat G obrit}
    For $n \geq 1$, let $\bar{\eta}, \bar{\chi}$ be in $\g^*_n$. Then
    \begin{equation}\label{eq: dim orbit equi}
        G_n \cdot \bar{\eta} \subsetneqq \overline{G_n \cdot \bar{\chi}} \Leftrightarrow \dim G_n \cdot \bar{\eta} < \dim G_n \cdot \bar{\chi}.
    \end{equation}
\end{proposition}
\begin{proof}
    The only if direction is obvious. We now show that if direction. Without loss of generality, assume $\bar{\chi}(v_{n-1}) \neq 0$.  If $n = 1$, we have $G_1 \cdot \bar{\chi} = \{ \bar{\chi}\}$, i.e., a closed point, there exists no $\bar{\eta}$ such that $G_n \cdot \bar{\eta} \subsetneqq \overline{G_n \cdot \bar{\chi}}$. Now for $n > 1$, as in Proposition~\ref{prop:Gnorbit}, pick the related presentation $\bar{\chi}_{\beta_1, \dots, \beta_n}$ for $\bar{\chi}$, with $\beta_n \neq 0$. If $n$ is even then the closure of $G_n \cdot \bar{\chi}$ equals $\g^*_n$ by Proposition~\ref{prop:Gnorbit}, completing the proof in this case.

    If $n= 2m+1$ and $m \geq 1$, by Proposition~\ref{prop:Gnorbit} the closure of $G_n \cdot \bar{\chi}$ is an irreducible subvariety of $\g^*_n$ of codimension 1, i.e., a hypersurface defined by some function $F_+$. We claim $F_+$ is semiinvariant with respect to $G_n$, i.e., for all $g \in G_n$, $g \cdot F_+ = F_+(g^{-1} x) = \tilde{\lambda}_g F_+ (x)$ for some $\tilde{\lambda}_g \in \kk$. In other words, $F_+ (G_n \cdot \chi) = 0$ or $0 \notin F_+ (G_n \cdot \chi)$.

    Let $R = \kk[\beta_1, ..., \beta_n, \beta_n^{-1}]$, where $\beta_1, ..., \beta_n$ are free variables. Denote by $(\cdot)_{\kk \to R}$ the base change from $\kk$ to $R$. By Proposition~\ref{prop:Gnorbit}, there exist a group element $g \in (G_n)_{\kk \to R}$ and $h \in R$ such that:
    \begin{align*}
        g \cdot \bar{\chi}_{0, ..., 0, h, 0, ..., \beta_n} = \bar{\chi}_{\beta_1, ..., \beta_n}
    \end{align*}
    where $h$ is at position $m = \frac{n-1}{2}$.
    By replacing each $\beta_i$ with $\frac{v_{i-1}}{i!}$, we may identify $R$ with $\kk[\g^*_n][v^{-1}_{2m}]$. Then $h = f/v^\ell_{2m}$ for some $f \in \kk[\g^*_n]$ and $\ell \in \NN$.

    Let $\bar{\zeta} \in \g_n^*$ with $v_{2m}(\bar{\zeta}) \neq 0$. Then
    \begin{align*}
        G_n \cdot \bar{\chi} = G_n \cdot \bar{\zeta} \Leftrightarrow v_{2m}(\bar{\chi})^{-1}\frac{f^2(\bar{\chi})}{v_{2m}(\bar{\chi})^{2\ell}} = v_{2m}(\bar{\zeta})^{-1}\frac{f^2(\bar{\zeta})}{v_{2m}(\bar{\zeta})^{2\ell}}.
    \end{align*}
    The rational function $F\coloneqq {f^2}/{v^{2\ell+1}_{2m}}$ is thus $G_n$-invariant and separates orbits.

    For $k \leq n$, let $\g^*_{\leq k} = \{ \chi_{\alpha_1, ..., \alpha_n} \in \g^*_n: \alpha_{k+1} = ... =\alpha_{n} = 0 \}$. The previous paragraph shows that $G_n \cdot \bar{\chi}$ is the hypersurface in $\g^*_n \backslash \g^*_{\leq n-1}$ defined by
    \begin{align*}
        F_+ \coloneqq v_{2m}(\bar{\chi})^{2 \ell+1}f^2 - f^2(\bar{\chi})v^{2 \ell+1}_{2m}.
    \end{align*}
    Note that $F_+$ is semiinvariant, as claimed.

    By Proposition~\ref{prop:Gnorbit}, all orbits of dimension less than $n-1$ belong to $\g^*_{\leq n-2}$, because otherwise if $\bar{\eta}(v_{n-2}) \neq 0$, then one can show that $\dim G_n \cdot \bar{\eta} = n-1 = \dim G_n \cdot \bar{\chi}$. We will check that the closure of $G_n \cdot \bar{\chi}$ contains $\g^*_{\leq n-2}$, i.e., check that $F_+|_{\g^*_{\leq n-2}} = 0$. We claim that
    \begin{equation}\label{eq:Fpolclaim}
        F_+|_{\g^*_{\leq n-1}} = v_{2m}(\bar{\chi})^{2 \ell+1}f^2 = c(v_{n-2})^k \text{ for some } k \in \NN, c \in \kk
    \end{equation}

    We now prove Equation~\eqref{eq:Fpolclaim}. The argment Proposition~\ref{prop:Gnorbit} shows that the induced action of the subgroup $H_n =\{g_{a + a^2s} | s \in \kk[a]/(a^{n-1})\} \leq G_n$ does not modify $v_{2m}$. Since $F = f^2/v^{2\ell+1}_{2m}$ is $G_n$ invariant, $f^2$ is $H_n$-invariant, hence so is $F_+$. At the same time, the argument of Proposition~\ref{prop:Gnorbit} implies that all $n-2$-dimensional $H_n$-invariant hypersurfaces of $\g^*_{\leq n-1}$ are defined by the equality $v_{n-2} = \beta_{n-1}$ for some $\beta_{n-1} \in \kk$. This implies that the restriction of $F_+$ to $\g^*_{\leq n-1}$ is a polynomial in $v_{n-2}$. The fact that $F_+$ is $G_n$-semiinvariant implies that $F_+|_{\g^*_{\leq n-1}} = c(v_{n-2})^k$ for some $k \in \NN$ and $c \in \kk$.

    Hence, we have established Claim~\ref{eq:Fpolclaim}, and thus $F_+|_{\g^*_{\leq n-2}} = 0$, i.e., $\overline{G_n \cdot \bar{\chi}}$ contains $\g^*_{\leq n-2}$.
\end{proof}

\begin{proposition}\label{prop:inclusionone}
    Let $\g$ be either $\W{-1}, W$ or $\Vir$ and $\chi, \eta$ be one-point local functions on $\g$. Then
    \begin{equation}
        P(\chi) \subseteq P(\eta) \Rightarrow Q_\chi \subseteq Q_\eta.
    \end{equation}
\end{proposition}

\begin{proof}
    Let the order of $\chi, \eta$ be $n, \hat{n}$ respectively. From~\cite{petukhov2022poisson}*{Corollary 4.2.9.}, $P(\chi) \subsetneqq P(\eta)$ implies that $\hat{n} < n$ and if $n=2m+1$ then $\hat{n} < 2m$. Recall the canonical projection $\tau_{\hat{n}}: \g_n \to \g_{\hat{n}}$ sending $v_{\geq \hat{n}}$ to $0$ and the notation $\hat{\eta} = \tau^*_{\hat{n}}(\bar{\eta})$ in Notation~\ref{not:barchi}. From Proposition~\ref{prop:Qchihatchione}, we have
    \begin{align}
        Q_\chi = (\Psi_{n})^{-1}(\tilde{A}_1 \otimes Q_{\bar{\chi}}); \quad \quad Q_\eta = (\Psi_{n})^{-1}(\tilde{A}_1 \otimes Q_{\widehat{\eta}})
    \end{align}
    It suffices to show that $Q_{\bar{\chi}} \subseteq Q_{\widehat{\eta}}$. From Corollary~\ref{cor:orbitrelation}, we note that $\dim G_n \cdot \widehat{\eta} < \dim G_n \cdot \bar{\chi}$, thus from Proposition~\ref{cor: Bruhat G obrit}, $G_n \cdot \widehat{\eta} \subsetneqq \overline{G_n \cdot \bar{\chi}}$. By Theorem~\ref{theorem: Kirillov method}, $\widetilde{\Dx}^{\g_n}$ is a homeomorphism, this implies $Q_{\bar{\chi}} \subsetneqq Q_{\widehat{\eta}}$.
\end{proof}

Note that the annihilator of the tensor density representation $\mathcal{F}_\lambda$ of $\W{-1}$ in~\cite{conley2024} equals $Q_{\chi_{x; \alpha, \lambda - \frac{1}{2}}}$. Proposition \ref{prop:inclusionone} gives us the negative answer to the following question in~\cite{conley2024}*{Section 11} that whether $Q_{\chi_{x; \alpha, \lambda - \frac{1}{2}}}$ and the augmentation ideal are the only primitive ideals of $\Ua(\W{-1})$. Note that $Q_{\chi_{x; \alpha, \lambda - \frac{1}{2}}}$ are all of $\mathrm{GK}$-codimension $2$.

\begin{proposition}\label{prop:GKdim3}
    There are primitive ideals of $\Ua(\W{-1})$ of $\mathrm{GK}$-codimension at least $3$, which are thus not kernels of maps to $A_1$.
\end{proposition}
\begin{proof}
    Let $\chi$ be a one-point local function on $\W{-1}$ of order $\geq 2$. By Proposition~\ref{prop:inclusionone}, $Q_\chi \subseteq Q_{\chi_{x; \alpha_0, \alpha_1}}$ for all $\alpha_0, \alpha_1 \in \kk$. By Corollary~\ref{cor:psiuniversal}, $\ker \Psi^{\W{-1}}_1 =  \bigcap_{\alpha_0, \alpha_1 \in \kk}Q_{\chi_{x; \alpha_0, \alpha_1}}$. Thus $Q_\chi \subseteq \ker \Psi^{\W{-1}}_1$. It is known that $\GK \im \Psi^{\W{-1}}_1 = 3$ (see e.g.~\cite{lucasthesis}*{Lemma 2.5.5.} for a proof). 
\end{proof}

To end, we state two conjectures which generalize Proposition~\ref{prop:inclusionone}. The converse of Proposition~\ref{prop:inclusionone} is false, however, we conjecture a ``partial converse''.

\begin{conjecture}
    Let $\g$ be either $\W{-1}, W$ or $\Vir$ and $\chi, \eta$ be one-point local functions on $\g$. Then
    \begin{equation}
        P(\chi) \subsetneqq P(\eta) \Leftrightarrow Q_\chi \subsetneqq Q_\eta.
    \end{equation} 
\end{conjecture}

For multi-point local functions, we conjecture the implication in Proposition~\ref{prop:inclusionone} still holds. 
\begin{conjecture}\label{conj:inclusionPQ}
    Let $\g$ be either $\W{-1}, W$ or $\Vir$ and $\chi, \eta$ be local functions on $\g$. Then
    \begin{equation}
        P(\chi) \subsetneqq P(\eta) \Rightarrow Q_\chi \subsetneqq Q_\eta.
    \end{equation} 
\end{conjecture}
We believe that the converse direction of Conjecture~\ref{conj:inclusionPQ} is false. Finally, we state a much more ambitious conjecture about surjectivity of the strong Dixmier map.
\begin{conjecture}\label{conj:surjective}
    Let $\g$ be $\W{-1}, W$ or $\Vir$. The strong Dixmier map $\widetilde{\Dx}^\g: \Pprim \Sa(\g) \to \Prim \Ua(\g)$ is surjective. 
\end{conjecture}

If this conjecture is true, we immediately obtain the three results.
\begin{conjecture}\label{conj:completelyprime}
    Let $\g$ be $\W{-1}, W$ or $\Vir$. Then every primitive ideal of $\Ua(\g)$ is completely prime.
\end{conjecture}

\begin{conjecture}\label{conj:ACCprimitive}
    Let $\g$ be $\W{-1}, W$ or $\Vir$. Then $\Ua(\g)$ satisfies ACC on primitive ideals. 
\end{conjecture}

\begin{conjecture}\label{conj:kerWeyl}
    Let $Q$ be a primitive ideal of $\Ua(\W{-1})$. Then there exists $k$ and $\varphi: \Ua(\W{-1}) \to A_k$ such that $Q = \ker \varphi$. 

    Let $\g$ be $W$ or $\Vir$ and let $Q$ be a primitive ideal of $\Ua(\g)$. Then there exists $k$ and $\varphi: \Ua(\g) \to \tilde{A}_k$ such that $Q = \ker \varphi$. 
\end{conjecture}
Conjectures \ref{conj:completelyprime}, \ref{conj:ACCprimitive}, \ref{conj:kerWeyl} are immediate corollaries of Conjecture~\ref{conj:surjective}.
\section*{Index of notation}\label{index}
\begin{multicols}{2}
{\small  \baselineskip 14pt

    $W$, $\Vir$, $\W{n}$: Lie algebras \hfill \pageref{pl:W},\pageref{pl:Vir}

    $\Gamma: \Vir \to W$: projection sending $z$ to $0$ \hfill \pageref{pl:pVirW}

    $e_i = t^{i+1}\del$: basis of $W$ \hfill \pageref{pl:ei}

    $W(f)$, $\Vir(f)$, $\W{-1}(f)$: a (polynomial) subalgebra \hfill\pageref{pl:Wf}

    $\Pprim A$: Poisson primitive spectrum of $A$ \hfill \pageref{pl:Pprim}

    $P(\chi)$: Poisson core of $\chi \in \g^*$ \hfill \pageref{ind:Pchi}

    $\orbit{\chi}$: Pseudo-orbit of $\chi \in \g^*$ \hfill\pageref{def:psorbit}

    $\chi = \chi_1 + \dots \chi_\ell$: a local function \hfill\pageref{not:localfunc}

    $\n = (n_1, \dots, n_\ell)$: order of a local function \hfill\pageref{not:localfunc} 

    $\supp(\chi)$: support of a local function \hfill\pageref{not:localfunc} 

    $m_i = \left\lfloor \frac{n_i}{2}\right\rfloor $ \hfill\pageref{not:localfunc}
    
    $B_\chi$: antisymmetric bilinear form  \hfill \pageref{pl:Bchi}
    
    $\p$: an (arbitrary) polarization \hfill \pageref{pl:polarization}

    $\rho_{\mf{h}, \g}= \frac{1}{2} \tr \ad_{\g/\mf{h}}(x)$: twist \hfill \pageref{def:onedimtwist}

    $\kk_\chi = \kk \cdot 1_\chi$: one-dimensional representation \hfill \pageref{def:orbitrep}

    $\kk'_{\chi, \p, \g} \cong \kk_\chi \otimes \rho_{\p, \g} $: twisted one-dimensional representation \hfill \pageref{def:orbitrep}

    $M'_{\chi, \p}$, $Q_{\chi,\p}$: (induced) twisted rep, annihilator \hfill \pageref{def:orbitrep}

    $\Prim A$: primitive spectrum \hfill \pageref{pl:Prim}

    $\Dx^\g$, $\overline{\Dx}^{\g}$: (quotient) Dixmier map \hfill \pageref{theorem: Kirillov method}

    $\widetilde{\Dx}^\g$: (strong) Dixmier map \hfill \pageref{theo:strongDx}
    
    $\p_\chi$: canonical polarization \hfill \pageref{prop:canpol}

    $\rad(f), \supp(f)$: radical and support of $f$\hfill \pageref{pl:radf}
    
    $M'_\chi$, $Q_\chi$: canonical local rep, annihilator \hfill \pageref{def:canrep}

    $\LT_\mathcal{C}$: leading term \hfill \pageref{not:leadterm} 

    $u_{j, x}$, $u_{j}$, $\tilde{u}_{q, j}$, $\tilde{u}_{q, j, x}$: elements of $\W{-1}$  \hfill \pageref{eq:tildeudef}

    $a_{q, i}$, $a_{q, i, x}$: constants \hfill \pageref{eq: a k d definition}
    
    $\xi_{q, i, x}$, $\xi_{q, i}$: elements of $\Ua(\W{-1})$ \hfill \pageref{eq:xidef} 

    $\sigma_i, \varsigma_i$: words in $\Ua(\g)$  \hfill \pageref{not:(var)sigma}

    $\phi_k(r, s)$ \hfill  \pageref{not:phi}

    $\chi'$: twisted local function ($M_{\chi'} \cong M'_\chi$) \hfill  \pageref{not:chi'} 

    $\Dx^{\W{-1}}, \Dx^{W}, \Dx^{\Vir}$: Dixmier maps \hfill  \pageref{pl:Dx}

    $\widetilde{\Dx}^{\W{-1}}, \widetilde{\Dx}^{W}, \widetilde{\Dx}^{\Vir}$: strong Dixmier maps  \hfill \pageref{eq:dxVir}

    $A_m$, $\widetilde{A}_m$: (localized) Weyl algebra \hfill  \pageref{pl:Weyl}

    $\g_n = \W{0}/\W{n}$ with basis  $v_i = e_i + \W{n}$ \hfill  \pageref{pl:gn}

    $T_\infty = A_1 \otimes \Ua(\g_n)$, $T_n = A_1 \otimes \Ua(\g_n)$, $\widetilde{T}_n = \widetilde{A}_1 \otimes \Ua(\g_n)$ \hfill  \pageref{pl:Tn}

    $\Psi^{\W{-1}}_{\infty}$, $\Psi^{\W{-1}}_{n}$, $\Psi^{W}_{n}$  \hfill \pageref{eq:psiW-1n}, \pageref{eq:psiW-1infty} 

    $\Phi^{\W{-1}}_{\infty}$, $\Phi^{\W{-1}}_{n}$, $\Phi^{W}_{n}$   \hfill  \pageref{eq:phi-1inf}, \pageref{eq:phi-1inf2}

    $\mathbb{S}_\infty$, $\mathbb{S}_n$ \hfill \pageref{pl:S}

    $\g_\n, T_\n, \widetilde{T}_\n$ \hfill \pageref{def:gboldn}

    $\Psi^{\W{-1}}_{\n}, \Psi^{W}_{\n}$ \hfill \pageref{prop: psi multi}
    
    $\n \leq \widehat{\n}$ \hfill \pageref{eq:biggertuple}

    $I_{\chi'}, L_{\chi'} = \tilde{T}_{\widehat{n}}/I_{\chi'}$ \hfill \pageref{def:Ichi} 

    $\bar{\chi}, Q_{\bar{\chi}} $ \hfill \pageref{not:barchi}

    $\widehat{\chi}:=\tau_{\widehat{\n}}^*(\bar{\chi}), Q_{\widehat{\chi}}$ \hfill \pageref{not:barchi} 

    $\pi^{\widehat{\n}}: \Loc^{\leq \widehat{\n}} \to \g^*_{\widehat{\n}}$  \hfill \pageref{eq:mapchibar}

    $\p_{\widehat{\chi}}, M'_{\widehat{\chi}}, J_{\widehat{\chi}}$  \hfill  \pageref{prop:polbarchi}

    $\varphi_m: \g_{2m} \to A_m$ \hfill  \pageref{eq:weylmap}

    $\overline{\Psi}^W_{2m}, \overline{\Psi}^{\W{-1}}_{2m}$ \hfill  \pageref{eq: map even}
   
    $\tilde{A}_{m+1} I_\chi$ and $N_\chi = \tilde{A}_{m+1}/ \tilde{A}_{m+1} I_\chi$ \hfill  \pageref{eq: restricted module even}

    $g_s$ \hfill \pageref{eq: G_n action del}
    
    $G_n$, $G_\n$: adjoint groups \hfill \pageref{def:Gn} 

    $\mathrm{lie}^{\leq n} = \ad(\g_n)$ \hfill \pageref{pl:adgrp}
} 
\end{multicols}

\bibliographystyle{amsalpha}
\bibliography{refs.bib}
\end{document}